\documentclass[11pt,a4paper,reqno]{amsart} 
\usepackage[titletoc]{appendix}
\usepackage[margin=1in]{geometry}
\usepackage[utf8]{inputenc}
\usepackage{amsfonts,amssymb,amscd,graphicx, amsmath,hyperref,subcaption,lscape}
\usepackage{fancybox}
\usepackage{color,float} 
\usepackage{cite} 
\usepackage{setspace}
\DeclareMathOperator{\sech}{sech}

\usepackage[shortlabels]{enumitem}
\usepackage{xcolor}
\usepackage{tikz}
\mathchardef\ordinarycolon\mathcode`\:
\mathcode`\:=\string"8000
\begingroup \catcode`\:=\active
\gdef:{\mathrel{\mathop\ordinarycolon}}
\endgroup
\makeatletter
\renewcommand*\env@matrix[1][\arraystretch]{%
	\edef\arraystretch{#1}%
	\hskip -\arraycolsep
	\let\@ifnextchar\new@ifnextchar
	\array{*\c@MaxMatrixCols c}}
\makeatother
\usepackage{graphicx}
\usepackage{amsthm}
\usepackage{rotating}
\newtheorem{theorem}{Theorem}[section]

\theoremstyle{definition}

\newtheorem{remark}[theorem]{Remark}
\newtheorem{example}[theorem]{Example}

\newcommand{\s}{\mathcal{S}}
\title[Dwell-flee relations]{Stability of bimodal planar switched linear systems with both stable and unstable subsystems} 
\author[Swapnil Tripathi]{Swapnil Tripathi}
\address{Department of Mathematics\\
	Indian Institute of Science Education and Research Bhopal\\
	Bhopal Bypass Road, Bhauri \\
	Bhopal 462 066, Madhya Pradesh\\
	India}
\email{swapnil93@iiserb.ac.in}
\author[Nikita Agarwal]{Nikita Agarwal}
\address{Department of Mathematics\\
	Indian Institute of Science Education and Research Bhopal\\
	Bhopal Bypass Road, Bhauri \\
	Bhopal 462 066, Madhya Pradesh\\
	India}
\email{nagarwal@iiserb.ac.in}

\date{\today}

\begin{document} 
	
	\begin{abstract}
	We study dynamics of a bimodal planar linear switched system with a Hurwitz stable and an unstable subsystem. For given \textit{flee time} from the unstable subsystem, the goal is to find corresponding \textit{dwell time} in the Hurwitz stable subsystem so that the switched system is asymptotically stable. The dwell-flee relations obtained are in terms of certain smooth functions of the eigenvalues and (generalized) eigenvectors of the subsystem matrices. The results are extended to a special class of symmetric bilinear systems. The results are also extended to a multimodal planar linear switched system in which the switching is governed by an undirected star graph, where the internal node corresponds to a Hurwitz stable (unstable, respectively) subsystem and all the leaves correspond to unstable (Hurwitz stable, respectively) subsystems. For this multimodal system, dwell-flee relations are obtained as solutions of a minimax optimization problem.
\end{abstract}

\maketitle

\noindent \textbf{Keywords}: Piecewise continuous dynamical systems, control theory, dwell time, flee time, stability, asymptotic stability.\\
\noindent \textbf{2010 Mathematics Subject Classification}: 37N35 (Primary); 93C05, 93D20 (Secondary)
\section{Introduction}

A \textit{switched system} comprises a family of subsystems and a switching signal. A \textit{switching signal} is a right continuous, piecewise constant function which determines the subsystem active at each time instant. Switched systems serve as a model for various control engineering applications of environmental, electrical, mechanical and chemical nature, see~\cite{zhang2010switched,xu2011fuel,belykh2014evolving}. Due to these applications in practical systems, stability and control analysis in switched systems has been an active area of research in the past several decades. The common Lyapunov function and multiple Lyapunov function techniques are used to study stability of time dependent switched systems, see~\cite{agrachev2001lie,lin2009stability}.
If all subsystems are stable, there may exist a signal which destabilizes the system, see~\cite{liberzon2003switching}. Thus it becomes relevant to find the class of signals for which the switched system is stable. Several constraints for signals such as \textit{dwell time}, \textit{average dwell time}, and \textit{mode-dependent average dwell time} are discussed in the literature. It is known that when all subsystems are stable, the switched system is stable if the dwell time is sufficiently large, that is, the switched system spends sufficiently large amount of time in each subsystem, see~\cite{morse1996supervisory,hespanha1999stability,karabacak2013dwell,geromel2006stability}.

In practical systems, it is unavoidable to have unstable behaviour, which may arise due to several reasons such as wear down of machines, failures due to environmental factors, and otherwise. Hence the issue of stability of switched systems with unstable subsystems has attracted significant attention in recent years, see~\cite{agarwal2018simple,zhai2001stability,yang2009stabilization,agarwal2019stabilizing,zhang2014stability}. Finding a switching signal which \textit{stabilizes} the switched system when either all or some subsystems are stable, with or without the presence of control, is extensively studied in the literature. Systems having all unstable subsystems have been studied in~\cite{agarwal2019stabilizing,xiang2014stabilization}. A survey of results on \textit{stabilizability} and \textit{controllability} is presented in~\cite{decarlo2000perspectives, sun2006switched}. Moreover in the past couple of decades, synchronization of coupled dynamical systems has attracted considerable attention, see~\cite{yang2020synchronization,zhou2020finite,yang2019synchronization,yang2020synchronizationcao,yang2018synchronization,tang2021finite,zou2021finite,zhang2020velocity} for results about systems with actuator fault, time delay and stochastic coupling. 

It is but intuitive that stability of a switched system with both stable and unstable subsystems can be achieved if the system spends sufficiently large time (dwell time) in stable subsystems and sufficiently small time (flee time) in unstable subsystems, see~\cite{agarwal2018simple, zhai2001stability}. This serves as a motivation behind this paper and allows us to ask the following question: can the dwell time be obtained in terms of the flee time and the properties of the subsystems?

In this paper we consider a \textit{bimodal planar switched linear system} -- a switched system on $\mathbb{R}^2$ with two linear subsystems, out of which exactly one is Hurwitz stable. For given flee time, our goal is to obtain dwell time which guarantees stability of such a system. The dwell-flee relations obtained in our work explicitly highlight the role of eigenvalues and eigenvectors of the subsystem matrices. Additionally most of the results in the literature give stability conditions in terms of the leading eigenvalues of the subsystem matrices, ignoring the other eigenvalues. However our dwell-flee relations are in terms of all the eigenvalues of the subsystem matrices. Furthermore, the dwell-flee relation obtained in~\cite{agarwal2018simple} are linear and use Jordan basis matrices of subsystem matrices having unit norm columns. Using a more efficient grouping of terms in the flow of the system and by varying over all possible Jordan basis matrices, we obtain lower dwell time. In addition, since we obtain dwell time as a function of flee time, if a certain amount of time is spent in the unstable subsystem at a particular switching instant, the unstable effect can be compensated by spending sufficiently large amount of time in the stable subsystem in the subsequent switching instance into the stable subsystem. Hence while defining a signal, there is no need of apriori fixing a flee time and a corresponding dwell time (using the dwell-flee relation) for all time, see signals defined in~\eqref{eq:signaldynam}. Hence, using our results, stability of the switched system can be achieved, even with any number of system failures during the time course of the signal.

The class of bimodal planar switched linear systems exhibit rich dynamical behaviour, see~\cite{liberzon2003switching,iwatani2006stability}. Though practical models are usually not linear, studying this class of systems is useful for gaining insight into its non-linear counterpart, which occurs in practice. For instance, in~\cite{palejiya2013stability}, the authors model an integrated wind turbine and battery system as a bimodal planar system and discuss stability issues using linearization of the subsystems. The issues of stability, stabilizability and controllability (with or without the presence of control) have also been studied specifically for bimodal setting. A necessary and sufficient condition for stability under arbitrary switching is given in~\cite{balde2009note} for bimodal planar linear systems, and in~\cite{eldem2009stability} for bimodal linear systems in $\mathbb{R}^3$. The stabilizability and controllability issues for bimodal planar linear systems have been addressed in~\cite{sename2013robust}.

\section{Problem setting and notations}
Let $A_1$ and $A_2$ be two planar ($2\times 2$) real matrices. In this paper, we will focus on switched systems on $\mathbb{R}^2$ of the form 
\begin{eqnarray}\label{eq:system}
	\dot{x}(t)&=&A_{\sigma(t)}x(t), 
\end{eqnarray}	
where $x(t)\in\mathbb{R}^2$ and $\sigma(t)\in \{1,2\}$, for $t\ge 0$. The function $\sigma$ is a right-continuous piecewise constant function, known as the \textit{switching signal}. The switched system~\eqref{eq:system} has two \textit{subsystems} $A_1$ and $A_2$. That is, at time instant $t$, if $\sigma(t)=1$, the flow of~\eqref{eq:system} is governed by the subsystem $\dot{x}(t)=A_{1}x(t)$, whereas if $\sigma(t)=2$, the flow is governed by the subsystem $\dot{x}(t)=A_{2}x(t)$. Thus the switching signal $\sigma$ determines the subsystem active at any time instant $t$. 

\subsection*{Assumptions and Notations}
\begin{itemize}
	\item For each $k\ge 1$, let $d_k>0$ denotes the $k^{th}$ discontinuity of $\sigma$, known as the $k^{th}$ \textit{switching instant}. Set $d_0=0$. Let $\Delta_k=d_k-d_{k-1}$. The switching signal $\sigma$ has infinitely many discontinuities. Further the system does not exhibit zeno behavior. That is, the sequence $(d_k)$ has no limit point. Hence $d_k\uparrow\infty$ as $k\rightarrow\infty$. 
	\item The switched system is governed by the subsystem $A_1$ at time instant $t=0$. This can be assumed without loss of generality since otherwise if $\sigma(0)=2$, then we can study the flow of the switched system for $t\ge d_1$. Since there are only two subsystems, we have $\sigma(d_{2k-2})=1$ and $\sigma(d_{2k-1})=2$, for $k\ge 1$. For $k\ge 1$, let $t_{k}=\Delta_{2k-1}$ be the time spent in the subsystem $A_1$ between the $(2k-2)^{th}$ and $(2k-1)^{th}$ switching instants and and $s_k=\Delta_{2k}$ be the time spent in the subsystem $A_2$ between the $(2k-1)^{th}$ and $(2k)^{th}$ switching instants. The signal $\sigma$ is shown in Figure~\ref{fig:signal}.   
	\item The state $x$ does not jump at any discontinuity point of $\sigma$. That is, $\lim_{t\rightarrow d_k^-}x(t)=x(d_k)$, for each $k\ge 1$.
	\item Both $A_1$ and $A_2$ are nonzero matrices. Since otherwise if $A_1$ is the zero matrix, then for each $k\ge 1$ and $t\in [d_{2k-1},d_{2k})$, $x(t)=x(d_{2k-1})$. That is, the state $x(t)$ does not change in the interval $[d_{2k-1},d_{2k})$. Hence the stability of the switched system is same as that of the subsystem $A_2$. The situation is similar when $A_2$ is the zero matrix. 
	\item The matrix $A_1$ is Hurwitz stable, that is, all its eigenvalues lie to the left of the imaginary axis. \textit{We will call $A_1$, a stable matrix throughout this paper}. The matrix $A_2$ is unstable, that is, one of its eigenvalues lies to the right of the imaginary axis. 
\end{itemize} 

Throughout the paper, for a given vector $x$, $\|x\|$ denotes the $\ell^2$-norm of $x$, and given matrix $A$, $\|A\|$ denotes the spectral norm of $A$. 

\begin{figure}[h!]
	\centering
	\includegraphics[scale=0.3]{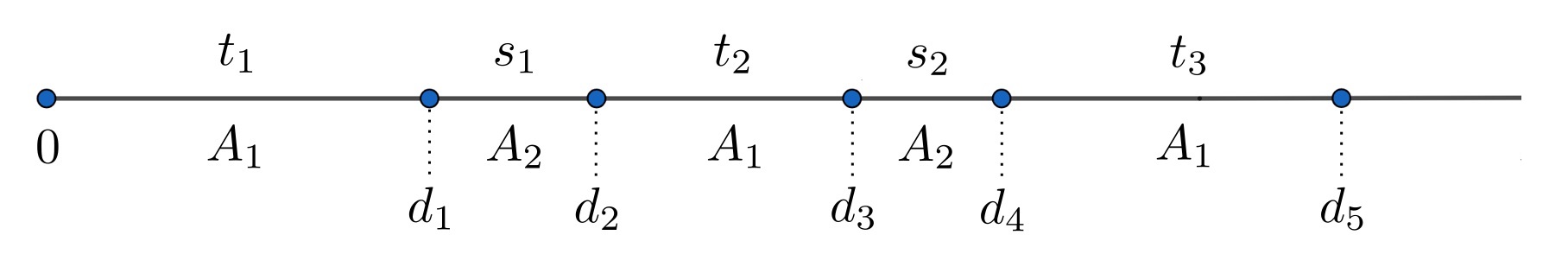}
	\caption{Switching signal $\sigma$ with $\sigma(0)=1$.}
	\label{fig:signal}
\end{figure} 

\noindent The switched system~\eqref{eq:system} is said to be 
\begin{enumerate}
	\item \textit{stable} if for every $\epsilon>0$, there exists $\delta>0$ such that $\| x(0)\|<\delta$ implies $\|x(t)\|<\epsilon$ for all $t>0$,
	\item \textit{asymptotically stable} if it is stable and $\Vert x(t)\Vert\to 0$ as $t\to\infty$ for every initial vector $x(0)\in\mathbb{R}^n$, and
	\item \textit{exponentially stable} if there exist $\alpha,\beta>0$ such that $\|x(t)\|\le \alpha e^{-\beta t}\|x(0)\|$ for all $t>0$, for every initial vector $x(0)\in\mathbb{R}^n$.
\end{enumerate} 
The above definitions are valid for more general switched system and not just for systems of form~\eqref{eq:system}. Note that the asymptotic stability and exponential stability are equivalent notions for switched linear systems,~\cite{hespanha2004uniform}, and are strictly stronger notions than stability, by definition.

Now, we formally define the constraints of dwell time and flee time discussed earlier. When both subsystem matrices $A_1$ and $A_2$ are stable, the signal $\sigma$ is said to have \textit{dwell time} $\tau> 0$ if $\Delta_k\ge \tau$, for all $k\ge 1$. In this setting, the switched system~\eqref{eq:system} is stable for all signals with sufficiently large dwell time~\cite{hespanha1999stability}. In our setting, when one of the subsystems is stable and other is unstable, it is intuitive that the switched system is stable if it spends sufficiently large time in the stable subsystem and spends small amount of time in the unstable subsystem, see~\cite{agarwal2018simple,agarwal2019stabilizing}. In this setting, the signal $\sigma$ is said to have \textit{dwell time} $\tau>0$ if $t_k=\Delta_{2k-1}\ge \tau$ and is said to have \textit{flee time} $\eta>0$ if $s_k=\Delta_{2k}\le\eta$, for all $k\ge 1$. In~\cite{agarwal2018simple}, for stability of a switched system with both stable and unstable subsystems, a relationship between dwell time $\tau$ and flee time $\eta$ is derived. 

\subsection*{Classes of switching signals}
The aim of this paper is to find dwell-flee relations (relationship between dwell time and flee time) guaranteeing the stability of the bimodal planar linear switched system of the form~\eqref{eq:system} with a stable and an unstable subsystem. We define the following classes of signals:
\begin{eqnarray}\label{eq:signalbm}
	\s\left[\tau,\eta\right]&=&\left\{\sigma\colon[0,\infty)\to \{1,2\}\ \vert\ \sigma(0)=1,\  t_k\ge \tau,\ s_k\le \eta, \ \text{for all } k\in\mathbb{N}\right\},\nonumber\\ 
	\s'\left[\tau,\eta\right]&=&\{\, 
	\sigma\in \s\left[\tau,\eta\right]\ \vert \ \text{the sequence } (t_k,s_k)_{k\ge 1} \text{ of tuples does not converge} \nonumber\\
	&& \text{    \ to the tuple } (\tau,\eta)\}.  
\end{eqnarray}

The collection $\mathcal{S}[\tau,\eta]$ consists of signals having dwell time $\tau$ and flee time $\eta$. Recall that $t_k$ is the time spent in the stable subsystem $A_1$ after the $(2k-2)^{th}$ switching instant before the signal jumps to unstable subsystem $A_2$ at the $(2k-1)^{th}$ switching instant. Thereafter the system spends time $s_k$ before the signal switches back to the stable subsystem $A_1$ at the $(2k)^{th}$ switching instant. The subcollection $\mathcal{S}'[\tau,\eta]$ of $\mathcal{S}[\tau,\eta]$ consists of those signals $\sigma$ for which the sequence $(t_k,s_k)_{k\ge 1}$ does not converge to $(\tau,\eta)$.


Recall that, we will study stability of the switched system~\eqref{eq:system} where \textit{$A_1$ will be a Hurwitz matrix (all the eigenvalues with negative real part) and $A_2$ will be an unstable matrix (having at least one of the eigenvalues with positive real part)}. For $i=1,2$, let $J_i$ be the \emph{real} Jordan form of $A_i$, that is, there exists an invertible matrix $V_i$ (known as \textit{Jordan basis matrix}) such that $A_i=V_i J_i V_i^{-1}$. Let us fix $V_1$ and $V_2$ for the time being. The flow of the switched system~\eqref{eq:system} is given by
\[
x(t) = 
\left\{\begin{aligned}
	& e^{A_1 (t-d_{2j})}\prod_{i=1}^j \left(e^{A_2 s_i}e^{A_1 t_i}\right) x(0), & t & \in[d_{2j},d_{2j+1})\\
	& e^{A_2(t-d_{2j+1})}e^{A_1 t_{j+1}}\prod_{i=1}^j \left(e^{A_2 s_i}e^{A_1 t_i}\right) x(0), &  t & \in[d_{2j+1},d_{2j+2}).
\end{aligned}\right.
\]

\noindent Let $\sigma\in\mathcal{S}[\tau,\eta]$. For $t\in[d_{2j},d_{2j+2})$,
\begin{equation}\label{eq:flow}
	\|x(t)\|\le \zeta\ \xi\ \min\left\lbrace\prod_{i=1}^j \left\|V_1^{-1}V_2e^{J_2 s_i}\,V_2^{-1}V_1 e^{J_1 t_i} \right\|, \prod_{i=1}^j \left\|V_2^{-1}V_1e^{J_1 t_{i+1}}\,V_1^{-1}V_2 e^{J_2 s_i} \right\| \right\rbrace \|x(0)\|,
\end{equation}
where $\zeta=\max\{\|V_1\|\|V_1^{-1}V_2\|\|V_2^{-1}V_1\|\|V_1^{-1}\|,\|V_1\|\|V_1^{-1}V_2\|\|V_2^{-1}\|\}$, \\
$\xi=\left(\sup_{t\in [0,\infty)}\left\|{\rm e}^{J_1 t}\right\|\right)^2\,\sup_{t\in [0,\eta]}\left\|{\rm e}^{J_2 t}\right\|$. Note that $\xi$ is finite since $J_1$ is a stable matrix. 

For given flee time $\eta>0$, we will compute $\tau_{1,2}(V_1,V_2,\eta)>0$ and $\tau_{2,1}(V_1,V_2,\eta)>0$ such that for $s<\eta$,
\begin{eqnarray}
	\left\|V_1^{-1}V_2e^{J_2 s}\,V_2^{-1}V_1 e^{J_1 t} \right\| &<&1, \text{ for all } t>\tau_{1,2}(V_1,V_2,\eta), \label{eq:tau12} \\
	\left\|V_2^{-1}V_1e^{J_1 t}\,V_1^{-1}V_2 e^{J_2 s} \right\| &<& 1, \text{ for all } t>\tau_{2,1}(V_1,V_2,\eta).\label{eq:tau21}
\end{eqnarray}

Such a value $\tau_{1,2}(V_1,V_2,\eta)$ ($\tau_{2,1}(V_1,V_2,\eta)$, respectively) exists since the norm terms in~\eqref{eq:tau12} and~\eqref{eq:tau21} are continuous functions of $s\in[0,\eta]$, and hence the first (second, respectively) inequality is satisfied for sufficiently large $t>0$. Moreover, this implies stability of the switched system for $\sigma\in\mathcal{S}[\tau_{1,2}(V_1,V_2,\eta),\eta]\bigcup \mathcal{S}[\tau_{2,1}(V_1,V_2,\eta),\eta]$, due to~\eqref{eq:flow}.

\subsection*{Procedure for computing $\tau_{1,2}(\eta)$ and $\tau_{2,1}(\eta)$}
For any $n\times n$ real matrix $K$, $\|K\|<1$ if and only if $\rho(K^\top K)<1$, where $\rho$ denotes the spectral radius of a matrix. It was proved in~\cite{fleming1998schur} that for a planar matrix $A$, $\rho(A)<1$ if and only if $\lvert\text{tr}(A)\rvert<1+\text{det}(A)$ and $\lvert\text{det}(A)\rvert<1$. Thus, for a $2\times 2$ matrix $K$, $\|K\|<1$ if and only if 
\begin{enumerate}[(C1)]
	\item $\lvert\text{det}(K^\top K)\rvert<1$, and
	\item $\lvert\text{tr}(K^\top K)\rvert<1+\text{det}(K^\top K)$.
\end{enumerate}	

Recall inequalities~\eqref{eq:tau12} and~\eqref{eq:tau21}. Let us consider real Jordan basis matrices $V_1$ of $A_1$ with Jordan form $J_1$, and $V_2$ of $A_2$ with Jordan form $J_2$. To compute $\tau_{1,2}(V_1,V_2,\eta)$, we will study the two inequalities obtained by substituting $K=V_1^{-1}V_2e^{J_2 s}\,V_2^{-1}V_1 e^{J_1 t}$ in (C1) and (C2). The sub-region of the open first quadrant $Q_1$ of the $ts$-plane, satisfying (C1) will be referred to as the \textit{feasible region}. Then $\tau_{1,2}(V_1,V_2,\eta)$ will be such that the (unbounded) rectangle $\mathcal{R}_{\tau_{1,2}(V_1,V_2,\eta),\eta}=\{(t,s)\in Q_1\ \colon\ t\ge \tau_{1,2}(V_1,V_2,\eta),\ s\le \eta\}$ lies within the closure of the region in $Q_1$, satisfying both (C1) and (C2). This process will be repeated with $K=M_{D_1,D_2}e^{J_1 t}\,M_{D_1,D_2}^{-1} e^{J_2 s}$ to obtain $\tau_{2,1}(V_1,V_2,\eta)$. We will then vary over Jordan basis matrices $V_1$ and $V_2$ to obtain the best dwell time bound, for a given flee time. 

We will fix real Jordan basis matrices $P_1$ of $A_1$ and $P_2$ of $A_2$, and let $M=P_2^{-1}P_1$, known as the \textit{transition matrix}. The fact that any other Jordan basis matrix $V_1$ ($V_2$) of $A_1$ ($A_2$) can be expressed in terms of $P_1$ ($P_2$) will be used to simplify the analysis. Moreover, since the inequalities~\eqref{eq:tau12} and~\eqref{eq:tau21} do not alter if $V_2^{-1}V_1$ is multiplied by a nonzero scalar, we will choose $P_1$ and $P_2$ so that the matrix $M=P_2^{-1}P_1$ has determinant either $1$ or $-1$ ($1$, if possible). We will take $M=\begin{pmatrix}
	a&b\\ c&d
\end{pmatrix}$ throughout the paper. 

If the infimum is attained, we let $\tau_{1,2}(\eta)=\inf_{V_1,V_2}\tau_{1,2}(V_1,V_2,\eta)$. Otherwise, for any $\epsilon>0$, let $\tau_{1,2}(\eta)=\inf_{V_1,V_2}\tau_{1,2}(V_1,V_2,\eta)+\epsilon$. Similarly, if the infimum is attained, let $\tau_{2,1}(\eta)=\inf_{V_1,V_2}\tau_{2,1}(V_1,V_2,\eta)$. Otherwise, for any $\epsilon>0$, let $\tau_{2,1}(\eta)=\inf_{V_1,V_2}\tau_{2,1}(V_1,V_2,\eta)+\epsilon$. 

Then the switched system~\eqref{eq:system} is stable for all $\sigma\in \s[\tau_{1,2}(\eta),\eta]$ and is asymptotically stable for all $\sigma\in \s'[\tau_{1,2}(\eta),\eta]$, since for such signals $\|x(t)\|$ is bounded above by a scalar multiple of $\rho^r$, for some $\rho<1$ (since $(t_k,s_k)_{k\ge 1}$ does not converge to $(\tau_{1,2}(\eta),\eta)$). Similarly the switched system~\eqref{eq:system} is stable (asymptotically stable) for all $\sigma\in\s[\tau_{2,1}(\eta),\eta]$ ($\sigma\in\s'[\tau_{2,1}(\eta),\eta]$).

In addition to the classes of signals defined in~\eqref{eq:signalbm}, we define two more classes of signals:
\begin{eqnarray}\label{eq:signaldynam}
	\s_{1,2}&=&\left\{\sigma\colon[0,\infty)\to \{1,2\}\ \vert\ \sigma(0)=1,\  t_{k}\ge \tau_{1,2}(s_k), \ \text{for all } k\in\mathbb{N}\right\},\nonumber\\
	\s_{2,1}&=&\left\{\sigma\colon[0,\infty)\to \{1,2\}\ \vert\ \sigma(0)=1,\  t_{k+1}\ge \tau_{2,1}(s_k), \ \text{for all } k\in\mathbb{N}\right\}.
\end{eqnarray}
In these classes, in contrast to those defined in~\eqref{eq:signalbm}, the flee time is not fixed apriori while defining the signal. These classes are discussed in Remarks~\ref{rem:dynRC21},~\ref{rem:dynCR12},~\ref{rem:dynNN},~\ref{rem:dynNC},~\ref{rem:dynCN},~\ref{rem:dynNR12},~\ref{rem:dynRN21}, and~\ref{rem:dynCC}.

Throughout the paper, the partial derivative of a function $f$ with respect to $t$ will be denoted as $(\partial/\partial t)f$ or $f_t$. Along each horizontal line in the $ts$-plane, if the point of intersection of the line and the graph of a function $f$ or a curve $\mathcal{C}_1$ is to the right (left) of the the point of intersection of the line and the graph of a function $g$ or a curve $\mathcal{C}_2$, we say that the graph of the function $f$ or the curve $\mathcal{C}_1$, \textit{bounds from the right (left)}, the graph of $g$ or the curve $\mathcal{C}_2$.

\section{Organization of the paper}
A real planar ($2\times 2$) matrix is classified into three categories based on the eigenvalues: (i) real-diagonalizable, (ii) having a pair of non-real eigenvalues, and (iii) defective (having a repeated eigenvalue with one-dimensional eigenspace). 

The sections in this paper have been divided on the basis of the Jordan forms of $A_1$ and $A_2$, the two subsystem matrices, as in Table~\ref{table:sections}. The list in Table~\ref{table:sections} consists of all possible Jordan form combinations of the subsystem matrices. In each of these sections, we will compute $\tau_{1,2}(\eta)$ and $\tau_{2,1}(\eta)$. As mentioned earlier, the switched system~\eqref{eq:system} is stable for all signals $\sigma\in \s[\tau_{1,2}(\eta),\eta] \bigcup \s[\tau_{2,1}(\eta),\eta]$ and asymptotically stable for all signals $\sigma\in \s'[\tau_{1,2}(\eta),\eta] \bigcup \s'[\tau_{2,1}(\eta),\eta]$. Hence for $\tau(\eta)=\min\{\tau_{1,2}(\eta),\tau_{2,1}(\eta)\}$, the switched system~\eqref{eq:system} is stable for each $\sigma\in \s[\tau(\eta),\eta]$ and asymptotically stable for each $\sigma\in \s'[\tau(\eta),\eta]$. 

\begin{table}[h!]
	\centering
	\begin{tabular}{|c|c|c|}
		\hline
		\textbf{Section} & ${\bf A_1}$ (Hurwitz stable) & ${\bf A_2}$ (unstable)\\
		\hline
		\ref{sec:RR} & real-diagonalizable &  real-diagonalizable\\
		\hline
		\ref{sec:RC} & real-diagonalizable & non-real eigenvalues\\
		\hline
		\ref{sec:CR} & non-real eigenvalues & real-diagonalizable\\ 
		\hline
		\ref{sec:NN} & defective & defective \\
		\hline
		\ref{sec:NC} & defective & non-real eigenvalues \\
		\hline
		\ref{sec:CN} & non-real eigenvalues & defective \\
		\hline
		\ref{sec:NR} & defective & real-diagonalizable \\
		\hline
		\ref{sec:RN} & real-diagonalizable & defective\\
		\hline
		\ref{sec:CC} & non-real eigenvalues & non-real eigenvalues  \\
		\hline
	\end{tabular}\caption{List of sections.}\label{table:sections}
\end{table}

The results in~\cite{agarwal2018simple} can also be applied to compute dwell-flee relations for the system~\eqref{eq:system}. Their results make use of \emph{complex} Jordan basis matrices. The results, however, remain unaltered when \emph{real} Jordan basis matrices are used instead. Moreover, the dwell-flee relations obtained in their work are linear and use complex Jordan basis matrices with unit column norms. In the results presented in this paper, the dwell-flee relations give lower dwell time for given flee time since we vary over all possible Jordan basis matrices and also due to more efficient grouping of terms of the flow, see~\eqref{eq:flow}.

An extension of results for switched system~\eqref{eq:system} to a special class of symmetric bilinear switched systems is discussed in Section~\ref{sec:sbs}, and to multimodal planar system in which the switching between the subsystems is governed by an undirected star graph is given in Section~\ref{sec:multimodal}. A comparison between dwell-flee relations with those in the existing literature is discussed in Section~\ref{sec:examples} using several examples.

\section{Both subsystems real-diagonalizable} \label{sec:RR}

In this section, we will consider subsystems $A_1$ and $A_2$ both of which are real-diagonalizable. Fix a Jordan basis matrices $P_1$ of $A_1$ with Jordan form $J_1=\text{diag}\left(-p_1,-q_1\right)$, with $0<p_1\le q_1$, and $P_2$ of $A_2$ with Jordan form $J_2=\text{diag}\left(p_2,q_2\right)$, with either $0<p_2\le q_2$ or $p_2\le 0<q_2$. Let $M=P_2^{-1}P_1$. Since the inequalities~\eqref{eq:tau12} and~\eqref{eq:tau21} do not alter if $M$ is multiplied by a nonzero scalar, $P_1$ and $P_2$ are chosen so that the matrix $M$ has determinant 1. Denote $M$ as $\begin{pmatrix}
	a&b\\ c&d
\end{pmatrix}$. Observe that one of the entries of $M$ is 0 if and only if the matrices $A_1$ and $A_2$ have a common eigenvector. For $i=1,2$, any other Jordan basis matrix of $A_i$ corresponding to the Jordan form $J_i$ is of the form $P_iD_i$ for some diagonal invertible matrix $D_i$. Let $D_1=\text{diag}(\lambda_1,1/\lambda_1)$ and $D_2=\text{diag}(\lambda_2,1/\lambda_2)$. We will vary $\lambda_1$ and $\lambda_2$ over nonzero real numbers to compute $\tau_{1,2}(\eta)$ and $\tau_{2,1}(\eta)$. Let $M(\lambda_1,\lambda_2)=D_2^{-1}MD_1$.

\begin{remark}\label{rem:RRdiag}
	If either $p_1=q_1>0$ or $p_2=q_2>0$, the subsystem matrices $A_1$ and $A_2$ commute and hence are simultaneously diagonalizable. Inequalities~\eqref{eq:tau12} and~\eqref{eq:tau21} both reduce to $\left\|\text{diag}({\rm{e}}^{-p_1 t+p_2 s},{\rm{e}}^{-q_1 t+q_2 s})\right\|<1$. For given flee time $\eta>0$, let $\tau(\eta)=\min(p_2/p_1,q_2/q_1)\eta$f. Then the switched system~\eqref{eq:system} is stable for all signals $\sigma\in \s[\tau(\eta),\eta]$ and asymptotically stable for all signals $\sigma\in \s'[\tau(\eta),\eta]$.
\end{remark}
In view of Remark~\ref{rem:RRdiag}, we will assume that $p_1\ne q_1$ and $p_2\ne q_2$ in Sections~\ref{sec:RR12} and~\ref{sec:RR21}.

\subsection{Computing $\tau_{1,2}(\eta)$} \label{sec:RR12}
In this section, we will compute $\tau_{1,2}(\eta)$. Let $\lambda_1,\lambda_2$ be nonzero real numbers. For given flee time $\eta>0$, we will compute a dwell time $\tau_{1,2}(P_1D_1,P_2D_2,\eta)$ such that for all $t>\tau_{1,2}(P_1D_1,P_2D_2,\eta)$ and $s<\eta$, $\|M(\lambda_1,\lambda_2)^{-1}{\rm{e}}^{J_2 s}\,M(\lambda_1,\lambda_2){\rm{e}}^{J_1 t}\|<1$. Let $M(\lambda_1)=MD_1$. Note that $M(\lambda_1,\lambda_2)^{-1}{\rm{e}}^{J_2 s}\,M(\lambda_1,\lambda_2){\rm{e}}^{J_1 t}=M(\lambda_1)^{-1}{\rm{e}}^{J_2 s}\,M(\lambda_1){\rm{e}}^{J_1 t}$. Now $\|M(\lambda_1)^{-1}{\rm{e}}^{J_2 s}\,M(\lambda_1){\rm{e}}^{J_1 t}\|<1$ if and only if
\begin{enumerate}[(C1)]
	\item $(q_2+p_2) s<(q_1+p_1) t$, and
	\item the function $f(\lambda_1,t,s)<0$, where $f(\lambda_1,t,s)$ equals
	\begin{eqnarray*}
		\left(ad\, {\rm{e}}^{p_2s-p_1 t}-bc\, {\rm{e}}^{q_2 s-p_1t}\right)^2+ \frac{b^2d^2}{\lambda_1^4}\left({\rm{e}}^{p_2s-q_1t}-{\rm{e}}^{q_2s-q_1 t}\right)^2+a^2c^2\,\lambda_1^4\times \\
		\left({\rm{e}}^{q_2 s-p_1t}-{\rm{e}}^{p_2 s-p_1t}\right)^2	+\left(ad\,{\rm{e}}^{q_2 s-q_1t}-bc\, {\rm{e}}^{p_2s-q_1t}\right)^2-1-{\rm{e}}^{2(p_2+q_2)s-2(p_1+q_1)t}.
	\end{eqnarray*} 
\end{enumerate}

\begin{theorem}\label{rr:12zero}
	For given $\epsilon>0$ and flee time $\eta>0$,
	\begin{enumerate}[(i)]
		\item if at least one of the off-diagonal entries of transition matrix $M$ is zero, let $\tau_{1,2}(\eta)=\max\left(q_2/q_1,\,p_2/p_1\right)\eta+\epsilon$.
		\item if at least one of the diagonal entries of transition matrix $M$ is zero, let $\tau_{1,2}(\eta)=(q_2/p_1)\eta+\epsilon$.
	\end{enumerate}
	Then the switched system~\eqref{eq:system} is asymptotically stable for all signals $\sigma\in \s[\tau_{1,2}(\eta),\eta]$.
\end{theorem}

\begin{proof} Note that $ad=1$ if and only if one of the off-diagonal entries of the transition matrix $M$ is zero, and $ad=0$ if and only if one of the diagonal entries of the transition matrix $M$ is zero. 
	
	\noindent When $c=0$, $f(\lambda_1,t,s)<0$ if and only if for $s>0$, \[
	\frac{b^2d^2}{\lambda_1^4}<\frac{(1-{\rm{e}}^{2p_2 s-2p_1 t})(1-{\rm{e}}^{2q_2 s-2q_1 t})}{({\rm{e}}^{p_2 s}-{\rm{e}}^{q_2 s})^2}{\rm{e}}^{2 q_1 t}=m(t,s). 
	\]
	The function $m(t,s)$ is zero on lines $t=(q_2/q_1)s$ and $t=(p_2/p_1)s$, negative in the region $\{(t,s)\in Q_1\ \colon\ \min\left(q_2/q_1,\,p_2/p_1\right)s<t<\max\left(q_2/q_1,\,p_2/p_1\right) s\}$, and is positive elsewhere. Fix $\lambda_1>0$, then since $\max\left(q_2/q_1,\,p_2/p_1\right)\ge (q_2+p_2)/(q_1+p_1)\ge\min(q_2/q_1,\,p_2/p_1)$, the points $(t,s)$ satisfying (C1) and the equality $b^2d^2/\lambda_1^4=m(t,s)$ lie entirely in $\Omega_1=\{(t,s)\in Q_1\ \colon\ t>\max\left(q_2/q_1,\,p_2/p_1\right)s\}$. 
	
	Also, $m(t,s)>0$, $m_t(t,s)>0$, and $m_s(t,s)<0$, in $\Omega_1$. Thus for each $s>0$, there exists a unique $\left(t(\lambda_1,s),s\right)\in\Omega_1$ such that $b^2d^2/\lambda_1^4=m\left(t(\lambda_1,s),s\right)$. Moreover, using the implicit function theorem, we obtain a function $\tau_{\lambda_1}\colon(0,\infty)\to\mathbb{R}$, defined as $\tau_{\lambda_1}(s)=t(\lambda_1,s)$, which parametrizes the solution in the feasible region and satisfies $\tau_{\lambda_1}'(s)>0$. The function $\tau_{\lambda_1}$ can be extended continuously on $[0,\infty)$ by defining $\tau_{\lambda_1}(0)=0$ since $f(\lambda_1,t,0)=0$ if and only if $t=0$. Note that for each fixed $s>0$, $\tau_{\lambda_1}(s)$ is decreasing in $\lambda_1$.

	\begin{figure}[h!]
		\centering
		\includegraphics[width=.4\textwidth]{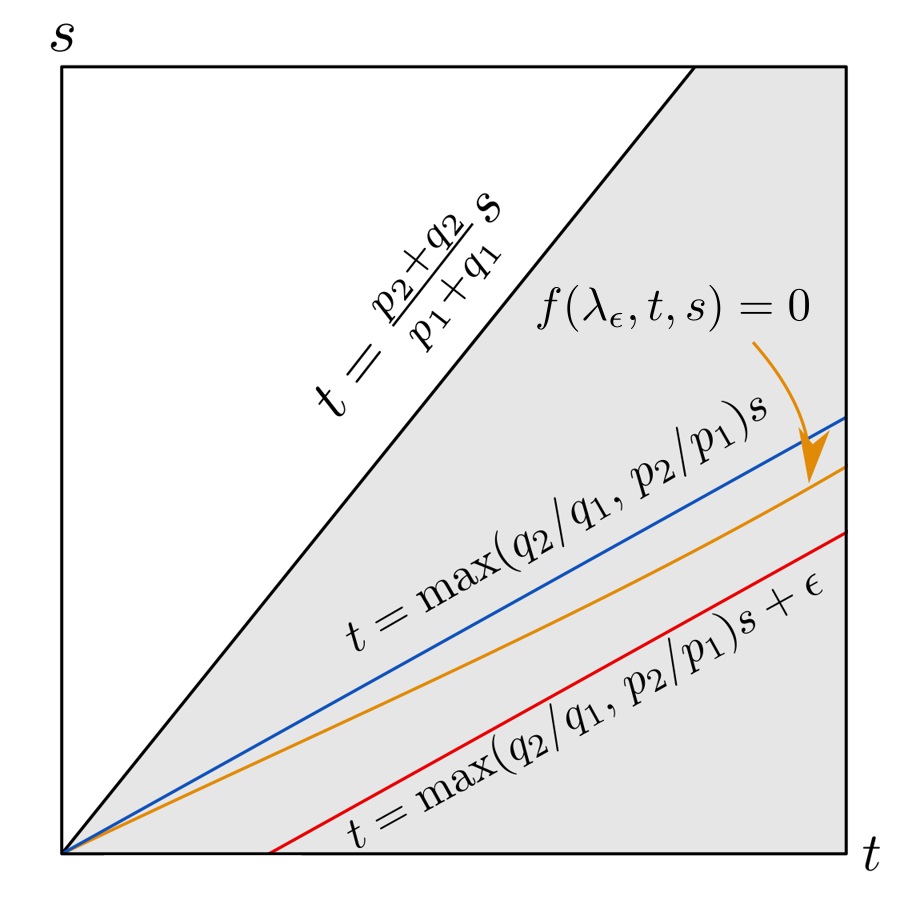}
		\caption{Zero set of $f\left(\lambda_{\epsilon},t,s\right)$ when $c=0$.}
		\label{fig:ad=1}
	\end{figure}
	
	Let $\epsilon>0$ and $\eta>0$ be given. The sequence of functions $\{\tau_k\}_{k\in\mathbb{N}}$ converges uniformly to the function $\tau(s)=\max\left(q_2/q_1,\,p_2/p_1\right)s$, on the compact set $[0,\eta]$, see Appendix~\ref{ap:proof} for details. Thus we choose a large enough $\lambda_{\epsilon}\in\mathbb{N}$ satisfying $\tau_{\lambda_{\epsilon}}(s)-\tau(s)<\epsilon/2$, for all $s\in[0,\eta]$, see Figure~\ref{fig:ad=1}. Now $f\left(\lambda_{\epsilon},t,s\right)<0$ in the region $\{(t,s)\in Q_1\ \colon\ t\ge \tau(\eta)+\epsilon/2,\ s\le \eta\}$. Also, (C1) is satisfied in this region by the earlier observations. Since $f\left(\lambda_{\epsilon},t,s\right)$ is bounded away from zero in the region $\mathcal{R}_{\tau(\eta)+\epsilon,\eta}=\{(t,s)\in Q_1\ \colon\ t\ge \tau(\eta)+\epsilon,\ s\le \eta\}$, there exists $\rho>0$ such that $\left\| M(\lambda_1)^{-1}{\rm{e}}^{J_2 s}\,M(\lambda_1){\rm{e}}^{J_1 t}\right\|<\rho<1$ in $\mathcal{R}_{\tau(\eta)+\epsilon,\eta}$. Hence we have proved the result. Further, the condition (C2) reduces to  
	\begin{eqnarray*}
		a^2 c^2\lambda_1^4&<&\frac{(1-{\rm{e}}^{2p_2 s-2p_1 t})(1-{\rm{e}}^{2q_2 s-2q_1 t})}{({\rm{e}}^{p_2 s}-{\rm{e}}^{q_2 s})^2}{\rm{e}}^{2 p_1 t} \ \ (\text{if $b=0$}),\\
		\frac{b^2 d^2}{\lambda_1^4}&<&\frac{(1-{\rm{e}}^{2q_2 s-2p_1 t})(1-{\rm{e}}^{2p_2 s-2q_1 t})}{({\rm{e}}^{p_2 s}-{\rm{e}}^{q_2 s})^2}{\rm{e}}^{2 q_1 t} \ \ (\text{if $a=0$}),\text{ and}\\
		a^2 c^2\lambda_1^4&<&\frac{(1-{\rm{e}}^{2q_2 s-2p_1 t})(1-{\rm{e}}^{2p_2 s-2q_1 t})}{({\rm{e}}^{p_2 s}-{\rm{e}}^{q_2 s})^2}{\rm{e}}^{2 p_1 t} \ \ (\text{if $d=0$}).\\
	\end{eqnarray*}
	The functions in right of each of these inequalities are increasing in $t$ and decreasing in $s$ in $\Omega_1$ when $b=0$, and in $\Omega_2=\{(t,s)\in Q_1\ \colon\ p_1 t>q_2 s\}$ when $a=0$ or $d=0$. Hence, for given $\epsilon$, similar arguments can be repeated by taking $\lambda_1$ large enough when $a=0$, and $\lambda_1$ small enough when $b=0$ or $d=0$.
\end{proof}

\noindent When each entry of the transition matrix $M$ is nonzero, we define the function 
\[
k(t,s)=f\left(\sqrt[8]{\frac{b^2d^2}{a^2c^2}\, {\rm{e}}^{-2(q_1-p_1)t}},t,s\right)
\]
which equals
\[
\left[ad\left({\rm{e}}^{p_2s-p_1t}-\epsilon{\rm{e}}^{q_2 s-q_1t}\right)-bc\left({\rm{e}}^{q_2 s-p_1t}-\epsilon{\rm{e}}^{p_2s-q_1 t}\right)\right]^2-\left({\rm{e}}^{(p_2+q_2)s-(p_1+q_1)t}-\epsilon\right)^2,
\] 
where $\epsilon=\dfrac{abcd}{|abcd|}$. Note that $\epsilon=1$ if and only if $ad<0$ or $ad>1$. Further, for all $\lambda_1>0$ and $t,s\ge 0$, $k(t,s)\le f(\lambda_1,t,s)$. Interestingly, we will observe that studying the zero set $\mathcal{C}=\{(t,s)\in Q_1\ \colon\ k(t,s)=0\}$ of $k$ will play a crucial role in obtaining an expression for $\tau_{1,2}(\eta)$.

\begin{theorem}\label{rr:12abcd>0} 
	Suppose $\epsilon=1$. For given flee time $\eta>0$,
	\begin{enumerate}[(i)]
		\item if $ad<0$, let $\tau_{1,2}(\eta)$ be the unique positive solution $t>0$ of \[
		\frac{\left({\rm e}^{q_2 \eta-p_1 t}-1\right)\left(1+{\rm e}^{p_2 \eta-q_1 t}\right)}{\left({\rm e}^{-p_1 t}+{\rm e}^{-q_1 t}\right)\left({\rm e}^{q_2 \eta}-{\rm e}^{p_2 \eta}\right)}=ad.
		\]
		\item if $ad>1$, let $\tau_{1,2}(\eta)=\max\left((p_2/p_1)\eta,T(\eta)\right)$, where $T(\eta)$ is the unique positive solution $t$ of
		\[
		\frac{\left(1+{\rm e}^{q_2 \eta-p_1 t}\right)\left(1-{\rm e}^{p_2 \eta-q_1 t}\right)}{\left({\rm e}^{-p_1 t}+{\rm e}^{-q_1 t}\right)\left({\rm e}^{q_2 \eta}-{\rm e}^{p_2 \eta}\right)}=ad.
		\]
	\end{enumerate}
	Then the switched system~\eqref{eq:system} is stable for all signals $\sigma\in \s[\tau_{1,2}(\eta),\eta]$ and asymptotically stable for all signals $\sigma\in \s'[\tau_{1,2}(\eta),\eta]$.
\end{theorem}

\begin{proof}
	Since $\epsilon=1$, the zero set of $k$ is $\mathcal{C}=\{(t,s)\in Q_1\ \colon\ ad=\ell_\pm (t,s)\}$, where \begin{eqnarray*}
		\ell_\pm(t,s)&=&\frac{\left({\rm e}^{q_2 s-p_1 t}\pm 1\right)\left(1\mp{\rm e}^{p_2 s-q_1 t}\right)}{\left({\rm e}^{-p_1 t}+{\rm e}^{-q_1 t}\right)\left({\rm e}^{q_2 s}-{\rm e}^{p_2 s}\right)},\text{ for }s\ne 0.
	\end{eqnarray*}
	Let $\Omega(\alpha,\beta)=\{(t,s)\in Q_1\colon\ \alpha t>\beta s\}$. Then (C1) is satisfied in the region $\Omega(p_1+q_1,p_2+q_2)$, denote it by $\Omega$. Let $\Omega_0=\Omega(p_1,p_2)$, $\Omega_1=\Omega(p_1,q_2)$, $\Omega_2=\Omega(q_1,p_2)$, and $\Omega_3=\Omega(q_1,q_2)$. These sets satisfy $\Omega_1\subset\Omega_0,\ \Omega,\ \Omega_3\subset\Omega_2$. Also, $\Omega_0\subseteq\Omega\subseteq\Omega_3$ if and only if $q_1p_2\ge q_2p_1$, and $\Omega_3\subseteq\Omega\subseteq\Omega_0$ if and only if $q_1p_2\le q_2p_1$. Moreover, when $p_2<0<q_2$, $\Omega_0=\Omega_2=Q_1$. Observe that
	\begin{itemize}
		\item $\ell_+$ is positive in $\Omega_2$, hence in the feasible region $\Omega$.
		\item $(\ell_+)_t>0$ in $Q_1$; $(\ell_+)_s<0$ in $\Omega\cup\Omega_3$ (when $0<p_2<q_2$) and in $\Omega_3$ (when $p_2<0<q_2$), see Appendix~\ref{ap:proof} for proofs.
		\item $\ell_+((q_2/q_1)s,s)=1$, for all $s>0$ and $\ell_+>1$ in $\Omega_3$.
		\item $\ell_-$ is negative in $\Omega_1$.
		\item $(\ell_-)_s>0$ and $(\ell_-)_t<0$ in $\Omega_1$, see Appendix~\ref{ap:proof} for proofs.
		\item $\ell_-((p_2/p_1)s,s)=1$, for all $s>0$ and $\ell_-(t,s)<1$ in $\Omega_0$.
	\end{itemize}
	
	By the above observations, when $ad<0$, the part of $\mathcal{C}$ in the feasible region is given by $ad=\ell_-(t,s)$, and lies in $\Omega_1$. Moreover, due to the above observations about the partial derivatives of $\ell_-$, we conclude the existence of a function $T\colon(0,\infty)\to\mathbb{R}$ such that $ad=\ell_-(T(s),s)$. Moreover $T'(s)>0$. Since $\lim_{t\to\infty}k(t,s)=-1$, we conclude $k(t,s)<0$ in the region lying to the right of the curve $(T(s),s)$. Hence $k(t,s)\le 0$ in the region $\mathcal{R}_{T(\eta),\eta}=\{(t,s)\in \Omega_1\ \colon\ s\le \eta,\ t\ge T(\eta)\}$, with $k(t,s)=0$ if and only if $(t,s)=(T(\eta),\eta)$.
	\begin{figure}[h!]
		\centering
		\includegraphics[width=.4\textwidth]{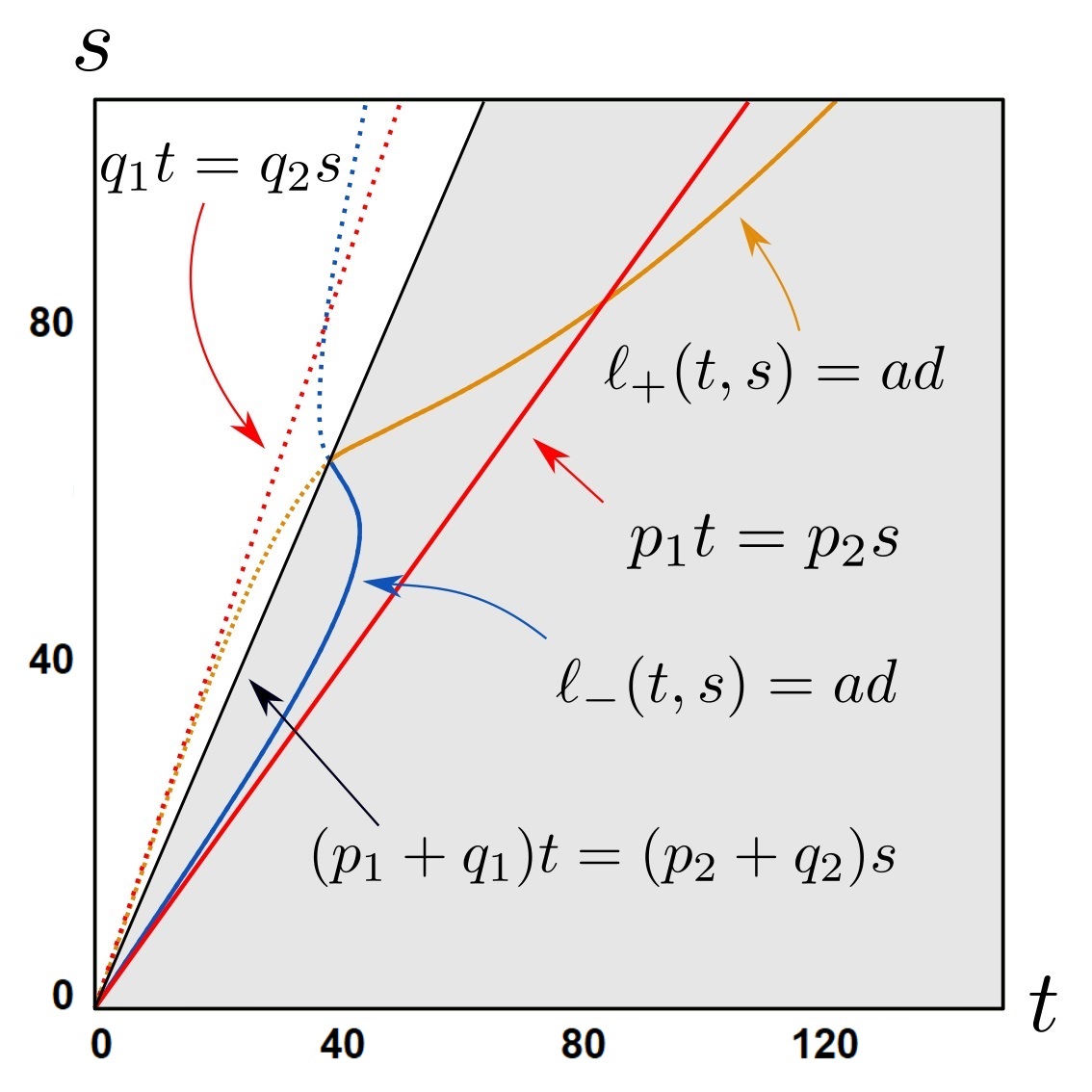}
		\caption{Zero set $\mathcal{C}$ when $ad>1$ for the values $ad=1.1$, $p_1=p_2=0.1$, $q_1=0.3$, and $q_2=0.14$.}
		\label{fig:ep1}
	\end{figure}
	
	When $ad>1$, observe that the zero set of $\ell_-(t,s)=ad$ lies in the region $Q_1\setminus\Omega_0$, since $(\ell_-)_s>0$ in $\Omega_0$ and $\ell_-(p_2 s/p_1,s)=1$. In this case, for given $t_0>0$, there may be more than one solutions $(s,t_0)\in\Omega$ such that $\ell_-(s,t_0)=ad$, refer Figure~\ref{fig:ep1}. The portion of $\mathcal{C}$ given by $\ell_-(t,s)=ad$ is bounded from the right by the line $t=(p_2/p_1) s$. Moreover the portion of $\mathcal{C}$ given by $\ell_+(t,s)=ad$ lies in the region $\Omega\cup\Omega_3$, where it can parameterized by a differentiable function $T\colon(0,\infty)\to\mathbb{R}$ using the implicit function theorem. The function $T$ is increasing due to the observations about the partial derivatives of $\ell_+$ made earlier.
	
	Define $\tau_{1,2}(s)=\max\left((p_2/p_1)s,T(s)\right)$, for $s>0$. We conclude that $\mathcal{C}$ lies to the left of the curve $(\tau_{1,2}(s),s)$. As in the earlier case, this implies $k(t,s)<0$ in the region lying on the right of the curve $(\tau_{1,2}(s),s)$. Hence, $k(t,s)\le 0$ in the region $\mathcal{R}_{\tau_{1,2}(\eta),\eta}$, with $k(t,s)=0$ only if $(t,s)=(\tau_{1,2}(\eta),\eta)$.
	
	Now, taking $\lambda_1=\sqrt[8]{\frac{b^2d^2}{a^2c^2}\, {\rm{e}}^{-2(q_1-p_1)\tau_{1,2}(\eta)}}$, we have $f(\lambda_1,\tau_{1,2}(\eta),s)=k(\tau_{1,2}(\eta),s)\le 0$, for all $s\le \eta$, with equality only if $s=\eta$. This implies $\|M(\lambda_1)^{-1}{\rm{e}}^{J_2 s}\,M(\lambda_1){\rm{e}}^{J_1 \tau_{1,2}(\eta)}\|\le 1$, for all $s\le \eta$. Now, for all $t\ge \tau_{1,2}(\eta)$ and $s\le \eta$, we have \[
	\left\|M(\lambda_1)^{-1}{\rm{e}}^{J_2 s}\,M(\lambda_1){\rm{e}}^{J_1 t}\right\|\le \left\|M(\lambda_1)^{-1}{\rm{e}}^{J_2 s}\,M(\lambda_1){\rm{e}}^{J_1 \tau_{1,2}(\eta)}\right\| \left\|{\rm{e}}^{J_1 (t-\tau_{1,2}(\eta))}\right\|\le 1.
	\]
	Hence the switched system~\eqref{eq:system} is stable for all signals $\sigma\in \s[\tau(\eta),\eta]$ and asymptotically stable for all signals $\sigma\in \s'[\tau(\eta),\eta]$.
\end{proof}

\begin{theorem}\label{rr:12abcd<0}
	Suppose $\epsilon=-1$. For given flee time $\eta>0$, let $\tau_{1,2}(\eta)$ be the unique solution $t>\max\left(p_2/p_1,q_2/q_1\right)\eta$ of \[
	-\frac{\left({\rm e}^{q_2 \eta-p_1 t}-1\right)\left({\rm e}^{p_2 \eta-q_1 t}-1\right)}{\left({\rm e}^{-p_1 t}-{\rm e}^{-q_1 t}\right)\left({\rm e}^{q_2 \eta}-{\rm e}^{p_2 \eta}\right)}=ad.
	\]
	Then the switched system~\eqref{eq:system} is stable for all signals $\sigma\in \s[\tau_{1,2}(\eta),\eta]$ and asymptotically stable for all signals $\sigma\in \s'[\tau_{1,2}(\eta),\eta]$.
\end{theorem}	 

\begin{proof}
	Since $\epsilon=-1$, we have $0<ad<1$. The zero set of $k$ is $\mathcal{C}=\{(t,s)\in Q_1\ \colon\ ad=R_\pm (t,s)\}$, where \begin{eqnarray*}
		R_\pm(t,s)&=&\mp\frac{\left({\rm e}^{q_2 s-p_1 t}\mp 1\right)\left({\rm e}^{p_2 s-q_1 t}\mp 1\right)}{\left({\rm e}^{-p_1 t}-{\rm e}^{-q_1 t}\right)\left({\rm e}^{q_2 s}-{\rm e}^{p_2 s}\right)},\text{ for }s\ne 0.\\
	\end{eqnarray*}
	The following observations can be made about these functions:
	\begin{itemize}
		\item $R_-(t,s)>1$, for all $(t,s)\in Q_1$. Thus $\mathcal{C}=\{(t,s)\in Q_1\ \colon\ ad=R_+ (t,s)\}$.
		\item $R_+(t,s)=1$ on lines $p_1t=p_2 s$ and $q_1 t=q_2 s$, hence on the line $t=\max\left(p_2/p_1,\ q_2/q_1\right)s$.
		\item Consider $\Omega_4=\{(t,s)\in Q_1\ \colon\ t>\max\left(p_2/p_1,\ q_2/q_1\right)s\}=\Omega_0\cap\Omega_3$. Then $(R_+)_t<0$ and $(R_+)_s>0$ in $\Omega_4$.
	\end{itemize}
	
	Since $0<ad<1$, $\mathcal{C}$ has a non-trivial intersection with $\Omega_4\subseteq \Omega$. This is true since $(R_+)_t<0$ in $\Omega_4$, $\lim_{t\to \max\left(\frac{p_2}{p_1},\frac{q_2}{q_1}\right)s_0}R_+(t,s_0)=1$, and $\lim_{t\to \infty} R_+(t,s_0)=-\infty$, for all $t_0>0$. Due to the observations about partial derivatives of $R_+$ in $\Omega_4$, we conclude that $\Omega_4\cap\, \mathcal{C}$ can be parameterized as the graph of an increasing function $T\colon (0,\infty)\to \mathbb{R}$ defined as follows: $T(s)$ is the unique solution $t>\max\left(p_2/p_1,\ q_2/q_1\right)s$ of $R_+(t,s)=ad$. Since $\mathcal{C}$ does not intersect with the region lying on the right of $(T(s),s)$, we conclude that $k(t,s)<0$ in this region, hence in the region $\mathcal{R}_{T(\eta),\eta}$. 
	
	Taking $\lambda_1=\sqrt[8]{\frac{b^2d^2}{a^2c^2}\, {\rm{e}}^{-2(q_1-p_1)T(\eta)}}$ and follow the same procedure as in the proof of Theorem~\ref{rr:12abcd>0} to obtain the result.
\end{proof}

\subsection{Computing $\tau_{2,1}(\eta)$}\label{sec:RR21}
In this section, we will compute $\tau_{2,1}(\eta)$. Let $\lambda_1,\lambda_2$ be nonzero real numbers. For given flee time $\eta>0$, we will compute a dwell time $\tau_{2,1}(P_1D_1,P_2D_2,\eta)$ such that for all $t>\tau_{2,1}(P_1D_1,P_2D_2,\eta)$ and $s<\eta$, $\|M(\lambda_1,\lambda_2){\rm{e}}^{J_1 t}\,M(\lambda_1,\lambda_2)^{-1}{\rm{e}}^{J_2 s}\|<1$. Let $M(\lambda_2)=D_2^{-1}M$. Note that $M(\lambda_1,\lambda_2){\rm{e}}^{J_1 t}\,M(\lambda_1,\lambda_2)^{-1}{\rm{e}}^{J_2 s}=M(\lambda_2){\rm{e}}^{J_1 t}\,M(\lambda_2)^{-1}{\rm{e}}^{J_2 s}$. Now $\|M(\lambda_2){\rm{e}}^{J_1 t}\,M(\lambda_2)^{-1}{\rm{e}}^{J_2 s}\|<1$ if and only if
\begin{enumerate}[(C1)]
	\item $(q_2+p_2) s<(q_1+p_1) t$, and
	\item the function $f(\lambda_2,t,s)<0$, where $f(\lambda_2,t,s)$ equals
	\begin{eqnarray*}
		\left(ad\, {\rm{e}}^{p_2s-p_1 t}-bc\, {\rm{e}}^{p_2 s-q_1t}\right)^2+ \frac{a^2 b^2}{\lambda_2^4}\left({\rm{e}}^{q_2s-p_1t}-{\rm{e}}^{q_2s-q_1 t}\right)^2+c^2d^2\,\lambda_2^4\times \\
		\left({\rm{e}}^{p_2 s-p_1t}-{\rm{e}}^{p_2 s-q_1t}\right)^2	+\left(ad\,{\rm{e}}^{q_2 s-q_1t}-bc\, {\rm{e}}^{q_2s-p_1t}\right)^2-1-{\rm{e}}^{2(p_2+q_2)s-2(p_1+q_1)t}.
	\end{eqnarray*} 
\end{enumerate}

\begin{theorem}\label{rr:21zero}
	For given $\epsilon>0$ and flee time $\eta>0$,
	\begin{enumerate}[(i)]
		\item if at least one of the off-diagonal entries of transition matrix $M$ is zero, let $\tau_{2,1}(\eta)=\max\left(q_2/q_1,\,p_2/p_1\right)\eta+\epsilon$.
		\item if at least one of the diagonal entries of transition matrix $M$ is zero, let $\tau_{2,1}(\eta)=(q_2/p_1)\eta+\epsilon$.
	\end{enumerate}
	Then the switched system~\eqref{eq:system} is asymptotically stable for all signals $\sigma\in \s[\tau_{2,1}(\eta),\eta]$.
\end{theorem}

\begin{proof}
	Condition (C2) reduces to 
	\begin{eqnarray*}
		c^2 d^2\lambda_2^4&<&\frac{(1-{\rm{e}}^{2p_2 s-2q_1 t})(1-{\rm{e}}^{2q_2 s-2p_1 t})}{({\rm{e}}^{-p_1 t}-{\rm{e}}^{-q_1 t})^2}{\rm{e}}^{-2 p_2 s} \ \ (\text{if $a=0$}),\\
		c^2 d^2\lambda_2^4&<&\frac{(1-{\rm{e}}^{2p_2 s-2p_1 t})(1-{\rm{e}}^{2q_2 s-2q_1 t})}{({\rm{e}}^{-p_1 t}-{\rm{e}}^{-q_1 t})^2}{\rm{e}}^{-2 p_2 s}\ \ (\text{if $b=0$}),\\
		\frac{a^2 b^2}{\lambda_2^4}&<&\frac{(1-{\rm{e}}^{2p_2 s-2p_1 t})(1-{\rm{e}}^{2q_2 s-2q_1 t})}{({\rm{e}}^{-p_1 t}-{\rm{e}}^{-q_1 t})^2}{\rm{e}}^{-2 q_2 s}\ \ (\text{if $c=0$}),\text{ and}\\
		\frac{a^2 b^2}{\lambda_2^4}&<&\frac{(1-{\rm{e}}^{2p_2 s-2q_1 t})(1-{\rm{e}}^{2q_2 s-2p_1 t})}{({\rm{e}}^{-p_1 t}-{\rm{e}}^{-q_1 t})^2}{\rm{e}}^{-2 q_2 s}\ \ (\text{if $d=0$}).
	\end{eqnarray*}
	When $b$ or $c$ is zero ($a$ or $d$ is zero, respectively), proceed as in the proof of Theorem~\ref{rr:12zero} by showing that the product of partial derivatives of the function on the right side of the corresponding inequality is $-1$ in $\Omega_1$ (in $\Omega_2$, respectively).
\end{proof}

When all entries of the transition matrix $M$ are nonzero, define the function $k(t,s)=f\left(\sqrt[8]{\frac{a^2b^2}{c^2d^2}\, {\rm{e}}^{2(q_2-p_2)s}},t,s\right)$, which can be expressed as
\[
\left[ad\left({\rm{e}}^{p_2s-p_1t}-\epsilon{\rm{e}}^{q_2 s-q_1t}\right)-bc\left({\rm{e}}^{p_2 s-q_1t}-\epsilon{\rm{e}}^{q_2s-p_1 t}\right)\right]^2-\left({\rm{e}}^{(p_2+q_2)s-(p_1+q_1)t}-\epsilon\right)^2,
\] 
where $\epsilon=\dfrac{abcd}{|abcd|}$. Further for all $\lambda_2$ and $t,s\ge 0$, $k(t,s)\le f(\lambda_2,t,s)$. Again, the zero set $\mathcal{C}=\{(t,s)\in Q_1\ \colon\ k(t,s)=0\}$ will be studied to obtain an expression for $\tau_{2,1}(\eta)$.

\begin{theorem}
	Suppose $\epsilon=-1$. For given flee time $\eta>0$, let $\tau_{2,1}(\eta)$ be the unique solution $t>\max\left(p_2/p_1,\ q_2/q_1\right)s$ of \[
	-\frac{\left({\rm e}^{q_2 \eta-p_1 t}-1\right)\left({\rm e}^{p_2 \eta-q_1 t}-1\right)}{\left({\rm e}^{-p_1 t}-{\rm e}^{-q_1 t}\right)\left({\rm e}^{q_2 \eta}-{\rm e}^{p_2 \eta}\right)}=ad.
	\]
	Then the switched system~\eqref{eq:system} is stable for all signals $\sigma\in \s[\tau_{2,1}(\eta),\eta]$ and asymptotically stable for all signals $\sigma\in \s'[\tau_{2,1}(\eta),\eta]$.
\end{theorem}

\begin{proof}
	The function $k$, just defined, coincides with the function given in Section~\ref{sec:RR12} when $0<ad<1$. Since the dwell-flee relations are obtained by points satisfying $k(\tau_{2,1}(\eta),\eta)$, choosing $\lambda_2=\sqrt[8]{\frac{a^2b^2}{c^2d^2}\, {\rm{e}}^{2(q_2-p_2)\eta}}$, we obtain the result.
\end{proof}

\begin{theorem}
	Suppose $\epsilon=1$. Let $\tilde{p}_2=\max\{0,p_2\}$. For given flee time $\eta>0$,
	\begin{enumerate}[(i)]
		\item if $ad<0$, let $\tau_{2,1}(\eta)$ be the unique solution $t>0$ of \[
		\frac{\left(1+{\rm e}^{\tilde{p}_2 \eta-q_1 t}\right)\left(1-{\rm e}^{q_2 \eta-p_1 t}\right)}{\left({\rm e}^{\tilde{p}_2 \eta}+{\rm e}^{q_2 \eta}\right)\left({\rm e}^{-q_1 t}-{\rm e}^{-p_1 t}\right)}=ad.
		\]
		\item if $ad>1$, let $\tau_{2,1}(\eta)=\max\left((q_2/p_1)\eta,T(\eta)\right)$, where $T(\eta)$ is the largest solution $t>0$ of \[
		\frac{\left({\rm e}^{\tilde{p}_2\eta-q_1 t}-1\right)\left(1+{\rm e}^{q_2\eta-p_1 t}\right)}{\left({\rm e}^{\tilde{p}_2 \eta}+{\rm e}^{q_2 \eta}\right)\left({\rm e}^{-q_1 t}-{\rm e}^{-p_1 t}\right)}=ad.
		\]
	\end{enumerate}
	Then the switched system~\eqref{eq:system} is stable for all signals $\sigma\in \s[\tau_{2,1}(\eta),\eta]$ and asymptotically stable for all signals $\sigma\in \s'[\tau_{2,1}(\eta),\eta]$.
\end{theorem}

\begin{proof}
	The zero set of $k$ is $\mathcal{C}=\{(t,s)\in Q_1\ \colon\ \ell_\pm(t,s)=ad\}$, where \begin{eqnarray*}
		\ell_\pm (t,s)&=&\frac{\left({\rm e}^{p_2 s-q_1 t}\pm 1\right)\left(1\mp {\rm e}^{q_2 s-p_1 t}\right)}{\left({\rm e}^{p_2 s}+{\rm e}^{q_2 s}\right)\left({\rm e}^{-q_1 t}-{\rm e}^{-p_1 t}\right)},\text{ for }t\ne 0.
	\end{eqnarray*}
	
	When $0\le p_2<q_2$, the following observations can be made, where $\Omega_0,\Omega_1,\Omega_2$ are the sets defined in Theorem~\ref{rr:12abcd>0}:
	
	\begin{enumerate}[(O1)]
		\item $(\ell_+)_t<0$ and $(\ell_+)_s>0$ in $Q_1$.
		\item $(\ell_-)_t>0$ and $(\ell_-)_s<0$ in $\Omega_0\supseteq\Omega_1$, if $p_2/p_1\ge q_2/q_1$.
		\item $(\ell_-)_t>0$ and $(\ell_-)_s<0$ in $\Omega_1$, if $p_2/p_1\le q_2/q_1$.
		\item $\ell_->0$ in $\Omega_2$ and $\ell_+<0$ in $\Omega_1$.
		\item $\ell_->1$ in $\Omega_0$ and $\ell_+<1$ in $\Omega_3$.
	\end{enumerate}
	
	As in the proof of Theorem~\ref{rr:12abcd>0}, we can use these observations to arrive at conclusions that follow (details have been omitted). We conclude that when $ad<0$, the curve which bounds the zero set $\mathcal{C}$ from the right is given by $\ell_+(t,s)=ad$ and can be parametrized by $S\colon(0,\infty)\to\mathbb{R}$ satisfying $\ell_+(t,S(t))=ad$. The function $S$ is increasing by the implicit function theorem using observation (O1).
	
	Consider the case when $ad>1$. If $p_2/p_1\ge q_2/q_1$, $\Omega_0\subseteq\Omega\subseteq\Omega_3$. The curve $\mathcal{C}\cap\Omega_0$ is given by $\ell_-(t,s)=ad$ and its parameterization $S\colon(0,\infty)\to\mathbb{R}$ satisfying $\ell_-(t,S(t))=ad$ is an increasing function by observation (O2). Due to this, no portion of the curve $\mathcal{C}$ lies to the right of the graph of the function $S$ in the $ts$-plane. However, if $p_2/p_1< q_2/q_1$, then the parameterization $S$ satisfying $\ell_-(t,S(t))=ad$, may not be increasing. In this case, recall observation (O3). It follows that, on the line $t=(q_2/p_1)s$, \begin{itemize}
		\item $\lim_{\xi\to 0}\ell_-\left(\frac{q_2}{p_1}\xi,\xi\right)=\frac{q_1q_2-p_1p_2}{q_1q_2-p_1q_2}=\ell_0$\text{ (say)}.
		\item $\lim_{\xi\to \infty}\ell_-\left(\frac{q_2}{p_1}\xi,\xi\right)=2$.
		\item  $\ell_-\left(\frac{q_2}{p_1}\xi,\xi\right)$ is increasing in $\xi$, if $2p_1q_2\le q_1q_2+p_1p_2$, and decreasing in $\xi$, if $2p_1q_2\ge q_1q_2+p_1p_2$ and a constant function when $2p_1q_2= q_1q_2+p_1p_2$.
	\end{itemize}
	
	If $ad\ge\max\left\{\ell_0,2\right\}$ and $\ell_0\ne 2$, then the solution set of $\ell_-(t,s)=ad$ lies entirely in $\Omega_1$ and can be parametrized, using the implicit function theorem and (O3), by an increasing function $S\colon(0,\infty)\to\mathbb{R}$, below which both (C1) and (C2) are satisfied, hence $\tau_{2,1}(\eta)=S(\eta)$. Otherwise if $ad\ge \ell_0=2$, then $\ell_-(t,s)=ad$ is given by $q_2 s=p_1 t$, hence $\tau_{2,1}(\eta)=q_2\eta/p_1$.
	
	If $ad\le\min\left\{\ell_0,2\right\}$, the solution set of $\ell_-(t,s)=ad$ lies entirely in $Q_1\setminus\Omega_1$ and hence the line $q_2 s=p_1 t$ bounds the zero set. Moreover both (C1) and (C2) are satisfied in $p_1 t>q_2 s$, hence $\tau_{1,2}(\eta)=q_2\eta /p_1$.
	
	\begin{figure}[h!]
		\centering
		\includegraphics[scale=0.4]{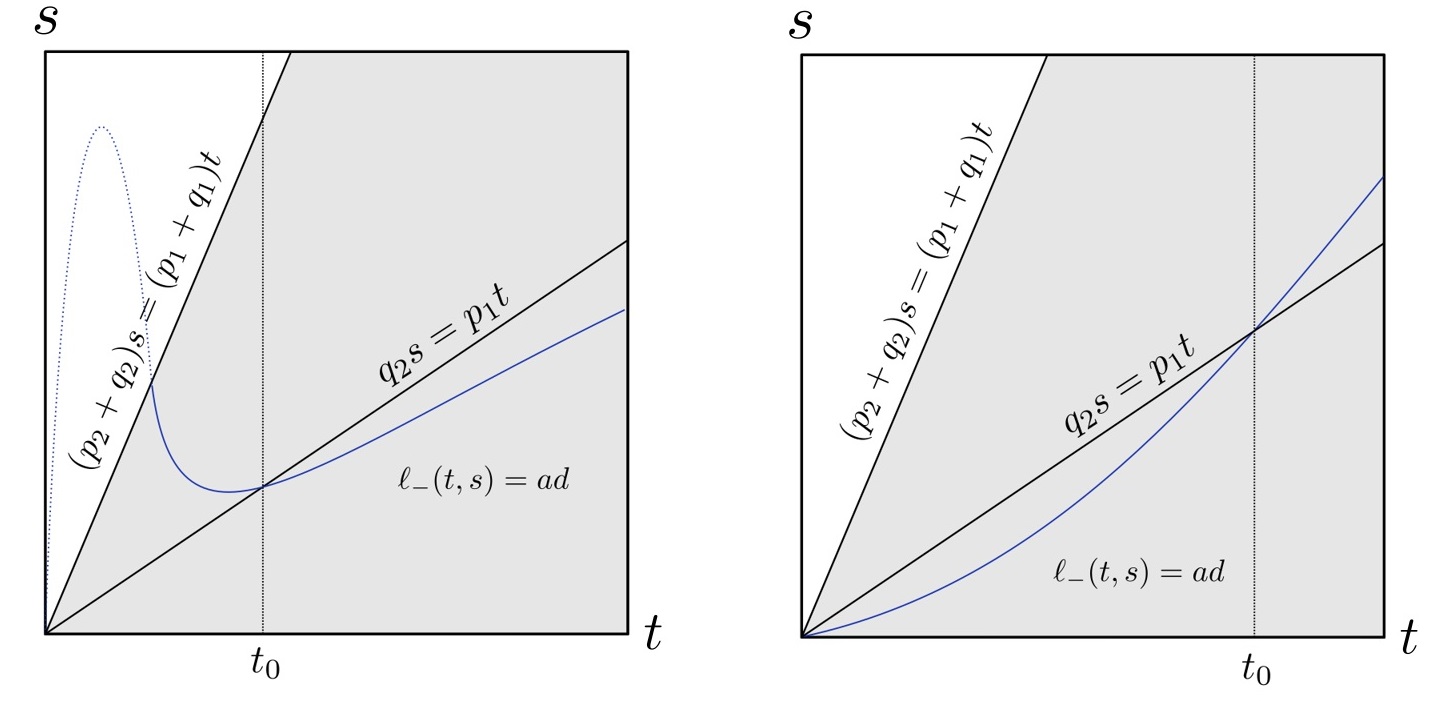}
		\caption{Zero set $\mathcal{C}$ when $2<ad<\ell_0$ (left) and $\ell_0<ad<2$ (right). }
		\label{fig:rrus}
	\end{figure}
	
	If $ad\in \left(\min\{\ell_0,2\},\max\{\ell_0,2\}\right)$ lies between $\ell_0$ and $2$, the solution set of $\ell_-(t,s)=ad$ intersects with the line $t=(q_2/p_1)s$ exactly once, at $t_0$ say. Using the fact that $\ell_-(t,0)$ is positive and increasing with range $[q_1/(q_1-p_1),\infty)$, we have two possibilities as shown in Figure~\ref{fig:rrus}, depending on the sign of $\ell_0-2$. Using the implicit function theorem and (O3), the portion of the zero set $\mathcal{C}$ in $\Omega_1$ can be parameterized by an increasing function $S_1\colon(t_0,\infty)\to\mathbb{R}$ when $2<ad<\ell_0$, and by an increasing function $S_2\colon(0,t_0)\to\mathbb{R}$ when $\ell_0<ad<2$. Hence the increasing curves \[
	\begin{cases}
		(p_1/q_2)t, & t\in(0,t_0)\\
		S_1(t), & t\in(t_0,\infty)
	\end{cases}, \ \text{and} \  \begin{cases}
		S_2(t), & t\in(0,t_0)\\
		(p_1/q_2)t, & t\in(t_0,\infty).
	\end{cases}
	\] 
	bound $\mathcal{C}$ from the right when $2<ad<\ell_0$ and $\ell_0<ad<2$, respectively. Since (C1) and (C2) are satisfied below the bounding curves obtained above, $\tau_{2,1}(\eta)=\max\left((q_2/p_1)\eta,\tau(\eta)\right)$, where $\tau(\eta)>0$ is the largest solution $t>0$ of $l_-(t,\eta)=ad$.
	
	When $p_2<0<q_2$, it is not possible to find a region of the form $\Omega(\alpha,\beta)=\{(t,s)\in Q_1\ \colon\ \alpha t>\beta s\}$ in which $(\ell_-)_t(\ell_-)_s$ is negative. Consider $\ell_\pm$ as a function of $p_2$ as well and denote it by $L_\pm$, then $L_-(0,t,s)$ lies below $L_-(p_2,t,s)$ and hence $L_-(p_2,t,s)=ac$ is bounded by the curve $L_-(0,t,s)=ac$ from the right. Similarly, $L_+(p_2,t,s)=ac$ is bounded by the curve $L_+(0,t,s)=ac$ from the right. Hence the result follows.
\end{proof}


\section{$A_1$ real-diagonalizable and $A_2$ with non-real eigenvalues}\label{sec:RC}

In this section, $A_1$ is a real-diagonalizable matrix with Jordan form $J_1=\text{diag}\left(-p_1,-q_1\right)$, where $0<p_1\le q_1$ and $A_2$ has non-real eigenvalues with Jordan form $J_2= \begin{pmatrix} \alpha_2 & \beta_2\\ -\beta_2 & \alpha_2\end{pmatrix}$, where $\alpha_2> 0$ and $\beta_2\ne 0$. Fix a Jordan basis matrix $P_1$ of $A_1$ and $P_2$ of $A_2$ such that the matrix $M=P_2^{-1}P_1=\begin{pmatrix}
	a&b\\ c&d
\end{pmatrix}$ has determinant 1. Any other Jordan basis matrix of $A_1$ corresponding to the Jordan form $J_1$ is of the form $P_1D_1$ for some diagonal invertible matrix $D_1=\text{diag}(\lambda_1,1/\lambda_1)$ with $\lambda_1\ne 0$. Moreover any other Jordan basis matrix of $A_2$ corresponding to the Jordan form $J_2$ is a nonzero scalar multiple of $P_2$. We will vary $\lambda_1$ over nonzero real numbers to compute $\tau_{1,2}(\eta)$ and $\tau_{2,1}(\eta)$. Let $M(\lambda_1)=MD_1$.

\begin{remark}\label{rem:RCdiag}
	When $p_1=q_1>0$, the inequalities~\eqref{eq:tau12} and~\eqref{eq:tau21} both reduce to ${\rm{e}}^{-p_1 t+\alpha_2 s}<1$. For given flee time $\eta>0$, let $\tau(\eta)=\alpha_2\eta/p_1$. Then the switched system~\eqref{eq:system} is stable for all signals $\sigma\in \s[\tau(\eta),\eta]$ and asymptotically stable for all signals $\sigma\in \s'[\tau(\eta),\eta]$.
\end{remark}
\noindent In view of Remark~\ref{rem:RCdiag}, we will assume that $p_1\ne q_1$ in this section.

\subsection{Computing $\tau_{1,2}(\eta)$}

In this section, we will compute $\tau_{1,2}(\eta)$. Let $\lambda_1$ be a nonzero real number. For given flee time $\eta>0$, we will compute a dwell time $\tau_{1,2}(P_1D_1,P_2,\eta)$ such that for all $t>\tau_{1,2}(P_1D_1,P_2,\eta)$ and $s<\eta$, $\|M(\lambda_1)^{-1}{\rm{e}}^{J_2 s}\,M(\lambda_1){\rm{e}}^{J_1 t}\|<1$.

\begin{theorem}\label{rc:12}
	For given flee time $\eta>0$, let $\tau_{1,2}(\eta)$ be the unique positive solution $t$ of
	\[	\alpha_2\eta=\left(\frac{q_1+p_1}{2}\right)t-\sinh^{-1}\left(\lvert ab+cd\rvert \cosh\left(\frac{q_1-p_1}{2}\right)t+\sinh\left(\frac{q_1-p_1}{2}\right)t\right).
	\]
	Then the switched system~\eqref{eq:system} is asymptotically stable for all $\sigma\in	\s\left[\tau_{1,2}(\eta),\eta\right]$.
\end{theorem}

\begin{proof}
	We have $\left\|M(\lambda_1)^{-1}{\rm{e}}^{J_2 s}\, M(\lambda_1) {\rm{e}}^{J_1 t} \right\|< 1$ if and only if
	\begin{enumerate}
		\item[(C1)] $4\alpha_2 s-2(p_1+q_1)t<0$, and 
		\item[(C2)] the function $f(\lambda_1,t,s)<0$, where $f(\lambda_1,t,s)$ equals \begin{align*}
			&\frac{1}{2}\left(\frac{(b^2+d^2)^2}{\lambda_1^4}{\rm{e}}^{-(q_1-p_1) t}+(a^2+c^2)^2\lambda_1^4 {\rm{e}}^{(q_1-p_1)t}\right)\sin^2\beta_2 s+R \sinh (q_1-p_1) t\ \times\\
			& \sin 2\beta_2 s +\left(\cos^2\beta_2 s+R^2\sin^2\beta_2 s\right)\, \cosh (q_1-p_1) t-\cosh\left((q_1+p_1) t-2\alpha_2 s\right),
		\end{align*}
		where $R=ab+cd$.
	\end{enumerate} 
	\noindent Define the function $k(t,s)=f\left(\sqrt[4]{\frac{b^2+d^2}{a^2+c^2}\, {\rm{e}}^{-(q_1-p_1)t}},t,s\right)$, which equals
	\begin{eqnarray*}(R^2+1) \sin^2 \beta_2 s+R \sinh (q_1-p_1) t\, \sin 2\beta_2 s+\left(\cos^2\beta_2 s+R^2\sin^2\beta_2 s\right)\cosh (q_1-p_1) t\nonumber \\
		-\cosh\left((q_1+p_1) t-2\alpha_2 s\right).
	\end{eqnarray*}
	Note that $k(t,s)\le f(\lambda_1,t,s)$, for all $\lambda_1,t,s$. 
	
	For $s_0\in\Lambda^0=\{0\}\cup(\pi/\beta_2)\mathbb{N}$, $k(t,s_0)=-2\sinh(q_1 t-\alpha_2 s_0)\sinh(p_1 t-\alpha_2 s_0)$. Hence the zero set $\mathcal{C}=\{(t,s)\in Q_1\ \colon\ k(t,s)=0\}$ of $k$ can be rewritten as
	\begin{eqnarray*}
		\mathcal{C}&=&\bigcup_{s_0\in\Lambda^0}\left\{\left(\frac{\alpha_2s_0}{p_1},s_0\right) ,\left(\frac{\alpha_2s_0}{q_1},s_0\right)\right\}\bigcup\left\{(t,s)\in Q_1\colon\ R=\ell_\pm(t,s)\right\},
	\end{eqnarray*} where \begin{eqnarray*}
		\ell_\pm(t,s)&=&\frac{-\cos\beta_2 s\sinh\left(\frac{q_1-p_1}{2}\right)t\pm\sinh\left(\left(\frac{q_1+p_1}{2}\right)t-\alpha_2 s\right)}{\sin\beta_2 s\cosh\left(\frac{q_1-p_1}{2}\right)t},\ s\notin \Lambda^0.
	\end{eqnarray*}
	The following observations can be made in the feasible region in which $(p_1+q_1)t>2\alpha_2 s$ (condition (C1)):
	\begin{itemize}
		\item When $s_0\in \Lambda^1=\cup_{r\in\{0\}\cup\mathbb{N}}\left(2r\pi/\beta_2,(2r+1)\pi/\beta_2\right)$, let $t_0=2\alpha_2 s_0(p_1+q_1)^{-1}$. Then
		\begin{enumerate}[(a)]
			\item $\ell_+(t,s_0)$ is increasing in $t$ and $\ell_-(t,s_0)$ is decreasing in $t$.
			\item since $\ell_+(t_0,s_0)=\ell_-(t_0,s_0)$, we have $\ell_+(t,s_0)>\ell_-(t,s_0)$, for all $t>t_0$.
		\end{enumerate}
		\item When $s_0\in \Lambda^2=\cup_{r\in\mathbb{N}}\left((2r-1)\pi/\beta_2,2r\pi/\beta_2\right)$, let $t_0=2\alpha_2 s_0(p_1+q_1)^{-1}$. Then $\ell_+(t,s_0)$ is decreasing in $t$ and $\ell_-(t,s_0)$ is increasing in $t$. Hence $\ell_+(t,s_0)<\ell_-(t,s_0)$ for all $t>t_0$.
	\end{itemize}
	Furthermore, in the subregion $\Omega=\{(t,s)\ \colon\ p_1 t>\alpha_2 s\}$, we make some additional observations, refer Figure~\ref{fig:rc12}: \begin{itemize}
		\item When $s_0\in\Lambda^1$, $\ell_+(t,s)$ is positive and $\ell_-(t,s)$ is negative.
		\item When $s_0\in\Lambda^2$, $\ell_+(t,s)$ is negative and $\ell_-(t,s)$ is positive.
	\end{itemize}
	
	\begin{figure}[h!]
		\centering
		\includegraphics[width=.4\textwidth]{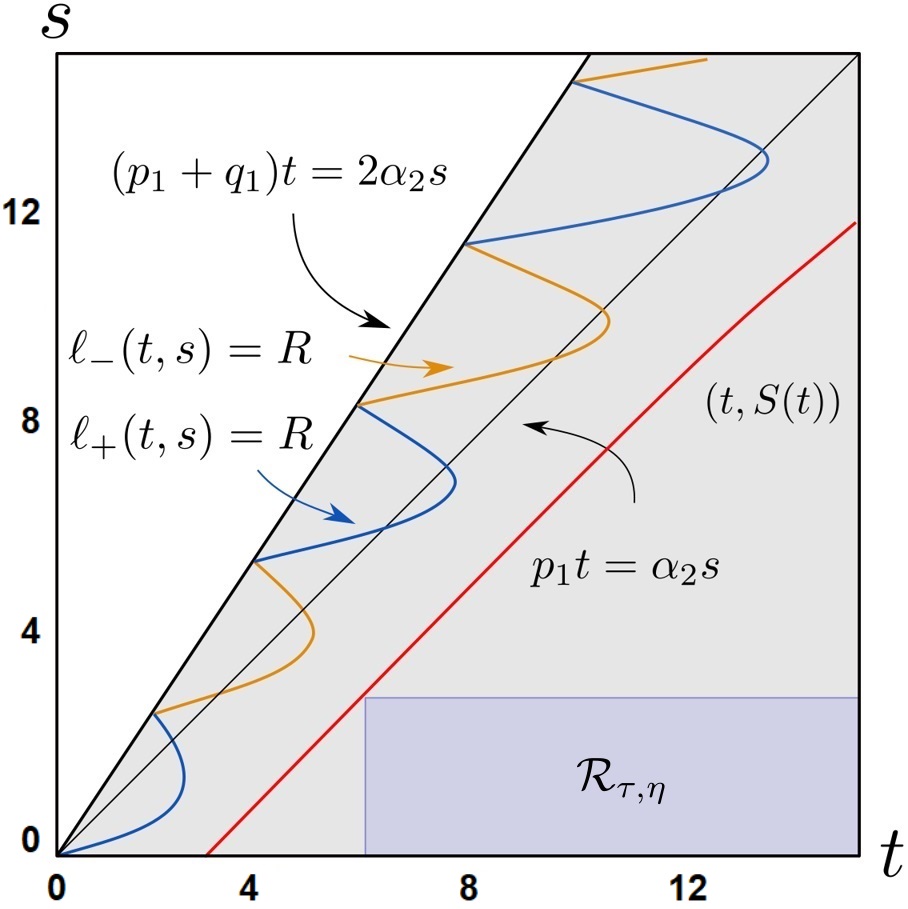}
		\caption{Graph of $\ell_\pm(t,s)$ over $\Omega$, and depiction of $\mathcal{C}$ in the feasible region for values $p_1=\alpha_2=0.1$, $\beta_2=1$ and $q_1=R=0.2$.}
		\label{fig:rc12}
	\end{figure}
	
	Now we will focus on finding a function whose graph lies to the right of the curve $\mathcal{C}$, and this graph will be used to find $\tau_{1,2}$. It turns out that it is enough to study the functions $\ell_+$ and $\ell_-$ in $\Omega$ to obtain such a function. We find this function for $R\ge 0$, the other case will follow similarly.
	
	Suppose $R\ge 0$. For $s\notin\Lambda^0$, we have $k(t,s)=\left(R-\ell_+(t,s)\right)\left(R-\ell_-(t,s)\right)$. Thus $k(t,s)<0$ if and only if $m(t,s)=\min\{\ell_+(t,s),\ell_-(t,s)\}<R<\max\{\ell_+(t,s),\ell_-(t,s)\}=\tilde{m}(t,s).$
	By the above observations, one can argue $m$ (and $\tilde{m}$) is negative and decreasing in $t$ (positive and increasing in $t$) in $\Omega$. Observe that
	\begin{eqnarray}
		\tilde{m}(t,s)&=& 
		\nonumber\left\{\begin{aligned}
			& \frac{\sinh\left(-\alpha_2 s+\left(\frac{q_1+p_1}{2}\right)t\right)-\cos\beta_2 s\sinh\left(\frac{q_1-p_1}{2}\right)t}{\lvert\sin\beta_2 s\rvert\cosh\left(\frac{q_1-p_1}{2}\right)t}, & s & \in\Lambda^1\\
			& \frac{\sinh\left(-\alpha_2 s+\left(\frac{q_1+p_1}{2}\right)t\right)+\cos\beta_2 s\sinh\left(\frac{q_1-p_1}{2}\right)t}{\lvert\sin\beta_2 s\rvert\cosh\left(\frac{q_1-p_1}{2}\right)t}, & s & \in\Lambda^2
		\end{aligned}\right.\\
		&\ge& \frac{\sinh\left(-\alpha_2 s+\left(\frac{q_1+p_1}{2}\right)t\right)-\sinh\left(\frac{q_1-p_1}{2}\right)t}{\cosh\left(\frac{q_1-p_1}{2}\right)t}.\label{rcthm:ineq:mtilde}
	\end{eqnarray}
	
	Since $\inf_{\Omega}\tilde{m}=0$, $\mathcal{C}$ has a nontrivial intersection with $\Omega$. Thus the graph of the function $S$ defined as 	\[	S(t)=\left(\frac{q_1+p_1}{2\alpha_2}\right)t-\frac{1}{\alpha_2}\sinh^{-1}\left(R\cosh\left(\frac{q_1-p_1}{2}\right)t+\sinh\left(\frac{q_1-p_1}{2}\right)t\right),
	\]
	lies to the right of the curve $\mathcal{C}$. Since the function $S$ is concave down and $\lim_{t\to\infty}S'(t)=p_1/\alpha_2>0$, $S$ is increasing in $t$.
	
	Similarly, for a general $R$, the graph of the function 
	\[	S(t)=\left(\frac{q_1+p_1}{2\alpha_2}\right)t-\frac{1}{\alpha_2}\sinh^{-1}\left(\lvert R\rvert \cosh\left(\frac{q_1-p_1}{2}\right)t+\sinh\left(\frac{q_1-p_1}{2}\right)t\right),
	\]
	lies to the right of the curve $\mathcal{C}$. Let $T(\eta)=S^{-1}(\eta)$, then $k(t,s)< 0$ in the region $\mathcal{R}_{T(\eta),\eta}=\{(t,s)\in Q_1\ \colon\ t\ge T(\eta)\text{ and }s\le \eta\}$, for all $\eta>0$. The function $k$ is strictly negative since the inequality~\eqref{rcthm:ineq:mtilde} is strict when $p_1\ne q_1$.
	
	Take $\lambda_1=\sqrt[4]{\frac{b^2+d^2}{a^2+c^2}\, {\rm{e}}^{-(q_1-p_1)T(\eta)}}$. For this choice of $\lambda_1$, we have $f(\lambda_1,T(s),s)=k(T(s),s)<0$, for all $s<\eta$. This further implies that there exists $\lambda_1>0$ (scaling matrix $D_1$) such that $\left\|M(\lambda_1)^{-1}{\rm{e}}^{J_2 s}\, M(\lambda_1) {\rm{e}}^{J_1 T(s)} \right\|< 1$, for all $s<\eta$. This implies $\left\|M(\lambda_1)^{-1}{\rm{e}}^{J_2 s}\, M(\lambda_1) {\rm{e}}^{J_1 t} \right\|\le\left\|M(\lambda_1)^{-1}{\rm{e}}^{J_2 s}\, M(\lambda_1) {\rm{e}}^{J_1 T(s)} \right\|< 1$ in the region $\mathcal{R}_{T(\eta),\eta}$. Hence the result follows.
\end{proof}

\subsection{Computing $\tau_{2,1}(\eta)$}

In this section, we will compute $\tau_{2,1}(\eta)$. Let $\lambda_1$ be a nonzero real number. For given flee time $\eta>0$, we will compute a dwell time $\tau_{2,1}(P_1D_1,P_2,\eta)$ such that for all $t>\tau_{2,1}(P_1D_1,P_2,\eta)$ and $s<\eta$, $\|M(\lambda_1){\rm{e}}^{J_1 t}\ M(\lambda_1)^{-1}{\rm{e}}^{J_2 s}\|<1$. Note that $M(\lambda_1){\rm{e}}^{J_1 t}\ M(\lambda_1)^{-1}{\rm{e}}^{J_2 s}=M{\rm{e}}^{J_1 t}\ M^{-1}{\rm{e}}^{J_2 s}$. Thus the analysis is independent of $\lambda_1$.

\begin{theorem}\label{rc:21}
	For given flee time $\eta>0$, let $\tau_{2,1}(\eta)$ be the unique positive solution $t$ of \[
	\alpha_2 \eta=\left(\frac{q_1+p_1}{2}\right)t- \sinh^{-1}\left(\sqrt{(a^2+c^2)(b^2+d^2)}\, \sinh\left(\frac{q_1-p_1}{2}\right)t\right).
	\]
	Then the switched system~\eqref{eq:system} is stable for all $\sigma\in	\s\left[\tau_{2,1}(\eta),\eta\right]$ and asymptotically stable for all $\sigma\in \s'\left[\tau_{2,1}(\eta),\eta\right]$.
\end{theorem}

\begin{proof}
	Note that $\|M{\rm{e}}^{J_1 t}\, M^{-1}{\rm{e}}^{J_2 s}\|<1$ if and only if
	\begin{enumerate}
		\item[(C1)] $2\alpha_2 s<(p_1+q_1) t$, and
		\item[(C2)] the function $f(t,s)<0$, where \[
		f(t,s)=(a^2+c^2)(b^2+d^2)\sinh^2\left(\frac{q_1-p_1}{2}\right)t-\sinh^2\left(\alpha_2 s-\left(\frac{q_1+p_1}{2}\right)t\right).
		\] 	
	\end{enumerate}
	
	Notice that the conditions do not depend on the scaling matrix $D_1$. Denote $K=(a^2+c^2)(b^2+d^2)$. The zero set of $f$ is given by \begin{eqnarray}\label{rc:zeroset}
		\nonumber \alpha_2 s&=&\left(\frac{q_1+p_1}{2}\right)t\pm \sinh^{-1}\left(\sqrt{K}\sinh\left(\frac{q_1-p_1}{2}\right)t\right)=F_\pm(t).
	\end{eqnarray}
	
	\begin{figure}[h!]
		\centering
		\includegraphics[width=.4\textwidth]{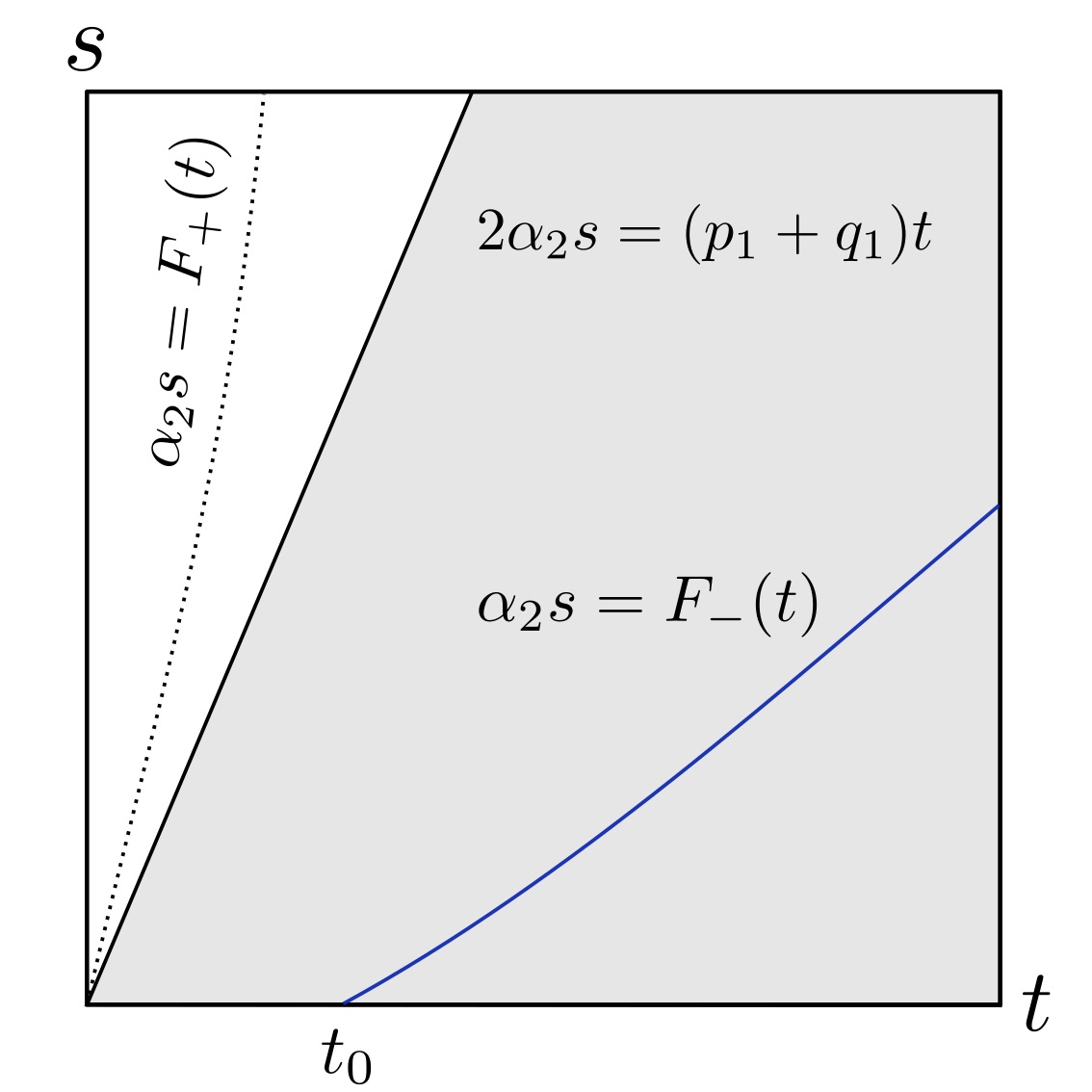}
		\caption{Plot of $f(t,s)=0$ in the feasible region.}\label{fig:rc}
	\end{figure}
	
	Refer to Figure~\ref{fig:rc} for the following discussion. The feasible region (condition (C1)) contains the part of the zero set of $f$ given by $\alpha_2 s=F_-(t)$. Since $F_-(0)=0$ and 
	{\small \[
		\frac{\partial}{\partial t} F_-(0)=\left(\frac{q_1+p_1}{2}\right)-\left(\frac{q_1-p_1}{2}\right)\sqrt{K},\ \ 
		\frac{\partial^2}{\partial t^2} F_-(t)=\frac{(K-1)\sqrt{K}(q_1-p_1)^2\sinh\left(\frac{q_1-p_1}{2}\right)t}{4\left(1+K\sinh^2 \left(\frac{q_1-p_1}{2}\right)t\right)^{3/2}},
		\]}
	$F_-$ is concave up and starts with a negative or nonnegative derivative at $t=0$ depending on the value of $K$. 
	
	When $K>((q_1+p_1)/(q_1-p_1))^2$, there exists a unique $t_0>0$ such that $F_-(t_0)=0$. For $t<t_0$, we have $F_-(t)<0<\alpha_2 s$, for all $s>0$. Thus, the part of zero set of $f$ given by $\alpha_2 s=F_-(t)$ does not intersect with the strip $0<t<t_0$.
	Now $F_-\colon\, [t_0,\infty)\to[0,\infty)$ is strictly increasing and hence is invertible. Thus for each $s>0$, there exists a unique $t_s>0$ such that $\alpha_2 s=F_-(t_s)$. 
	
	When $1\le K\le((q_1+p_1)/(q_1-p_1))^2$, $F_-$ is strictly increasing and hence invertible on $[0,\infty)$. Thus for each $s$, there exists a unique $t_s>0$ such that $\alpha_2 s=F_-(t_s)$. It is clear that both conditions (C1) and (C2) are satisfied in the region lying below the curve $\alpha_2 s=F_-(t)$.
	In particular, both the conditions are true in the region $\mathcal{R}_{\tau,\eta}=\{(t,s)\in Q_1\ \colon\, t>\tau\text{ and }s<\eta\}$, where $\eta=F_-(\tau)/\alpha_2$. Hence the result follows.
\end{proof}

\begin{remark}\label{rem:rcbetter}
	Since $ad-bc=1$ , we have $K=R^2+1$ ($K,R$ as in Theorems~\ref{rc:12} and~\ref{rc:21}). Thus we can compare the functions $S(t)$ (in the proof of Theorem~\ref{rc:12}) and $F_-(t)/\alpha_2$ (in the proof of Theorem~\ref{rc:21}). Note that $F_-(t)/\alpha_2\ge S(t)$, for all $t\ge 0$, since $\left(\sqrt{R^2+1}-1\right)\tanh\left(\frac{q_1-p_1}{2}\right)t\le \sqrt{R^2+1}-1\le \lvert R\rvert$. Therefore $\tau_{2,1}(\eta)\le \tau_{1,2}(\eta)$, for all $\eta>0$.
\end{remark}

\begin{remark}\label{rem:dynRC21}
	For the function $\tau_{2,1}$ obtained in Theorem~\ref{rc:21}, let $\sigma\in \s_{2,1}$ (defined in~\eqref{eq:signaldynam}). Then $\|Me^{J_1 t_{k+1}} M^{-1}e^{J_2 s_k}\|\le 1$, for all $k\in\mathbb{N}$. Substituting $V_1=P_1, V_2=P_2$ in~\eqref{eq:flow}, the switched system~\eqref{eq:system} is stable for all $\sigma\in	\s_{2,1}$.
\end{remark}

\section{$A_1$ has non-real eigenvalues and $A_2$ is real-diagonalizable}\label{sec:CR}

In this section, $A_1$ has non-real eigenvalues with Jordan form $J_1= \begin{pmatrix} -\alpha_1 & \beta_1\\ -\beta_1 & -\alpha_1\end{pmatrix}$, where $\alpha_1> 0$ and $\beta_1\ne 0$ and $A_2$ is a real-diagonalizable matrix with Jordan form $J_2=\text{diag}\left(p_2,q_2\right)$, where either $0<p_2\le q_2$ or $p_2\le 0<q_2$. Fix a Jordan basis matrix $P_1$ of $A_1$ and $P_2$ of $A_2$ such that the matrix $M=P_2^{-1}P_1=\begin{pmatrix}
	a&b\\ c&d
\end{pmatrix}$ has determinant 1.
Any other Jordan basis matrix of $A_1$ corresponding to the Jordan form $J_1$ is a nonzero scalar multiple of $P_1$. Moreover any other Jordan basis matrix of $A_2$ corresponding to the Jordan form $J_2$ is of the form $P_2D_2$ for some diagonal invertible matrix $D_2=\text{diag}(\lambda_2,1/\lambda_2)$ with $\lambda_2\ne 0$. We will vary $\lambda_2$ over nonzero real numbers to compute $\tau_{1,2}(\eta)$ and $\tau_{2,1}(\eta)$. Let $M(\lambda_2)=D_2^{-1}M$.

\begin{remark}\label{rem:CRdiag}
	When $p_2=q_2>0$, the inequalities~\eqref{eq:tau12} and~\eqref{eq:tau21} both reduce to ${\rm{e}}^{p_2 s-\alpha_1 t}<1$. For given flee time $\eta>0$, let $\tau(\eta)=p_2\eta/\alpha_1$. Then the switched system~\eqref{eq:system} is stable for all signals $\sigma\in \s[\tau(\eta),\eta]$ and asymptotically stable for all signals $\sigma\in \s'[\tau(\eta),\eta]$.
\end{remark}

\noindent In view of Remark~\ref{rem:CRdiag}, we will assume that $p_2\ne q_2$ in this section.

\subsection{Computing $\tau_{1,2}(\eta)$}

In this section, we will compute $\tau_{1,2}(\eta)$. Let $\lambda_2$ be a nonzero real number. For given flee time $\eta>0$, we will compute a dwell time $\tau_{1,2}(P_1,P_2D_2,\eta)$ such that for all $t>\tau_{1,2}(P_1,P_2D_2,\eta)$ and $s<\eta$, $\|M(\lambda_2)^{-1}e^{J_2 s}\,M(\lambda_2)e^{J_1 t}\|<1$. Note that $M(\lambda_2)^{-1}e^{J_2 s}\,M(\lambda_2)e^{J_1 t}=M^{-1}e^{J_2 s}\,Me^{J_1 t}$.

\begin{theorem}\label{cr:12}
	For given flee time $\eta>0$, let \[
	\tau_{1,2}(\eta)=\left(\frac{q_2+p_2}{2\alpha_1}\right)\eta+\frac{1}{\alpha_1} \sinh^{-1}\left(\sqrt{(a^2+b^2)(c^2+d^2)}\sinh\left(\frac{q_2-p_2}{2}\right)\eta\right).
	\]	
	Then the switched system~\eqref{eq:system} is stable for all $\sigma\in	\s\left[\tau_{1,2}(\eta),\eta\right]$ and asymptotically stable for all $\sigma\in \s'\left[\tau_{1,2}(\eta),\eta\right]$.
\end{theorem}

\begin{proof}
	Note that $\|M^{-1}e^{J_2 s}\,Me^{J_1 t}\|<1$ if and only if
	\begin{enumerate}
		\item[(C1)] $2\alpha_1 t>(p_2+q_2) s$, and
		\item[(C2)] the function $f(t,s)<0$, where \[
		f(t,s)=(a^2+b^2)(c^2+d^2)\sinh^2\left(\frac{q_2-p_2}{2}\right)s-\sinh^2\left(\alpha_1 t-\left(\frac{q_2+p_2}{2}\right)s\right).
		\] 	
	\end{enumerate}
	These conditions can also be obtained by making the change of variables \[(a,b,c,d,\alpha_2,\beta_2,p_1,q_1,s,t)\rightarrow(d,-b,-c,a,-\alpha_1,\beta_1,-p_2,-q_2,t,s),\]
	in the conditions in Theorem~\ref{rc:21}.
	The proof follows on observing that the part of zero set of (C2) in the feasible region is given by \[
	\alpha_1 t=\left(\frac{q_2+p_2}{2}\right)s+ \sinh^{-1}\left(\sqrt{(a^2+b^2)(c^2+d^2)}\sinh\left(\frac{q_2-p_2}{2}\right)s\right).
	\]
\end{proof}

\begin{remark}\label{rem:dynCR12}
	For the function $\tau_{1,2}$ obtained in Theorem~\ref{cr:12}, let $\sigma\in \s_{1,2}$ (defined in~\eqref{eq:signaldynam}). Then $\|M^{-1}e^{J_2 s_k} Me^{J_1 t_{k}} \|\le 1$, for all $k\in\mathbb{N}$. Substituting $V_1=P_1, V_2=P_2$ in~\eqref{eq:flow}, the switched system~\eqref{eq:system} is stable for all $\sigma\in\s_{1,2}$.
\end{remark}

\subsection{Computing $\tau_{2,1}(\eta)$}
In this section, we will compute $\tau_{2,1}(\eta)$. Let $\lambda_2$ be a nonzero real number. For given flee time $\eta>0$, we will compute a dwell time $\tau_{2,1}(P_1,P_2D_2,\eta)$ such that for all $t>\tau_{2,1}(P_1,P_2D_2,\eta)$ and $s<\eta$, $\|M(\lambda_2)e^{J_1 t}\ M(\lambda_2)^{-1}e^{J_2 s}\|<1$. 

\begin{theorem}\label{cr:21}
	For given flee time $\eta>0$, let
	\[	\tau_{2,1}(\eta)=\left(\frac{q_2+p_2}{2\alpha_1}\right)\eta+\frac{1}{\alpha_1}\sinh^{-1}\left(\lvert ac+bd\rvert \cosh\left(\frac{q_2-p_2}{2}\right)\eta+\sinh\left(\frac{q_2-p_2}{2}\right)\eta\right).
	\]
	Then the switched system~\eqref{eq:system} is asymptotically stable for all $\sigma\in	\s\left[\tau_{2,1}(\eta),\eta\right]$.
\end{theorem}

\begin{proof}
	Note that $\|M(\lambda_2)e^{J_1 t}\,M(\lambda_2)^{-1}e^{J_2 s}\|<1$ if and only if
	\begin{enumerate}
		\item[(C1)] $4\alpha_1 t>2(p_2+q_2)s$, and 
		\item[(C2)] the function $f(\lambda_2,t,s)<0$, where $f(\lambda_2,t,s)$ equals 
		\begin{align*}
			&\frac{1}{2}\left(\frac{(a^2+b^2)^2}{\lambda_2^4}e^{(q_2-p_2) s}+(c^2+d^2)^2\lambda_2^4 e^{-(q_2-p_2)s}\right)\sin^2\beta_1 t+R \sinh (q_2-p_2) s\ \times\\
			& \sin 2\beta_1 t +\left(\cos^2\beta_1 t+R^2\sin^2\beta_1 t\right)\, \cosh (q_2-p_2) s-\cosh\left((q_2+p_2) s-2\alpha_1 t\right),
		\end{align*}
		where $R=ac+bd$.
	\end{enumerate}
	These conditions can also be obtained by making the following change of variables in the conditions in Theorem~\ref{rc:12}: \[(a,b,c,d,\alpha_2,\beta_2,p_1,q_1,s,t)\rightarrow(d,-b,-c,a,-\alpha_1,\beta_1,-p_2,-q_2,t,s).\]
	
	\noindent Define the function $k(t,s)=f\left(\sqrt[4]{\frac{a^2+b^2}{c^2+d^2}\, e^{(q_2-p_2)s}},t,s\right)$, which equals
	\begin{eqnarray*}(R^2+1) \sin^2 \beta_1 t+R \sinh (q_2-p_2) s\, \sin 2\beta_1 t+\left(\cos^2\beta_1 t+R^2\sin^2\beta_1 t\right)\cosh (q_2-p_2) s\nonumber \\
		-\cosh\left((q_2+p_2) s-2\alpha_1 t\right).
	\end{eqnarray*}
	Proceeding as in the proof of Theorem~\ref{rc:12} by finding corresponding $\ell_\pm$, $m$ and $\tilde{m}$ with the change of variables mentioned above. When $R>0$, the proof follows since \[\tilde{m}(t,s)\ge \frac{\sinh\left(\alpha_1 t-\left(\frac{q_2+p_2}{2}\right)s\right)-\sinh\left(\frac{q_2-p_2}{2}\right)s}{\cosh\left(\frac{q_2-p_2}{2}\right)s}.\]
\end{proof}

\begin{remark}
	If $R=ac+bd$, then $\left(\sqrt{R^2+1}-1\right)\tanh\left(\frac{q_2-p_2}{2}\right)\eta\le \sqrt{R^2+1}-1\le \lvert R\rvert$. Hence using Theorems~\ref{cr:12} and~\ref{cr:21}, $\tau_{1,2}(\eta)\le\tau_{2,1}(\eta)$.
\end{remark}

\section{Both subsystems defective}\label{sec:NN}
Suppose $A_1$ and $A_2$ are defective matrices with Jordan forms $J_1=\begin{pmatrix}
	-n_1 & 1\\0 & -n_1
\end{pmatrix}$ and $J_2=\begin{pmatrix}
	n_2 & 1\\0 & n_2
\end{pmatrix}$, with $n_1>0$ and $n_2\ge 0$. For $i=1,2$, fix Jordan basis matrix $P_i=[u_i\ v_i]$ of $A_i$ with Jordan form $J_i$ such that $|\text{det}(P_i)|=1$. Let $M=P_2^{-1}P_1=\begin{pmatrix}
	a&b\\c&d
\end{pmatrix}$. Any other Jordan basis matrix of $A_i$ with Jordan form $J_i$ is a scalar multiple of $P_i(\epsilon_i)=P_iE_i$, where $E_i=\begin{pmatrix}
	1&\epsilon_i\\ 0&1
\end{pmatrix}$ with $\epsilon_i\in\mathbb{R}$. We will vary $\epsilon_1,\epsilon_2$ over real numbers to compute $\tau_{1,2}(\eta)$ and $\tau_{2,1}(\eta)$. Let $M(\epsilon_1,\epsilon_2)=P_2^{-1}(\epsilon_2)P_1(\epsilon_1)=\begin{pmatrix}
	a-\epsilon_2 c & b-\epsilon_2 d+\epsilon_1 a-\epsilon_1\epsilon_2 c \\ c& d+\epsilon_1 c
\end{pmatrix}$. \\
Define \[
\theta(t)=\begin{Vmatrix}\begin{pmatrix}
		1 & t\\ 0 & 1\end{pmatrix}
\end{Vmatrix}=\sqrt{1+\frac{t^2}{2}+t\sqrt{1+\frac{t^2}{4}}}. 
\]The function $\theta$ will be used in all the sections where a subsystem matrix is defective.

\subsection{Computing $\tau_{1,2}(\eta)$}

In this section, we will compute $\tau_{1,2}(\eta)$. Let $\epsilon_1,\epsilon_2$ be real numbers. For given flee time $\eta>0$, we will compute a dwell time $\tau_{1,2}(P_1(\epsilon_1),P_2(\epsilon_2),\eta)$ such that for all $t>\tau_{1,2}(P_1(\epsilon_1),P_2(\epsilon_2),\eta)$ and $s<\eta$, $\|M(\epsilon_1,\epsilon_2)^{-1}e^{J_2 s}\,M(\epsilon_1,\epsilon_2)e^{J_1 t}\|<1$. Note that $M(\epsilon_1,\epsilon_2)^{-1}e^{J_2 s}\,M(\epsilon_1,\epsilon_2)e^{J_1 t}=M(\epsilon_1)^{-1}e^{J_2 s}\,M(\epsilon_1)e^{J_1 t}$, where $M(\epsilon_1)=ME_1$.

\begin{theorem}\label{nn:12}
	For given flee time $\eta>0$, 
	\begin{enumerate}
		\item[(i)] if $c=0$, let $\tau_{1,2}(\eta)$ be the unique positive solution $t$ of $
		n_1 t-\ln\theta(t)=n_2\eta+\sinh^{-1}\left(\frac{d^2}{2}\eta\right)$.
		\item[(ii)] if $c\ne 0$, let $\tau_{1,2}(\eta)$ be the unique positive solution $t$ of $n_1 t-\ln\theta(t)=n_2\eta+\sinh^{-1}\left(\frac{c^2}{2}\eta\right)$.
	\end{enumerate}
	Then the switched system~\eqref{eq:system} is stable for all $\sigma\in	\s\left[\tau_{1,2}(\eta),\eta\right]$ and asymptotically stable for all $\sigma\in \s'\left[\tau_{1,2}(\eta),\eta\right]$.
\end{theorem}

\begin{proof} 
	Since $\left\|M(\epsilon_1)^{-1}e^{J_2 s}\,M(\epsilon_1) e^{J_1 t} \right\|\le \left\| M(\epsilon_1)^{-1}e^{J_2 s}\,M(\epsilon_1) e^{-n_1 t} \right\|\theta(t)$, $\left\|M(\epsilon_1)^{-1}e^{J_2 s}\,M(\epsilon_1) e^{J_1 t} \right\|<1$ if \begin{enumerate}
		\item[(C1)] $n_1 t-\ln\theta (t)-n_2 s>0$, and
		\item[(C2)] $f(t,s)<0$, where $f(t,s)=K(\epsilon_1)^2 s^2-\sinh^2(n_1 t-\ln\theta (t)-n_2 s)$, with $K(\epsilon_1)=\left((d+\epsilon_1 c)^2+c^2\right)/2$.
	\end{enumerate} 
	Conditions (C1) and (C2) are both satisfied in the region $g(t)>h(\epsilon_1,s)$, where
	\[
	g(t) = n_1 t-\ln\theta(t), \ h(\epsilon_1,s)=n_2 s+\sinh^{-1}\left(K(\epsilon_1)s \right).
	\]
	
	If $n_1\ge 1/2$, $g$ is increasing. Thus the zero set of $f(t,s)$ in the feasible region (region where $(C1)$ is satisfied) given by $g(t)=h(\epsilon_1,s)$ can be parametrized as $(T(\epsilon_1,s),s)$, where $T(\epsilon_1,s)=g^{-1}(h(\epsilon_1,s))$. Since $h$ is increasing in $s$ and $g^{-1}$ is increasing in $t$, $T(\epsilon_1,s)$ is increasing in $s$. Hence for all $\eta>0$, both the conditions (C1) and (C2) are satisfied in the region $\mathcal{R}_{T(\epsilon_1,\eta),\eta}=\{(t,s)\ \colon\, t\ge T(\epsilon_1,\eta)\text{ and } s\le\eta\}$ except at the point $(T(\epsilon_1,\eta),\eta)$. 
	
	If $c=0$, $K(\epsilon_1)=d^2/2$, for all $\epsilon_1$. Further if $c\ne 0$, $T(\epsilon_1,\eta)$ is minimum when $\epsilon_1=-d/c$. Since $K(-d/c)=c^2/2$, the result follows.
	
	When $n_1<1/2$, the function $g$ is concave up and has a unique positive root $t_0>0$. Also, $g$ is increasing on the interval $[t_0,\infty)$. Thus, for each $s\ge 0$, there is a unique $T(\epsilon_1,s)\ge t_0$ such that $g(T(\epsilon_1, s))=h(s)$. Note $T(\epsilon_1,0)=t_0$. Proceeding as in the case when $n_1\ge 1/2$, we obtain the result.
	
\end{proof}

\subsection{Computing $\tau_{2,1}(\eta)$}
In this section, we will compute $\tau_{2,1}(\eta)$. Let $\epsilon_1,\epsilon_2$ be real numbers. For given flee time $\eta>0$, we will compute a dwell time $\tau_{2,1}(P_1(\epsilon_1),P_2(\epsilon_2),\eta)$ such that for all $t>\tau_{2,1}(P_1(\epsilon_1),P_2(\epsilon_2),\eta)$ and $s<\eta$, $\|M(\epsilon_1,\epsilon_2)e^{J_1 t}\ M(\epsilon_1,\epsilon_2)^{-1}e^{J_2 s}\|<1$. Note that $M(\epsilon_1,\epsilon_2)e^{J_1 t}\ M(\epsilon_1,\epsilon_2)^{-1}e^{J_2 s}=M(\epsilon_2)e^{J_1 t} \ M(\epsilon_2)^{-1}e^{J_2 s}$, where $M(\epsilon_2)=E_2^{-1}M$.

\begin{theorem}\label{nn:21}
	For given flee time $\eta>0$, 
	\begin{enumerate}
		\item[(i)] if $c=0$, let $\tau_{2,1}(\eta)$ be the unique positive solution $t$ of $n_2\eta+\ln\theta(\eta)=n_1 t-\sinh^{-1}\left(\frac{a^2}{2}t\right)$.
		\item[(ii)] if $c\ne 0$, let $\tau_{2,1}(\eta)$ be the unique positive solution $t$ of $n_2\eta+\ln\theta(\eta)=n_1 t-\sinh^{-1}\left(\frac{c^2}{2}t\right)$.
	\end{enumerate}
	Then the switched system~\eqref{eq:system} is stable for all $\sigma\in	\s\left[\tau_{2,1}(\eta),\eta\right]$ and asymptotically stable for all $\sigma\in \s'\left[\tau_{2,1}(\eta),\eta\right]$.
\end{theorem}

\begin{proof} 
	Since $\left\|M(\epsilon_2)e^{J_1 t} \ M(\epsilon_2)^{-1}e^{J_2 s} \right\|\le \left\| M(\epsilon_2)e^{J_1 t} \ M(\epsilon_2)^{-1}e^{n_2 s} \right\|\theta(s)$, $\left\|M(\epsilon_2)e^{J_1 t} \ M(\epsilon_2)^{-1}e^{J_2 s} \right\|<1$ if \begin{enumerate}
		\item[(C1)] $n_1 t-\ln\theta (s)-n_2 s>0$, and
		\item[(C2)] $f(t,s)<0$, where $f(t,s)=L(\epsilon_2)^2 t^2-\sinh^2(n_1 t-\ln\theta (s)-n_2 s)$, with $L(\epsilon_2)=\left((a-\epsilon_2 c)^2+c^2\right)/2$.
	\end{enumerate} 
	Conditions (C1) and (C2) are both satisfied in the region $g(s)>h(t)$, where
	\[
	g(s) = n_2 s+\ln\theta(s), \ h(\epsilon_2,t)=n_1 t-\sinh^{-1}\left(L(\epsilon_2)t \right).
	\]
	
	The function $g$ is increasing on $[0,\infty)$. Let $\ell_0$ be a positive real number. For $\ell\ge \ell_0$ and $t\ge 0$, define a function $k(\ell,t)=n_1 t-\sinh^{-1}(\ell t)$. Note $k(\ell,0)=0$. Then for each $t\ge 0$, $k$ is decreasing in $\ell$. Furthermore if $\ell\le n_1$, $k$ is an increasing function in $t$. Otherwise if $\ell>n_1$, there exists a unique $t_0(\ell)>0$ such that $k(\ell,t)<0$ for all $t\in (0,t_0(\ell))$, $k(\ell,t_0(\ell))=0$, and $k$ is an increasing function in $t$ on the interval $[t_0(\ell),\infty)$. 
	
	Note that $h(\epsilon_2,t)=k(L(\epsilon_2),t)$, for all $\epsilon_2$, $t\ge 0$. Hence for given $\epsilon_2$ and $s> 0$, there exists a unique positive real number $T(\epsilon_2,s)$ such that $g(s)=h(\epsilon_2,T(\epsilon_2,s))=k(L(\epsilon_2),T(\epsilon_2,s))$. Since $k(\ell,t)$ is decreasing in $\ell$, $T(\epsilon_2,s)$ is minimum when $L(\epsilon_2)$ is minimum. If $c\ne 0$, the function $L(\epsilon_2)$ takes minimum value $c^2/2$ at $\epsilon_2=a/c$, otherwise if $c=0$, $L$ is the constant function $a^2/2$. Hence the result follows. 	
\end{proof}

\begin{remark}\label{rem:dynNN}
	If $c\ne 0$, let $\epsilon_1=-d/c$ and $\epsilon_2=a/c$, otherwise let $\epsilon_1,\epsilon_2$ be any real numbers. For the functions $\tau_{1,2}$ and $\tau_{2,1}$ obtained in Theorems~\ref{nn:12} and~\ref{nn:21}, let $\sigma\in \s_{1,2}$ (defined in~\eqref{eq:signaldynam}). Then $\|M(\epsilon_1)^{-1}e^{J_2 s_k} M(\epsilon_1)e^{J_1 t_{k}} \|\le 1$, for all $k\in\mathbb{N}$. Substituting $V_1=P_1(\epsilon_1), V_2=P_2(\epsilon_2)$ in~\eqref{eq:flow}, the switched system~\eqref{eq:system} is stable for all $\sigma\in\s_{1,2}$. Similarly, the switched system~\eqref{eq:system} is stable for all $\sigma\in\s_{2,1}$ (defined in~\eqref{eq:signaldynam}).
\end{remark}

\vspace{-.3in}
\section{$A_1$ defective and $A_2$ with non-real eigenvalues}\label{sec:NC}

Suppose $A_1$ is a defective matrix with Jordan form $J_1=\begin{pmatrix}
	-n_1 & 1\\0 & -n_1
\end{pmatrix}$, where $n_1>0$ and $A_2$ has non-real eigenvalues with Jordan form $J_2= \begin{pmatrix} \alpha_2 & \beta_2\\ -\beta_2 & \alpha_2\end{pmatrix}$, where $\alpha_2> 0$ and $\beta_2\ne 0$. For $i=1,2$, fix Jordan basis matrix $P_i$ of $A_i$ with Jordan form $J_i$ such that $M=P_2^{-1}P_1$ has determinant $+1$ or $-1$. Any other Jordan basis matrix of $A_1$ with Jordan form $J_1$ is a scalar multiple of $P_1(\epsilon_1)=P_1E_1$, where $E_1=\begin{pmatrix}
	1&\epsilon_1\\ 0&1
\end{pmatrix}$ with $\epsilon_1\in\mathbb{R}$. Further any other Jordan basis matrix of $A_2$ with Jordan form $J_2$ is a nonzero scalar multiple of $P_2$. We will vary $\epsilon_1$ over real numbers to compute $\tau_{1,2}(\eta)$ and $\tau_{2,1}(\eta)$. Let $M(\epsilon_1)=P_2^{-1}P_1(\epsilon_1)=\begin{pmatrix}
	a & b+\epsilon_1 a \\ c& d+\epsilon_1 c
\end{pmatrix}$.

\subsection{Computing $\tau_{1,2}(\eta)$}

In this section, we will compute $\tau_{1,2}(\eta)$. Let $\epsilon_1$ be a real number. For given flee time $\eta>0$, we will compute a dwell time $\tau_{1,2}(P_1(\epsilon_1),P_2,\eta)$ such that for all $t>\tau_{1,2}(P_1(\epsilon_1),P_2,\eta)$ and $s<\eta$, $\|M(\epsilon_1)^{-1}e^{J_2 s}\,M(\epsilon_1)e^{J_1 t}\|<1$.

\begin{theorem}\label{nc:12}
	For given flee time $\eta>0$, let $\tau_{1,2}(\eta)$ be the unique positive solution $t$ of $n_1 t-\ln\theta(t)=\alpha_2\eta+\sinh^{-1}\sqrt{\frac{1}{4}\left(a^2+c^2+\frac{1}{a^2+c^2}\right)^2-1}$.
	
	Then the switched system~\eqref{eq:system} is stable for all $\sigma\in	\s\left[\tau_{1,2}(\eta),\eta\right]$ and asymptotically stable for all $\sigma\in \s'\left[\tau_{1,2}(\eta),\eta\right]$.
\end{theorem}

\begin{proof}
	Since $\left\|M(\epsilon_1)^{-1}e^{J_2 s}\,M(\epsilon_1) e^{J_1 t} \right\|\le \left\| M(\epsilon_1)^{-1}e^{J_2 s}\,M(\epsilon_1) e^{-n_1 t} \right\|\theta(t)$, we have\\ $\left\|M(\epsilon_1)^{-1}e^{J_2 s}\,M(\epsilon_1) e^{J_1 t} \right\|<1$ if \begin{enumerate}
		\item[(C1)] $n_1t-\ln\theta(t)>\alpha_2 s$, and
		\item[(C2)] $f(t,s)<0$, where $f(t,s)=R(\epsilon_1)^2\sin^2\beta_2 s-\sinh^2(n_1t-\ln\theta(t)-\alpha_2 s)$, with $R(\epsilon_1)=\sqrt{\frac{\left(a^2+(b+\epsilon_1 a)^2+c^2+(d+\epsilon_1 c)^2\right)^2}{4}-1}$.
	\end{enumerate}
	
	Let $\epsilon_1$ be a fixed real number. The zero set of $f$ can be expressed as $n_1t-\ln\theta(t)=\alpha_2 s\pm \sinh^{-1}(R(\epsilon_1)\sin\beta_2 s)$ which is bounded between the curves $L_\pm(\epsilon_1)\colon\, n_1t-\ln\theta(t)=\alpha_2 s\pm \sinh^{-1} R(\epsilon_1)$. The curve $L_+(\epsilon_1)$ lies in the region satisfying (C1).
	
	Let $g(t)=n_1t-\ln\theta(t)$ be defined for all $t\ge 0$. Note $g(0)=0$. When $n_1\ge 1/2$, the function $g$ is increasing. When $n_1<1/2$, there exists a unique positive real number $t_0$ such that $g(t)<0$ for all $0< t<t_0$, $g(t_0)=0$, and $g$ is increasing on the interval $[t_0,\infty)$. In both the cases, for each $s>0$, there exists a unique $T(\epsilon_1,s)>0$ such that $g(T(\epsilon_1,s))=\alpha_2 s+\sinh^{-1}R(\epsilon_1)$. Let $\eta>0$ be given. Since the region below $L_+$ satisfies both the conditions (C1) and (C2), in particular, these conditions are satisfied in $\mathcal{R}_{T(\epsilon_1,\eta),\eta}=\{(t,s)\ \colon\, t>T(\epsilon_1,\eta)\text{ and }s<\eta\}$, except at the point $(T(\epsilon_1,\eta),\eta)$. Further note that $T(\epsilon_1,\eta)$ increases with $R(\epsilon_1))$. Hence $T(\epsilon_1,\eta)$ will be minimum when $\epsilon_1$ is such that $R(\epsilon_1)$ is minimum. Since $\inf_{\epsilon_1} R(\epsilon_1)=\sqrt{\frac{1}{4}\left(a^2+c^2+\frac{1}{a^2+c^2}\right)^2-1}$, the result follows.
\end{proof}

\subsection{Computing $\tau_{2,1}(\eta)$}

In this section, we will compute $\tau_{2,1}(\eta)$. Let $\epsilon_1$ be a real number. For given flee time $\eta>0$, we will compute a dwell time $\tau_{2,1}(P_1(\epsilon_1),P_2,\eta)$ such that for all $t>\tau_{2,1}(P_1(\epsilon_1),P_2,\eta)$ and $s<\eta$, $\|M(\epsilon_1)e^{J_1 t}\ M(\epsilon_1)^{-1}e^{J_2 s}\|<1$. 

\begin{theorem}\label{nc:21}
	For given flee time $\eta>0$, let $\tau_{2,1}(\eta)$ be the unique positive solution $t$ of $\alpha_2\eta=n_1 t-\sinh^{-1}\left(\left(\frac{a^2+c^2}{2}\right)t\right)$. Then the switched system~\eqref{eq:system} is stable for all $\sigma\in	\s\left[\tau_{2,1}(\eta),\eta\right]$ and asymptotically stable for all $\sigma\in \s'\left[\tau_{2,1}(\eta),\eta\right]$.
\end{theorem}

\begin{proof}
	We first note that $\left\|M(\epsilon_1) e^{J_1 t\,}M(\epsilon_1)^{-1}e^{J_2 s}\right\| = \left\|M(\epsilon_1) e^{J_1 t\,}M(\epsilon_1)^{-1}e^{\alpha_2 s} \right\|$. \\
	Now
	$\left\|M(\epsilon_1) e^{J_1 t\,}M(\epsilon_1)^{-1}e^{\alpha_2 s} \right\|<1$ if and only if \begin{enumerate}
		\item[(C1)] $n_1t>\alpha_2 s$, and
		\item[(C2)] $f(t,s)<0$, where $f(t,s)=\left(\frac{a^2+c^2}{2}\right)^2 t^2-\sinh^2(n_1t-\alpha_2 s)$.
	\end{enumerate}
	Note that these conditions are independent of $\epsilon_1$. In the feasible region (where (C1) is satisfied), the zero set of $f$ can be expressed as $n_1t-\sinh^{-1}(\ell t)=\alpha_2 s$, where $\ell=\frac{a^2+c^2}{2}$. The result follows by proceeding along the lines of the proof of Theorem~\ref{nn:21} with $\theta\equiv 1$ and $n_2=\alpha_2$. 
\end{proof}

\begin{remark}\label{rem:dynNC}
	Let $\epsilon_1=-(ab+cd)/(a^2+c^2)$. For the functions $\tau_{1,2}$ and $\tau_{2,1}$ obtained in Theorems~\ref{nc:12} and~\ref{nc:21}, let $\sigma\in \s_{1,2}$ (defined in~\eqref{eq:signaldynam}). Then $\|M(\epsilon_1)^{-1}e^{J_2 s_k} M(\epsilon_1)e^{J_1 t_{k}} \|\le 1$, for all $k\in\mathbb{N}$. Substituting $V_1=P_1(\epsilon_1), V_2=P_2$ in~\eqref{eq:flow}, the switched system~\eqref{eq:system} is stable for all $\sigma\in\s_{1,2}$. Similarly, the switched system~\eqref{eq:system} is stable for all $\sigma\in\s_{2,1}$ (defined in~\eqref{eq:signaldynam}).
\end{remark}

\section{$A_1$ has non-real eigenvalues and $A_2$ is defective}\label{sec:CN}

Suppose $A_1$ has non-real eigenvalues with Jordan form $J_1= \begin{pmatrix} -\alpha_1 & \beta_1\\ -\beta_1 & -\alpha_1\end{pmatrix}$, where $\alpha_1> 0$ and $\beta_2\ne 0$ and $A_2$ is a defective matrix with Jordan form $J_2=\begin{pmatrix}
	n_2 & 1\\0 & n_2
\end{pmatrix}$, where $n_2\ge 0$. For $i=1,2$, fix Jordan basis matrix $P_i$ of $A_i$ with Jordan form $J_i$ such that $M=P_2^{-1}P_1$ has determinant $+1$ or $-1$. Any other Jordan basis matrix of $A_1$ with Jordan form $J_1$ is a nonzero scalar multiple of $P_1$ and any other Jordan basis matrix of $A_2$ with Jordan form $J_2$ is a scalar multiple of $P_2(\epsilon_2)=P_2E_2$, where $E_2=\begin{pmatrix}
	1&\epsilon_2\\ 0&1
\end{pmatrix}$ with $\epsilon_2\in\mathbb{R}$. We will vary $\epsilon_2$ over real numbers to compute $\tau_{1,2}(\eta)$ and $\tau_{2,1}(\eta)$. Let $M(\epsilon_2)=P_2(\epsilon_2)^{-1}P_1=\begin{pmatrix}
	a-\epsilon_2 c &  b-\epsilon_2 d\\ c& d
\end{pmatrix}$. 

\subsection{Computing $\tau_{1,2}(\eta)$}

In this section, we will compute $\tau_{1,2}(\eta)$. Let $\epsilon_2$ be a real number. For given flee time $\eta>0$, we will compute a dwell time $\tau_{1,2}(P_1,P_2(\epsilon_2),\eta)$ such that for all $t>\tau_{1,2}(P_1,P_2(\epsilon_2),\eta)$ and $s<\eta$, $\|M(\epsilon_2)^{-1}e^{J_2 s}\,M(\epsilon_2)e^{J_1 t}\|<1$.

\begin{theorem}\label{cn:12}
	For given flee time $\eta>0$, let $\tau_{1,2}(\eta)=\frac{n_2}{\alpha_1}\eta+\frac{1}{\alpha_1}\sinh^{-1}\left(\left(\frac{c^2+d^2}{2}\right)\eta\right)$. Then the switched system~\eqref{eq:system} is stable for all $\sigma\in	\s\left[\tau_{1,2}(\eta),\eta\right]$ and asymptotically stable for all $\sigma\in \s'\left[\tau_{1,2}(\eta),\eta\right]$.
\end{theorem}

\begin{proof}
	Note that $\left\|M^{-1}(\epsilon_2)e^{J_2 s}\,M(\epsilon_2) e^{J_1 t} \right\|<1$ if and only if \begin{enumerate}
		\item[(C1)] $\alpha_1 t-n_2s>0$, and
		\item[(C2)] $f(t,s)<0$, where $f(t,s)=\left(\frac{c^2+d^2}{2}\right)^2 s^2-\sinh^2(\alpha_1 t-n_2 s)$. 
	\end{enumerate}
	Note that these conditions are independent of $\epsilon_2$. In the feasible region (where (C1) is satisfied), the zero set of $f$ can be expressed as $\alpha_1 t=n_2 s+\sinh^{-1}(\ell s)$, where $\ell=\frac{c^2+d^2}{2}$. Hence the result follows.
\end{proof}

\subsection{Computing $\tau_{2,1}(\eta)$}

In this section, we will compute $\tau_{2,1}(\eta)$. Let $\epsilon_2$ be a real number. For given flee time $\eta>0$, we will compute a dwell time $\tau_{2,1}(P_1,P_2(\epsilon_2),\eta)$ such that for all $t>\tau_{2,1}(P_1,P_2(\epsilon_2),\eta)$ and $s<\eta$, $\|M(\epsilon_2)e^{J_1 t}\ M(\epsilon_2)^{-1}e^{J_2 s}\|<1$.

\begin{theorem}\label{cn:21}
	For given flee time $\eta>0$, let 
	\[
	\tau_{2,1}(\eta)=\frac{n_2}{\alpha_1}\eta+\frac{1}{\alpha_1}\ln\theta(\eta)+\frac{1}{\alpha_1}\sinh^{-1}\sqrt{\frac{1}{4}\left(c^2+d^2+\frac{1}{c^2+d^2}\right)^2-1}.
	\]
	Then the switched system~\eqref{eq:system} is stable for all $\sigma\in	\s\left[\tau_{2,1}(\eta),\eta\right]$ and asymptotically stable for all $\sigma\in \s'\left[\tau_{2,1}(\eta),\eta\right]$.
\end{theorem}

\begin{proof}
	We first note that $\left\|M(\epsilon_2)e^{J_1 t}\ M(\epsilon_2)^{-1}e^{J_2 s}\right\| \le \left\|M(\epsilon_2)e^{J_1 t}\ M(\epsilon_2)^{-1}e^{n_2 s} \right\|\theta(s)$. Now
	$\left\|M(\epsilon_2)e^{J_1 t}\ M(\epsilon_2)^{-1}e^{n_2 s} \right\|\theta(s)<1$ if and only if 
	\begin{enumerate}
		\item[(C1)] $\alpha_1t-n_2s-\ln\theta(s)>0$, and
		\item[(C2)] $f(t,s)<0$, where $f(t,s)=R(\epsilon_2)^2\sin^2(\beta_1 t)-\sinh^2 (\alpha_1t-n_2s-\ln\theta(s))$, with $R(\epsilon_2)=\sqrt{\frac{\left((a-\epsilon_2 c)^2+(b-\epsilon_2 d)^2+c^2+d^2\right)^2}{4}-1}$.
	\end{enumerate}
	Let $\epsilon_2$ be a fixed real number. The zero set of $f$ can be expressed as $n_2 s+\ln\theta(s)=\alpha_1 t\pm \sinh^{-1}(R(\epsilon_2)\sin\beta_1 t)$ which is bounded between the curves $L_\pm(\epsilon_2)\colon\, n_2 s+\ln\theta(s)=\alpha_1 t\pm \sinh^{-1}(R(\epsilon_2))$. The curve $L_-(\epsilon_2)$ lies in the region satisfying the first condition (C1). For each $s>0$, there exists a unique $T(\epsilon_2,s)=\frac{n_2 s+\ln\theta(s)+\sinh^{-1}(R(\epsilon_2))}{\alpha_1}$ such that $f(T(\epsilon_2,s),s)=0$ and $f$ is negative in the region $\mathcal{R}_{T(\epsilon_2,s),s}$ except at the point $(T(\epsilon_2,s),s)$. Note that $T(\epsilon_2,s)$ is minimum when $R(\epsilon_2)$ is minimum. Since $\min_{\epsilon_2} R(\epsilon_2) = \sqrt{\frac{1}{4}\left(c^2+d^2+\frac{1}{c^2+d^2}\right)^2-1}$, the result follows.
\end{proof}

\begin{remark}\label{rem:dynCN}
	Let $\epsilon_2=(ac+bd)/(c^2+d^2)$. For the functions $\tau_{1,2}$ and $\tau_{2,1}$ obtained in Theorems~\ref{cn:12} and~\ref{cn:21}, let $\sigma\in \s_{1,2}$ (defined in~\eqref{eq:signaldynam}). Then $\|M(\epsilon_2)^{-1}e^{J_2 s_k} M(\epsilon_2)e^{J_1 t_{k}} \|\le 1$, for all $k\in\mathbb{N}$. Substituting $V_1=P_1, V_2=P_2(\epsilon_2)$ in~\eqref{eq:flow}, the switched system~\eqref{eq:system} is stable for all $\sigma\in\s_{1,2}$. Similarly, the switched system~\eqref{eq:system} is stable for all $\sigma\in\s_{2,1}$ (defined in~\eqref{eq:signaldynam}).
\end{remark}


\section{$A_1$ defective and $A_2$ real-diagonalizable}\label{sec:NR}

Suppose $A_1$ is a defective matrix with Jordan form $J_1=\begin{pmatrix}
	-n_1 & 1\\0 & -n_1
\end{pmatrix}$, where $n_1>0$ and $A_2$ is real-diagonalizable with Jordan form $J_2= \begin{pmatrix} p_2 & 0\\ 0 & q_2\end{pmatrix}$, where either $0<p_2\le q_2$ or $p_2\le 0<q_2$. For $i=1,2$, fix Jordan basis matrix $P_i$ of $A_i$ with Jordan form $J_i$ such that $M=P_2^{-1}P_1$ has determinant $+1$ or $-1$. Any other Jordan basis matrix of $A_1$ with Jordan form $J_1$ is a scalar multiple of $P_1(\epsilon_1)=P_1E_1$, where $E_1=\begin{pmatrix}
	1&\epsilon_1\\ 0&1
\end{pmatrix}$ with $\epsilon_1\in\mathbb{R}$. Further any other Jordan basis matrix of $A_2$ with Jordan form $J_2$ is a nonzero scalar multiple of $P_2D_2$ for some diagonal invertible matrix $D_2=\text{diag}(\lambda_2,1/\lambda_2)$ with $\lambda_2\ne 0$. We will vary $\epsilon_1$ over real numbers and $\lambda_2$ over nonzero real numbers to compute $\tau_{1,2}(\eta)$ and $\tau_{2,1}(\eta)$. Let $M(\epsilon_1,\lambda_2)=D_2^{-1}P_2^{-1}P_1(\epsilon_1)=\begin{pmatrix}
	a/\lambda_2 & (b+\epsilon_1 a)/\lambda_2 \\ \lambda_2c& \lambda_2(d+\epsilon_1 c)
\end{pmatrix}$. 

\subsection{Computing $\tau_{1,2}(\eta)$}

In this section, we will compute $\tau_{1,2}(\eta)$. Let $\epsilon_1$ be a real number and $\lambda_2$ be a nonzero real number. For given flee time $\eta>0$, we will compute a dwell time $\tau_{1,2}(P_1(\epsilon_1),P_2D_2,\eta)$ such that for all $t>\tau_{1,2}(P_1(\epsilon_1),P_2D_2,\eta)$ and $s<\eta$, $\|M(\epsilon_1,\lambda_2)^{-1}e^{J_2 s}\,M(\epsilon_1,\lambda_2)e^{J_1 t}\|<1$. Note that $M(\epsilon_1,\lambda_2)^{-1}e^{J_2 s}\,M(\epsilon_1,\lambda_2)e^{J_1 t}=M(\epsilon_1)^{-1}e^{J_2 s}\,M(\epsilon_1)e^{J_1 t}$, where $M(\epsilon_1)=ME_1$.

\begin{theorem}\label{nr:12}
	For given flee time $\eta>0$, let $\tau_{1,2}(\eta)$ be the unique positive solution $t$ of
	\[
	n_1 t-\ln\theta(t)=\left(\frac{p_2+q_2}{2}\right)\eta+\sinh^{-1}\left(\sqrt{K}\sinh\left(\frac{q_2-p_2}{2}\right)\eta\right),
	\] where $K=\frac{1}{4}\left(r^{-1}+r\right)^2$, with $r=\max\{2\lvert ac\rvert,1\}$. 
	
	Then the switched system~\eqref{eq:system} is stable for all $\sigma\in\s\left[\tau_{1,2}(\eta),\eta\right]$ and asymptotically stable for all $\sigma\in \s'\left[\tau_{1,2}(\eta),\eta\right]$.
\end{theorem}

\begin{proof}
	Since $\left\|M(\epsilon_1)^{-1}e^{J_2 s}\,M(\epsilon_1)e^{J_1 t} \right\| \le \left\|M(\epsilon_1)^{-1}e^{J_2 s}\,M(\epsilon_1)e^{-n_1 t} \right\|\theta(t)$, $\|M(\epsilon_1)^{-1}e^{J_2 s}\,M(\epsilon_1)e^{J_1 t}\|<1$ if
	\begin{enumerate}
		\item[(C1)] $n_1t-\ln\theta(t)>(p_2+q_2)s/2$, and
		\item[(C2)] $f(t,s)<0$, where
		\[
		f(t,s)=K(\epsilon_1)\sinh^2\left(\frac{q_2-p_2}{2}\right)s-\sinh^2\left(-\left(\frac{q_2+p_2}{2}\right)s+n_1t-\ln\theta(t) \right),
		\]
		where $K(\epsilon_1)=(a^2+(b+\epsilon_1 a)^2)(c^2+(d+\epsilon_1 c)^2)$.
	\end{enumerate}
	
	\noindent The zero set of $f$ is a union of two curves given by $g(t)=h_\pm(s)$, where \[
	g(t)=n_1 t-\ln\theta(t), \ h_\pm(s)=\left(\frac{q_2+p_2}{2}\right)s\pm \sinh^{-1}\left(\sqrt{K(\epsilon_1)}\sinh\left(\frac{q_2-p_2}{2}\right)s\right).
	\]
	
	The curve $g(t)=h_+(s)$ lies in the feasible region (in which condition (C1) is satisfied) and hence both the conditions are satisfied in the region below the curve $g(t)=h_+(s)$ (since $f$ is continuous and $f(t,0)$ is negative for large $t$). 
	
	The function $g$ was studied in the proof of Theorem~\ref{nn:12}. Further the function $h_+$ is increasing in $s$. Therefore for given $s>0$, there exists a unique $T(\epsilon_1,s)>0$ such that $g(T(\epsilon_1,s))=h_+(s)$. Since $T(\epsilon_1,s)$ increases with $K(\epsilon_1)$, we have $\min_{\epsilon_1} K(\epsilon_1) =\frac{1}{4}\left(r^{-1}+r\right)^2$, for $r=\max\{2\lvert ac\rvert,1\}$, hence the result follows.
\end{proof}

\begin{remark}\label{rem:dynNR12}
	Let $\epsilon_1$ be such that $K(\epsilon_1)$ is minimum. Then
	\[
	\epsilon_1=\begin{cases}
		-\frac{d}{c}, & c\ne 0, a=0,\\
		-\frac{b}{a}, & a\ne 0, c=0,\\
		-\frac{ad+bc}{2ac}+\sqrt{\frac{1}{\min\{2|ac|,1\}}-1}, & ac\ne 0.
	\end{cases}
	\]
	For the function $\tau_{1,2}$ obtained in Theorem~\ref{nr:12}, let $\sigma\in \s_{1,2}$ (defined in~\eqref{eq:signaldynam}).\\
	Then $\|M(\epsilon_1)^{-1}e^{J_2 s_k} M(\epsilon_1)e^{J_1 t_{k}} \|\le 1$, for all $k\in\mathbb{N}$. Substituting $V_1=P_1(\epsilon_1), V_2=P_2$ in~\eqref{eq:flow}, the switched system~\eqref{eq:system} is stable for all $\sigma\in\s_{1,2}$. 
\end{remark}

\subsection{Computing $\tau_{2,1}(\eta)$}

In this section, we will compute $\tau_{2,1}(\eta)$. Let $\epsilon_1$ be a real number and $\lambda_2$ be a nonzero real number. For given flee time $\eta>0$, we will compute a dwell time $\tau_{2,1}(P_1(\epsilon_1),P_2D_2,\eta)$ such that for all $t>\tau_{2,1}(P_1(\epsilon_1),P_2D_2,\eta)$ and $s<\eta$,\\
$\|M(\epsilon_1,\lambda_2)e^{J_1 t}\ M(\epsilon_1,\lambda_2)^{-1}e^{J_2 s}\|<1$. 

\begin{theorem}\label{nr:21}
	(With notations as in this section)
	\begin{enumerate}[(i)]
		\item Let $0<p_2=q_2$. For given flee time $\eta>0$, let $\tau_{2,1}(\eta)$ be the unique positive solution $t$ of $q_2\eta=n_1 t-\sinh^{-1}(\lvert ac\rvert t)$, when $ac\ne 0$; and otherwise let $\tau_{2,1}(\eta)=(q_2/n_1)\eta+\epsilon_0$, for any $\epsilon_0>0$.
		\item When $p_2\ne q_2$, let $\tilde{p}_2=\max\{0,p_2\}$, $\mu=\frac{q_2-\tilde{p}_2}{2}$ and $\nu=\frac{q_2+\tilde{p}_2}{2}$. For given flee time $\eta>0$, 
		\begin{enumerate}[(a)]
			\item if $ac=0$, let $\tau_{2,1}(\eta)=(q_2/n_1)\eta+\epsilon_0$, for any $\epsilon_0>0$.
			\item if $ac\notin (-n_1,0]$, let $\tau_{2,1}(\eta)$ be the unique positive solution $t$ of 
			\[
			act \cosh\left(\mu\eta\right)=-\sinh\left(\mu\eta \right)+\text{sgn}(ac) \sinh\left(n_1 t-\nu\eta\right).
			\]
			\item if $ac\in(-n_1,0)$, let $\tau_{2,1}(\eta)=\max\{\tau_+(\eta),\tau_-(\eta)\}$, where $\tau_+(\eta)$ and $\tau_-(\eta)$ are unique positive solutions $t$ of 
			\begin{eqnarray*}
				ac\tau \cosh\left(\mu t\right)&=&-\sinh\left(\mu\eta\right)+ \sinh\left(n_1 t-\nu\eta\right),\text{ and}\\
				-n_1 t \cosh\left(\mu\eta\right)&=&-\sinh\left(\mu\eta\right)-\sinh\left(n_1 t-\nu\eta\right), \text{ respectively.}
			\end{eqnarray*}
		\end{enumerate}
	\end{enumerate}
	Then the switched system~\eqref{eq:system} is stable for all $\sigma\in\s\left[\tau_{2,1}(\eta),\eta\right]$ and asymptotically stable for all $\sigma\in \s'\left[\tau_{2,1}(\eta),\eta\right]$.
\end{theorem}

\begin{proof}
	Note that $\left\|M(\epsilon_1,\lambda_2) e^{J_1 t\,}M(\epsilon_1,\lambda_2)^{-1}e^{J_2 s} \right\|<1$ if and only if
	\begin{enumerate}
		\item[(C1)] $2(p_2+q_2) s-4n_1 t<0$, and
		\item[(C2)] $f(\lambda_2,t,s)<0$, where $f(\lambda_2,t,s)$ equals
		\begin{eqnarray*}
			\frac{1}{2}\left(c^4\lambda_2^4e^{-(q_2-p_2) s}+\frac{a^4}{\lambda_2^4} e^{(q_2-p_2)s}\right) t^2+\left(1+a^2c^2t^2\right) \cosh((q_2-p_2) s)\\
			+2act \sinh((q_2-p_2) s)-\cosh((q_2+p_2) s-2 n_1 t).
		\end{eqnarray*}
	\end{enumerate}
	\noindent Note that these conditions are independent of $\epsilon_1$. When $0<p_2=q_2$, the result follows immediately.
	
	Suppose $0\le p_2<q_2$. Then (as per notation in the statement of the theorem) $\tilde{p}_2=p_2$, $\mu=\dfrac{q_2-p_2}{2}$, and $\nu=\dfrac{q_2+p_2}{2}$. Suppose $ac=0$. Without loss of generality, assume $a=0$, then $c\ne 0$. The condition (C2) reduces to \[
	\lambda_2^4<\frac{4\, \sinh (n_1t-q_2 s)\, \sinh (n_1t-p_2 s)}{c^4t^2\, e^{-2\mu s}}=m(t,s).
	\]
	The function $m(t,s)>0$ and increasing in $t$ in the sub-region $\Omega=\{(t,s)\in Q_1\colon\ n_1t>q_2 s\}$. Moreover, for any $\epsilon_0>0$, $m(t,s)\ge 4c^{-4}e^{2\mu s}(q_2 s/n_1+\epsilon_0)^{-2}\sinh^2 n_1\epsilon_0$ on the line $n_1t=q_2 s+n_1\epsilon_0$ and is thus bounded below by a positive number $p_0$. Thus, making a choice of $\lambda_2<\sqrt[4]{p_0}$, we observe that $f(\lambda_2,t,s)<0$, on $\mathcal{R}_{\tau,\eta}=\{(t,s)\in Q_1\colon\ t>\tau\text{ and }s<\eta\}$, where $n_1\tau=q_2\eta+n_1\epsilon_0$. Since $\epsilon_0>0$ was arbitrary, we have proved the claim.	
	
	When $ac\ne 0$, we define the function $k(t,s)=f\left(\sqrt[4]{\frac{a^2}{c^2}\, e^{2\mu s}},t,s\right)$. Then
	\[
	k(t,s)=a^2c^2t^2+\left(1+a^2c^2t^2\right)\, \cosh(2\mu s)+2act\, \sinh(2\mu s)-\cosh(2\nu s-2 n_1 t).
	\] 	
	Note that $k(t,s)\le f(\lambda_2,t,s)$, for all $\lambda_2,t,s$. 
	The zero set $\mathcal{C}=\{(t,s)\in Q_1\ \colon\ k(t,s)=0\}$ of $k$ can be rewritten as $\mathcal{C}=\{0,0\}\cup\{(t,s)\in Q_1\ \colon\ ac=\ell_\pm(t,s)\}$, where \[
	\ell_\pm(t,s)=\frac{1}{t} \sech\mu s \left[-\sinh \mu s\pm \sinh\left(n_1 t-\nu s\right)\right].
	\]
	In the region $\Omega_0=\{(t,s)\in Q_1\ \colon\ n_1t>\nu s\}$, the following observations can be made, see Figure~\ref{fig:nr:12} (refer to Appendix~\ref{ap:proof} for proofs of the following facts): 
	\begin{enumerate}[(O1)]
		\item $\ell_-<\ell_+$ in $\Omega_0$,
		\item the function $\ell_+(t,s)$ is decreasing in $s$ and increasing in $t$, 
		\item the function $\ell_-(t,s$) is increasing in $s$.
		\item $\ell_-$ and $\ell_+$ agree on the boundary of $\Omega_0$ in $Q_1$, and $\ell_\pm\left(\frac{\nu}{n_1}s,s\right)=\frac{-n_1\tanh\mu s}{\nu s}$, 
		the value of which increases from $-n_1\mu/\nu$ to $0$ as $s$ increases over the interval $[0,\infty)$. Thus $\ell_+(t,s)\ge -n_1\mu/\nu$ and $\ell_-(t,s)<0$ in $\Omega_0$.
	\end{enumerate} 
	
	Also, $\ell_+(t,s)>0$ in $\Omega\subset \Omega_0$. Thus we can conclude the following:
	\begin{itemize}
		\item If $ac>0$, the part of $\mathcal{C}$ in the feasible region is given by $ac=\ell_+(t,s)$. Also, $\ell_+(t,0)=t^{-1}\sinh n_1 t\ge n_1$, for all $t\ge 0$.\\	
		If $0<ac\le n_1$, for each $t>0$, there exists a unique $s_t$ such that $\ell_+(t,s_t)=ac$. Moreover, since $(\ell_+)_s<0$, by repeated application of the implicit function theorem at each $t>0$, we get a function $S\colon (0,\infty)\to\mathbb{R}$ whose graph is $\mathcal{C}\cap\Omega_0$.\\
		If $ac>n_1$, there is a unique $\tilde{t}>0$ such that $\ell_+(\tilde{t},0)=ac$. As in the previous step, we obtain a function $S\colon (\tilde{t},\infty)\to\mathbb{R}$ whose graph is $\mathcal{C}\cap\Omega_0$. \\
		Moreover, in both the cases, $S'(t)>0$.
		\item If $ac\le -n_1$, the set $\mathcal{C}\cap\Omega_0$ is given by $ac=\ell_-(t,s)$. Also, $\ell_-(t,0)=-t^{-1}\sinh n_1t\le -n_1$. Thus, there exists a $\tilde{t}\ge 0$ such that $\ell_-(\tilde{t},0)=ac$. For each $t>\tilde{t}$, there is a unique $s_t>0$ such that $\ell_-(t,s_t)=ac$. Applying the implicit function theorem at each $t>\tilde{t}$ implies the existence of the function $S\colon (\tilde{t},\infty)\to \mathbb{R}$ whose graph is $\mathcal{C}\cap\Omega_0$. Fix $t_0>\tilde{t}$, consider the function $\ell_-(t,S(t_0))$ for $t>\frac{\nu}{n_1}S(t_0)$. The function has a unique maxima (refer Appendix~\ref{ap:proof}) after which it is decreasing. Since $ac<-n_1\le \ell_-\left(\frac{\nu}{n_1}s,s\right)$ for all $s>0$ (by (O4)), we can conclude that $\left(\ell_-\right)_t(t_0,S(t_0))<0$. Hence, $S'(t_0)>0$.
		\item Refer to Figure~\ref{fig:nr:12} for the following discussion. If $-n_1<ac\le 0$, the curve given by $ac=\ell_-(t,s)$ may not be the graph of an increasing function as in the other cases. The curve $\ell_-(t,s)=-n_1$, which is the graph of an increasing function, say $s=S_1(t)$, bounds the curve  $ac=\ell_-(t,s)$ from the right. Also the curve $ac=\ell_+(t,s)$ is the graph of an increasing function, say $S_2(t)$. Thus, the function $S\colon (0,\infty)\to\mathbb{R}$ defined as $S(t)=\min \{S_1(t),S_2(t)\}$ bounds the zero set of $k$ from the right.
	\end{itemize} 
	
	\begin{figure}[h!]
		\centering
		\includegraphics[width=.4\textwidth]{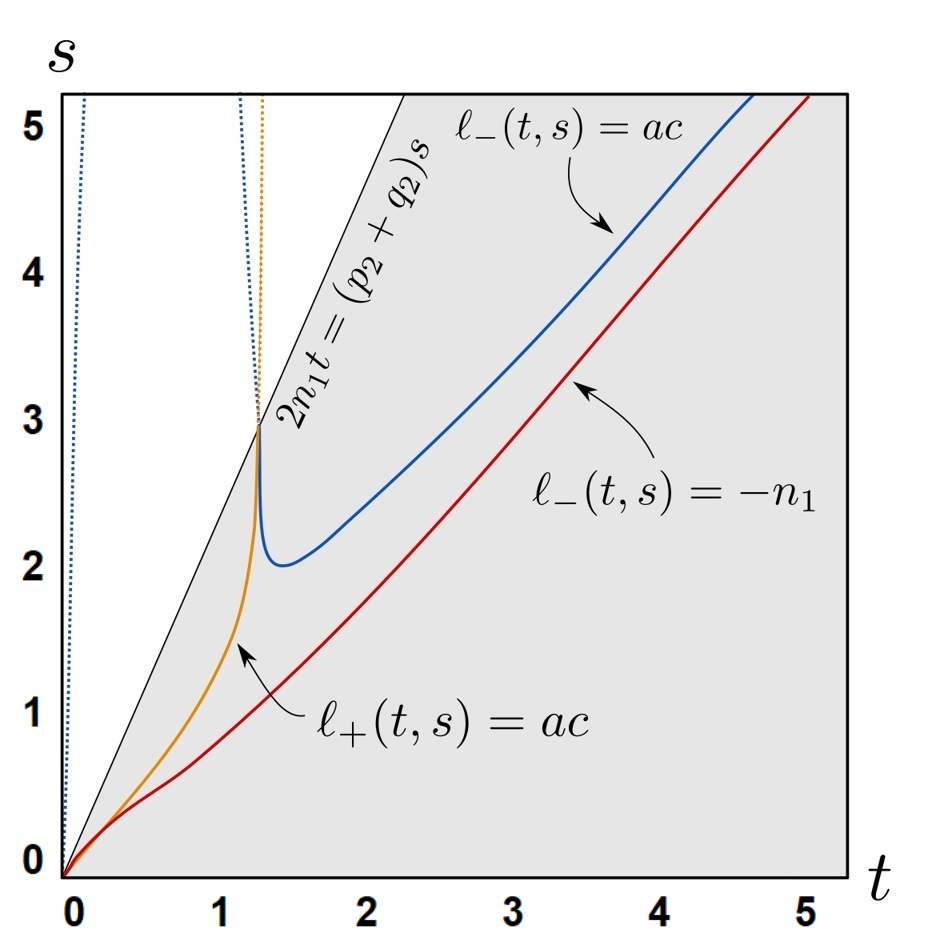}
		\caption{Graphs of $\ell_\pm(t,s)$ in $\Omega_0$; and depiction of the case $-n_1<ac\le 0$. The graph corresponds to the values $n_1=7$, $p_2=0.1$, $q_2=6$ and $ac=-0.8$.}\label{fig:nr:12}
	\end{figure}
	
	In all the three cases, we have bounded $\mathcal{C}$ from the right by the graph $s=S(t)$ of an increasing function $S$. Hence $k(t,s)\le 0$ in the regions $\mathcal{R}_{t_0,S(t_0)}=\{(t,s)\in Q_1\ \colon\ t\ge t_0\text{ and }s\le S(t_0)\}$, for all $t_0>\tilde{t}$ with $k(t,s)=0$ only if $(t,s)=(t_0,S(t_0))$. 
	
	For $\lambda_2=\sqrt[4]{\frac{a^2}{c^2}\, e^{2\mu S(t_0)}}$, $f(\lambda_2,t,S(t_0))=k(t,S(t_0))<0$, for all $t>t_0$. Thus there exists scaling matrix $D_2$ such that $\left\|M(\epsilon_1,\lambda_2) e^{J_1 t\,}M(\epsilon_1,\lambda_2)^{-1}e^{J_2 S(t_0)} \right\|<1$, for all $t>t_0$. Now for any $s<S(t_0)$ and $t>t_0$,
	\begin{eqnarray*}
		\left\|M(\epsilon_1,\lambda_2) e^{J_1 t\,}M(\epsilon_1,\lambda_2)^{-1}e^{J_2 s} \right\|&=&\left\|M(\epsilon_1,\lambda_2) e^{J_1 t\,}M(\epsilon_1,\lambda_2)^{-1}e^{J_2 S(t_0)} \right\|\left\|e^{-J_2 \left(S(t_0)-s\right)} \right\|\\
		&\le & \left\|M(\epsilon_1,\lambda_2) e^{J_1 t\,}M(\epsilon_1,\lambda_2)^{-1}e^{J_2 S(t_0)} \right\|<1.
	\end{eqnarray*}
	Thus $\left\|M(\epsilon_1,\lambda_2) e^{J_1 t\,}M(\epsilon_1,\lambda_2)^{-1}e^{J_2 S(t_0)} \right\|\le 1$, for all $(t,s)\in\mathcal{R}_{t_0,S(t_0)}$, with equality only if $(t,s)=(t_0,S(t_0))$. Hence $f(\lambda_2,t,s)<0$, for all $t>t_0$ and $s<S(t_0)$. 
	
	Now consider $p_2<0<q_2$. Then (as per notation in the statement of the theorem) $\tilde{p}_2=0$ and $\mu=\nu=\dfrac{q_2}{2}$. Here it is not possible to find a region of the form $\Omega(\alpha,\beta)=\{(t,s)\in Q_1\colon\ \alpha t>\beta s\}$ in which $(\ell_-)_t(\ell_-)_s$ is negative. Consider the functions $\ell_\pm$ as functions of $p_1$ as well. To avoid confusion, we will denote these functions as $L_\pm$. We refer to Figure~\ref{fig:nr:12} for the following discussion. Since $(L_+)_{p_2}(p_2,t,s)<0$ and $(L_-)_{p_2}(p_2,t,s)>0$, the surface $L_+(0,t,s)$ lies below $L_+(p_2,t,s)$ and hence the level curve $L_+(p_2,t,s)=ac$ is bounded by the curve $L_+(0,t,s)=ac$ from the right. Similarly $L_-(0,t,s)$ lies above $L_-(p_2,t,s)$ and hence the level curve $L_-(p_2,t,s)=ac$ is bounded by the curve $L_-(0,t,s)=ac$ from the right.
\end{proof}

\section{$A_1$ is real-diagonalizable and $A_2$ is defective}\label{sec:RN}

Suppose $A_1$ is real-diagonalizable with Jordan form $J_1= \begin{pmatrix} -p_1 & 0\\ 0 & -q_1\end{pmatrix}$, where $0<p_1\le q_1$ and $A_2$ is a defective matrix with Jordan form $J_2=\begin{pmatrix}
	n_2 & 1\\0 & n_2
\end{pmatrix}$, where $n_2\ge 0$. For $i=1,2$, fix Jordan basis matrix $P_i$ of $A_i$ with Jordan form $J_i$ such that $M=P_2^{-1}P_1$ has determinant $+1$ or $-1$. Any other Jordan basis matrix of $A_1$ with Jordan form $J_1$ is a nonzero scalar multiple of $P_1D_1$ for some diagonal invertible matrix $D_1=\text{diag}(\lambda_1,1/\lambda_1)$ with $\lambda_1\ne 0$. Any other Jordan basis matrix of $A_2$ with Jordan form $J_2$ is a scalar multiple of $P_2(\epsilon_2)=P_2E_2$, where $E_2=\begin{pmatrix}
	1&\epsilon_2\\ 0&1
\end{pmatrix}$ with $\epsilon_2\in\mathbb{R}$. We will vary $\epsilon_2$ over real numbers and $\lambda_1$ over nonzero real numbers to compute $\tau_{1,2}(\eta)$ and $\tau_{2,1}(\eta)$. Let $M(\epsilon_2,\lambda_1)=P_2(\epsilon_2)^{-1}P_1D_1=\begin{pmatrix}
	(a-c\epsilon_2)\lambda_1 & (b-d\epsilon_2)/\lambda_1 \\ c\lambda_1& d/\lambda_1
\end{pmatrix}$.

\subsection{Computing $\tau_{1,2}(\eta)$}

In this section, we will compute $\tau_{1,2}(\eta)$. Let $\epsilon_2$ be a real number and $\lambda_1$ be a nonzero real number. For given flee time $\eta>0$, we will compute a dwell time $\tau_{1,2}(P_1D_1,P_2(\epsilon_2),\eta)$ such that for all $t>\tau_{1,2}(P_1D_1,P_2(\epsilon_2),\eta)$ and $s<\eta$, $\|M(\epsilon_2,\lambda_1)^{-1}e^{J_2 s}\,M(\epsilon_2,\lambda_1)e^{J_1 t}\|<1$. Note that $M(\epsilon_2,\lambda_1)^{-1}e^{J_2 s}\,M(\epsilon_2,\lambda_1)e^{J_1 t}=M(\lambda_1)^{-1}e^{J_2 s}\,M(\lambda_1)e^{J_1 t}$, where $M(\lambda_1)=MD_1$.

\begin{theorem}\label{rn:12}
	For given flee time $\eta>0$,  \begin{enumerate}[(i)]
		\item if $cd=0$, let $\tau_{1,2}(\eta)=(n_2/p_1)\eta+\epsilon_0$, for any $\epsilon_0>0$. 
		\item if $cd\notin (-n_2,0]$, let $\tau_{1,2}(\eta)$ be the unique positive solution $t$ of \[\lvert cd\rvert \eta \cosh\left(\left(\frac{q_1-p_1}{2}\right)t\right)=- \sinh\left(n_2 \eta-\left(\frac{q_1+p_1}{2}\right)t\right)-\text{sgn}(cd)\sinh\left(\frac{q_1-p_1}{2}\right)t.\]
		\item if $cd\in(-n_2,0)$, let $\tau_{1,2}(\eta)=\max\{\tau_+(\eta),\tau_-(\eta)\}$, where $\tau_+(\eta)$ and $\tau_-(\eta)$ are the unique positive solutions $t$ of \begin{eqnarray*}
			\sinh\left(\frac{q_1-p_1}{2}\right)t&=&\sinh\left(\left(\frac{q_1+p_1}{2}\right)t-n_2 \eta\right),\text{ and}\\
			cd\eta\cosh\left(\left(\frac{q_1-p_1}{2}\right)t\right)&=&-\sinh\left(\frac{q_1-p_1}{2}\right)t-\sinh\left(\left(\frac{q_1+p_1}{2}\right)t-n_2 \eta\right),
		\end{eqnarray*}
		respectively.
	\end{enumerate}
	Then the switched system~\eqref{eq:system} is stable for all $\sigma\in\s\left[\tau_{1,2}(\eta),\eta\right]$ and and asymptotically stable for all $\sigma\in \s'\left[\tau_{1,2}(\eta),\eta\right]$.
\end{theorem}
\begin{proof}
	Note that the equivalent conditions for $\|M(\lambda_1)^{-1}e^{J_2s}M(\lambda_1)e^{J_1t}\|<1$ can be obtained from those in Theorem~\ref{nr:21} by making the following change of variables: $(M,t,s,n_1,p_2,q_2)\rightarrow (M^{-1},s,t,-n_2,-p_1,-q_1)$. Hence the result follows.
\end{proof} 

\subsection{Computing $\tau_{2,1}(\eta)$}

In this section, we will compute $\tau_{2,1}(\eta)$. Let $\epsilon_2$ be a real number and $\lambda_1$ be a nonzero real number. For given flee time $\eta>0$, we will compute a dwell time $\tau_{2,1}(P_1D_1,P_2(\epsilon_2),\eta)$ such that for all $t>\tau_{2,1}(P_1D_1,P_2(\epsilon_2),\eta)$ and $s<\eta$, $\|M(\epsilon_2,\lambda_1)e^{J_1 t}\ M(\epsilon_2,\lambda_1)^{-1}e^{J_2 s}\|<1$. Note that $M(\epsilon_2,\lambda_1)e^{J_1 t}\ M(\epsilon_1,\lambda_2)^{-1}e^{J_2 s}=M(\epsilon_2)^{-1}e^{J_2 s}\,M(\epsilon_2)e^{J_1 t}$, where $M(\epsilon_2)=ME_2$.

\begin{theorem}\label{rn:21}
	For given flee time $\eta>0$, let $\tau_{2,1}(\eta)$ be the unique positive solution $t$  of \[
	n_2\eta+\ln\theta(\eta)=\left(\frac{p_1+q_1}{2}\right)t-\sinh^{-1}\left(\sqrt{K}\sinh\left(\frac{q_1-p_1}{2}\right)t\right),
	\] where $K=\frac{1}{4}\left(r^{-1}+r\right)^2$ with $r=\max\{2\lvert cd\rvert,1\}$. Then the switched system~\eqref{eq:system} is stable for all $\sigma\in\s\left[\tau_{2,1}(\eta),\eta\right]$ and asymptotically stable for all $\sigma\in \s'\left[\tau_{2,1}(\eta),\eta\right]$.
\end{theorem}

\begin{proof}
	Since $\left\|M(\epsilon_2)e^{J_1 t}\,M(\epsilon_2)^{-1}e^{J_2 s} \right\| \le \left\|M(\epsilon_2)e^{J_1 t}\,M(\epsilon_2)^{-1}e^{n_2 s} \right\|\theta(s)$, $\left\|M(\epsilon_2)e^{J_1 t}\,M(\epsilon_2)^{-1}e^{J_2 s} \right\|<1$ if
	\begin{enumerate}
		\item[(C1)] $n_2s+\ln\theta(s)<(p_1+q_1)t/2$, and
		\item[(C2)] $f(t,s)<0$, where
		\[
		f(t,s)=K(\epsilon_2)\sinh^2\left(\frac{q_1-p_1}{2}\right)2-\sinh^2\left(\left(\frac{q_1+p_1}{2}\right)t-n_2s-\ln\theta(s) \right),
		\]
		where $K(\epsilon_2)=((a-\epsilon_2 c)^2+c^2)((b-\epsilon_2 d)^2+d^2)$.
	\end{enumerate}
	
	\noindent The zero set of $f$ is a union of two curves given by $g(s)=h_\pm(t)$, where \[
	g(s)=n_2 s+\ln\theta(s), \ h_\pm(t)=\left(\frac{q_1+p_1}{2}\right)t\pm \sinh^{-1}\left(\sqrt{K(\epsilon_2)}\sinh\left(\frac{q_1-p_1}{2}\right)t\right).
	\]
	
	The curve $g(s)=h_-(t)$ lies in the feasible region (in which condition (C1) is satisfied) and hence both the conditions are satisfied in the region below the curve $g(s)=h_-(t)$ (since $f$ is continuous and $f(t,0)$ is negative for large $t$). 
	
	The function $g$ is increasing. The function $h_-$ is increasing when $K\ge (q_1+p_1)^2/(q_1-p_1)^2$. When $1\le K<(q_1+p_1)^2/(q_1-p_1)^2$, $h_-$ has a unique positive root after which it is increasing. In either of the cases, for given $s>0$, there exists a unique $T(\epsilon_2,s)>0$ such that $g(s)=h_-(T(\epsilon_2,s))$. Since $T(\epsilon_2,s)$ increases with $K(\epsilon_2)$, we have $\min_{\epsilon_2} K(\epsilon_2)=\frac{1}{4}\left(r^{-1}+r\right)^2$, for $r=\max\{2\lvert cd\rvert,1\}$, and hence the result follows.
\end{proof}

\begin{remark}\label{rem:dynRN21}
	Let $\epsilon_2$ be such that $K(\epsilon_2)$ is minimum. Then
	\[
	\epsilon_2=\begin{cases}
		\frac{a}{c}, & c\ne 0, d=0,\\
		\frac{b}{d}, & d\ne 0, c=0,\\
		\frac{ad+bc}{2cd}+\sqrt{\frac{1}{\min\{2|cd|,1\}}-1}, & cd\ne 0.
	\end{cases}
	\] 
	For the function $\tau_{2,1}$ obtained in Theorem~\ref{rn:21}, let $\sigma\in \s_{2,1}$ (defined in~\eqref{eq:signaldynam}).\\
	Then $\|M(\epsilon_2)e^{J_1 t_{k}}  M(\epsilon_2)^{-1}e^{J_2 s_k} \|\le 1$, for all $k\in\mathbb{N}$. Substituting $V_1=P_1, V_2=P_2(\epsilon_2)$ in~\eqref{eq:flow}, the switched system~\eqref{eq:system} is stable for all $\sigma\in\s_{2,1}$ (defined in~\eqref{eq:signaldynam}). 
\end{remark}

\section{Both subsystems with non-real eigenvalues}\label{sec:CC}
In this section, we will consider subsystem matrices $A_1$ and $A_2$ both having non-real eigenvalues with Jordan forms $J_1=\begin{pmatrix} -\alpha_1 & \beta_1\\ -\beta_1 & -\alpha_1\end{pmatrix}$ and $J_2=\begin{pmatrix} \alpha_2 & \beta_2\\ -\beta_2 & \alpha_2\end{pmatrix}$, where $\alpha_1,\alpha_2> 0$ and $\beta_i\ne 0$. Fix Jordan basis matrices $P_1,P_2$ of $A_1,A_2$ with Jordan forms $J_1,J_2$, respectively such that the determinant of $M=P_2^{-1}P_1$ is $1$ or $-1$. For $i=1,2$, any other Jordan basis matrix of $A_i$ with Jordan form $J_i$ is a nonzero scalar multiple of $P_i$. Hence $\tau_{1,2}(\eta)=\tau_{1,2}(P_1,P_2,\eta)$ and $\tau_{2,1}(\eta)=\tau_{2,1}(P_1,P_2,\eta)$.

\subsection{Computing $\tau_{1,2}(\eta)$}
For given flee time $\eta>0$, we will compute a dwell time $\tau_{1,2}(\eta)$ such that for all $t>\tau_{1,2}(\eta)$ and $s<\eta$, $\|M^{-1}e^{J_2 s}\,Me^{J_1 t}\|<1$. 
\begin{theorem}\label{cc:12}
	For given flee time $\eta>0$, let
	\[
	\tau_{1,2}(\eta)=\frac{\alpha_2}{\alpha_1} \eta+\frac{1}{\alpha_1}\cosh^{-1}\left(\frac{a^2+b^2+c^2+d^2}{2}\right).
	\]
	Then the switched system~\eqref{eq:system} is stable for all $\sigma\in	\s\left[\tau_{1,2}(\eta),\eta\right]$ and asymptotically stable for all $\sigma\in \s'\left[\tau_{1,2}(\eta),\eta\right]$.
\end{theorem}

\begin{figure}[h!]
	\centering
	\includegraphics[width=.35\textwidth]{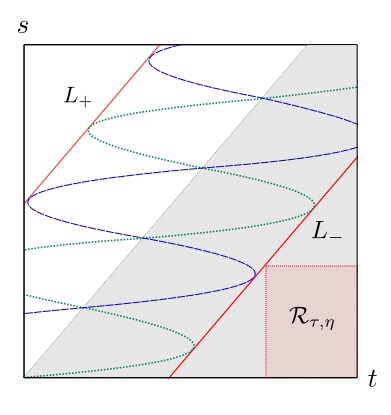}
	\caption{Zero set in the feasible region is bounded below by $L_-$. The dotted curves are $\alpha_1t=\alpha_2 s\pm \sinh^{-1}(\sqrt{K}\sin\beta_2 s)$, depending on sign of $\beta_2$.}
	\label{fig:CC2}
\end{figure}

\begin{proof}
	Note that $\|M^{-1}e^{J_2 s}\,M e^{J_1 t}\|<1$ if and only if  
	\begin{enumerate}
		\item[(C1)] $\alpha_2 s-\alpha_1 t<0$, and
		\item[(C2)] $f(t,s)=K \sin^2 \beta_2 s-\sinh^2(\alpha_2 s-\alpha_1 t)<0$, where $K=\frac{\left(a^2+b^2+c^2+d^2\right)^2}{4}-1$.
	\end{enumerate}
	
	The condition (C1) can be rewritten as $s<(\alpha_1/\alpha_2)t$. The zero set of $f$ lies between the lines $L_\pm\colon\,s=(\alpha_1/\alpha_2)t\pm (1/\alpha_2)\sinh^{-1}\sqrt{K}$. Refer to Figure~\ref{fig:CC2} for the following discussion. Consider the region $\mathcal{R}_{\tau,\eta}=\{(t,s)\in Q_1\colon\, t>\tau\text{ and }s<\eta\}$, where $\tau$ and $\eta$ satisfy $\eta=(\alpha_1/\alpha_2)\tau-(1/\alpha_2)\sinh^{-1}\sqrt{K}$. This region lies below the line $L_-$ and hence (C1) is satisfied in this set. Also (C2) is satisfied since $f$ is continuous, has zero set lying above the line $L_-$, and since $f(t,0)<0$ for sufficiently large $t$. Hence the result follows.
\end{proof}

\subsection{Computing $\tau_{2,1}(\eta)$}
For given flee time $\eta>0$, we will compute a dwell time $\tau_{2,1}(\eta)$ such that for all $t>\tau_{2,1}(\eta)$ and $s<\eta$, $\|Me^{J_1 t}\,M^{-1}e^{J_2 s}\|<1$. 

\begin{theorem}\label{cc:21}
	For given flee time $\eta>0$, let \[\tau_{2,1}(\eta)=\frac{\alpha_2}{\alpha_1} \eta+\frac{1}{\alpha_1}\cosh^{-1}\left(\frac{a^2+b^2+c^2+d^2}{2}\right).
	\]
	Then the switched system~\eqref{eq:system} is stable for all $\sigma\in	\s\left[\tau_{2,1}(\eta),\eta\right]$ and asymptotically stable for all $\sigma\in \s'\left[\tau_{2,1}(\eta),\eta\right]$.
\end{theorem}

\begin{proof}
	$\|Me^{J_1 t}\,M^{-1}e^{J_2 s}\|<1$ if and only if 
	\begin{enumerate}
		\item[(C1)] $\alpha_2 s-\alpha_1 t<0$, and
		\item[(C2)] the function $f(t,s)=K \sin^2 \beta_1 t-\sinh^2(\alpha_2 s-\alpha_1 t)<0$, where $K=\frac{\left(a^2+b^2+c^2+d^2\right)^2}{4}-1$.
	\end{enumerate}
	These conditions are obtained from those in Theorem~\ref{cc:12} by the following change of variables $(M,\alpha_1,\beta_1,t)\leftrightarrow (M^{-1},-\alpha_2,\beta_2,s)$. The proof follows on observing that the zero set is bounded between the lines $\alpha_1 t=\alpha_2 s\pm\sinh^{-1}\sqrt{K}$.
\end{proof}

\begin{remark}\label{rem:dynCC}
	For the functions $\tau_{1,2}$ and $\tau_{2,1}$ obtained in Theorem~\ref{cc:12} and~\ref{cc:21}, let $\sigma\in \s_{1,2}$ (defined in~\eqref{eq:signaldynam}). Then $\|M^{-1}e^{J_2 s_k} Me^{J_1 t_{k}} \|\le 1$, for all $k\in\mathbb{N}$. Substituting $V_1=P_1, V_2=P_2$ in~\eqref{eq:flow}, the switched system~\eqref{eq:system} is stable for all $\sigma\in\s_{1,2}$. Similarly, the switched system~\eqref{eq:system} is stable for all $\sigma\in\s_{2,1}$ (defined in~\eqref{eq:signaldynam}). 
\end{remark}

\section{Extension to more general classes of switched systems}

In this section, we will present two possible extensions of our results for bimodal planar switched systems to a special class of symmetric bilinear systems and multimodal systems governed by an undirected star graph.

\subsection{Symmetric bilinear systems}\label{sec:sbs}
Let $\mathcal{U}$ denotes the collection of bounded piecewise constant functions $u:[0,\infty)\rightarrow\mathbb{R}$ with $\liminf_{t\rightarrow \infty}u(t)>0$. Consider the following symmetric bilinear system
\begin{eqnarray}\label{eq:system-sbs}
	\dot{x}(t)&=&u(t)A_{\sigma(t)}x(t), 
\end{eqnarray}
where $x(t)\in\mathbb{R}^2$, for all $t\ge 0$, the switching signal $\sigma:[0,\infty)\rightarrow \{1,2\}$ is a piecewise constant right continuous function, and $u\in\mathcal{U}$. We will use notations and make assumptions for $\sigma$ like in the earlier sections.

So far we have focussed on system~\eqref{eq:system} where $u\equiv 1$. For given $\eta>0$, we computed $\tau(\eta)$ such that the switched system~\eqref{eq:system} is stable for all signals $\sigma\in \mathcal{S}[\tau(\eta),\eta]$ and asymptotically stable for all signals $\sigma\in \mathcal{S}'[\tau(\eta),\eta]$. We will use these results to obtain dwell-flee relations for stability of system~\eqref{eq:system-sbs}. \\~\\
The flow of system~\eqref{eq:system-sbs} is given by
\begin{eqnarray*}
	x(t)&=&e^{A_{\sigma_j}I(u,d_{j-1},t)}\left(\prod_{k=1}^{j-1}e^{A_{\sigma_k}I(u,d_{k-1},d_k)}\right)x(0),\ \ \ t\in[d_{j-1},d_j),
\end{eqnarray*} 
where $I(u,d_{k-1},d_k)=\int_{d_{k-1}}^{d_k} u(r)\ dr$. Assume that $A_1$ is Hurwitz stable and $A_2$ is unstable. For $i=1,2$, let $P_i$ be a Jordan basis matrix of the subsystem matrix $A_i$ with real Jordan basis matrix $J_i$. Let $M=P_2^{-1}P_1$. 

\begin{theorem}\label{thm:sbs}
	Let $u\in \mathcal{U}$ with $\underline{u}=\liminf_{t\rightarrow\infty} u(t)$ and $\overline{u}=\limsup_{t\rightarrow\infty} u(t)$. 
	\begin{enumerate}
		\item[(a)] For each $\eta>0$, let $\tau(\eta,P_1,P_2)>0$ be such that either $\|M^{-1}e^{J_2 s}Me^{J_1 t}\|\le 1$ or $\|Me^{J_1 t}M^{-1}e^{J_2 s}\|\le 1$, for all $t> \tau(\eta,P_1,P_2)$ and $s<\eta$. Then for each $\eta>0$, the switched system~\eqref{eq:system-sbs} is stable for all signals $\sigma\in\s[\tau(\overline{u}\eta,P_1,P_2)/\underline{u},\eta]$. 
		\item[(b)] For each $\eta>0$, let $\tau(\eta,P_1,P_2)>0$ and $\rho<1$ be such that either $\|M^{-1}e^{J_2 s}Me^{J_1 t}\|\le \rho$ or $\|Me^{J_1 t}M^{-1}e^{J_2 s}\|\le \rho$, for all $t> \tau(\eta,P_1,P_2)$ and $s<\eta$. Then for each $\eta>0$, the switched system~\eqref{eq:system-sbs} is asymptotically stable for all signals $\sigma\in\s[\tau(\overline{u}\eta,P_1,P_2)/\underline{u},\eta]$. 
	\end{enumerate}
\end{theorem}

\begin{proof}
	Note that $0<\underline{u}\le \overline{u}<\infty$ since $u\in\mathcal{U}$. Let $\sigma\in\s[\tau(\overline{u}\eta,P_1,P_2)/\underline{u},\eta]$. Let $T>0$ be such that $\underline{u}\le u(t)\le \overline{u}$, for all $t\ge T$. Let $j\ge 1$ be the least number such that $d_{2j}\ge T$. Thus for $i\ge j$, $I(u,d_{2i},d_{2i+1})\ge \underline{u} (d_{2i+1}-d_{2i})$, and $I(u,d_{2i+1},d_{2i+2})\le \overline{u} (d_{2i+2}-d_{2i+1})$. For $t\in[d_{2k},d_{2k+2})$, by a certain grouping of terms,
	\begin{eqnarray*}
		\|x(t)\|&\le& \zeta\ \xi\ \min\left\lbrace\prod_{i=j}^{k-1} \left\|M^{-1}e^{J_2 I(u,d_{2i+1},d_{2i+2})}\,M e^{J_1 I(u,d_{2i},d_{2i+1})} \right\|, \right. \\
		& & \ \ \ \ \ \ \ \left. \prod_{i=j}^{k-1} \left\|Me^{J_1 I(u,d_{2i+2},d_{2i+3})}\,M^{-1}e^{J_2 I(u,d_{2i+1},d_{2i+2})} \right\| \right\rbrace \|x(d_{2j})\|,
	\end{eqnarray*}
	where $\zeta=\max\{\|P_1\|\|P_1^{-1}P_2\|\|P_2^{-1}P_1\|\|P_1^{-1}\|,\|P_1\|\|P_1^{-1}P_2\|\|P_2^{-1}\|\}$, \\
	$\xi=\left(\sup_{t\in [0,\infty)}\left\|{\rm e}^{J_1 t}\right\|\right)^2\,\sup_{t\in [0,\eta]}\left\|{\rm e}^{J_2 t}\right\|$. Note that $\xi$ is finite since $J_1$ is a stable matrix. Since $\sigma\in \s[\tau(\overline{u}\eta,P_1,P_2),\eta]$, $d_{2i+1}-d_{2i}\ge \tau(\overline{u}\eta,P_1,P_2)/\underline{u}$ and $d_{2i+2}-d_{2i+1}\le \eta$, then for all $i\ge j$, 
	\begin{enumerate}
		\item[(a)] if the hypothesis of part (a) in the statement of the Theorem holds true then either \begin{eqnarray*}
			\left\|M^{-1}e^{J_2 I(u,d_{2i+1},d_{2i+2})}\,M e^{J_1 I(u,d_{2i},d_{2i+1})} \right\| &\le& 1, \ \text{or } \\
			\left\|Me^{J_1 I(u,d_{2i+2},d_{2i+3})}\,M^{-1}e^{J_2 I(u,d_{2i+1},d_{2i+2})} \right\|&\le& 1.
		\end{eqnarray*}
		\item[(b)] if the hypothesis of part (b) in the statement of the Theorem holds true then either
		\begin{eqnarray*}
			\left\|M^{-1}e^{J_2 I(u,d_{2i+1},d_{2i+2})}\,M e^{J_1 I(u,d_{2i},d_{2i+1})} \right\| &\le& \rho, \ \text{or } \\
			\left\|Me^{J_1 I(u,d_{2i+2},d_{2i+3})}\,M^{-1}e^{J_2 I(u,d_{2i+1},d_{2i+2})} \right\|&\le& \rho.
		\end{eqnarray*}		
	\end{enumerate}
	The result follows.
\end{proof}

\begin{remark}
	The results for~\eqref{eq:system-sbs} can be extended to the switched systems of the form
	\begin{eqnarray*}
		\dot{x}(t)&=&u_{\sigma(t)}(t)A_{\sigma(t)}x(t),
	\end{eqnarray*}
	where $x(t)\in\mathbb{R}^2$, for all $t\ge 0$, the switching signal $\sigma:[0,\infty)\rightarrow \{1,2\}$ is a piecewise constant right continuous function, and $u_1,u_2\in\mathcal{U}$.
\end{remark}

\subsection{Multimodal systems}\label{sec:multimodal}
In the section, we will obtain dwell-flee relation for a multimodal switched system having only one Hurwitz stable subsystem. Moreover the switching between subsystems is governed by the star graph as shown in Figure~\ref{fig:graph}.

\begin{figure}[h!]
	\centering
	\includegraphics[width=.4\textwidth]{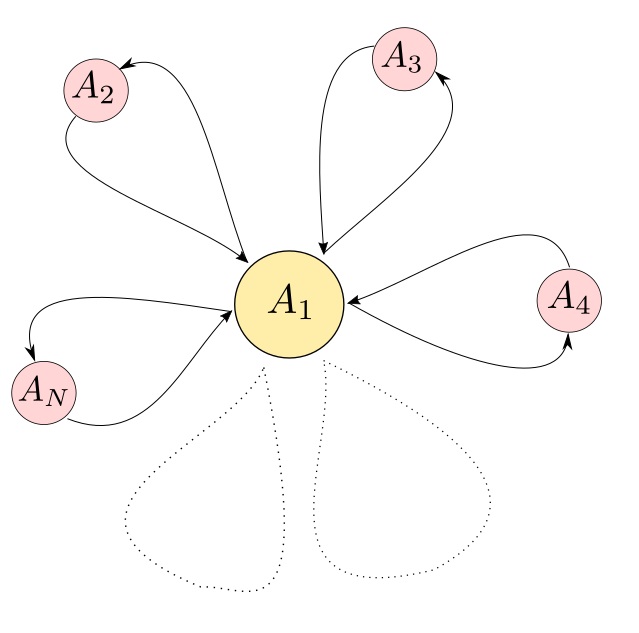}
	\caption{Underlying star graph.}
	\label{fig:graph}
\end{figure}

Let $A_1,\dots,A_N$ be planar matrices with Jordan decomposition $A_k=P_kJ_kP_k^{-1}$, for each $k=1,\dots,N$. Assume that $A_1$ is Hurwitz stable and all the other matrices are unstable. Let $\mathcal{G}$ be the star graph with vertices labelled $A_1,\dots,A_N$, as shown in Figure~\ref{fig:graph}. Consider a multimodal planar linear switched system given by
\begin{eqnarray}\label{eq:multisystem}
	\dot{x}(t) &=& A_{\sigma(t)}x(t),
\end{eqnarray}
where the switching signal $\sigma$ is a piecewise constant right-continuous function, and for $t\ge 0$, $x(t)\in\mathbb{R}^2$ and $\sigma(t)\in\{1,\dots,N\}$. Let us define the following notations and make the following assumptions:
\begin{itemize}
	\item $\sigma(0)=1$. The set of discontinuities of $\sigma$ is given by $(d_k)_{k\ge 1}$ where $d_k\uparrow\infty$, $d_0=0$. 
	\item $\sigma$ is governed by the graph $\mathcal{G}$. That is, for each $k\ge 0$, there is an edge from the vertex $A_{\sigma(d_{2k})}$ to the vertex $A_{\sigma(d_{2k+1})}$. Since $\sigma(0)=1$, $\sigma(d_{2k})=1$, $\sigma(d_{2k-1})\in \{2,\dots,N\}$, for all $k\ge 1$.
	\item Let $\Delta_{2k-1}=t_k$ (time spent in subsystem $A_1$) and $\Delta_{2k}=s_k$, for $k\ge 1$ (time spent in subsystems other than $A_1$). 
\end{itemize}
We wish to study the stability of system~\eqref{eq:multisystem}. We now define the following classes of signals:
\begin{eqnarray*}
	\s_N\left[\tau,\eta\right]&=&\left\{\sigma\colon[0,\infty)\to \{1,\dots,N\}\ \vert\ \sigma(0)=1,\  t_k\ge \tau,\ s_k\le \eta, \ \text{for all } k\in\mathbb{N}\right\},\nonumber\\ 
	\s_N'\left[\tau,\eta\right]&=&\{\, 
	\sigma\in \s\left[\tau,\eta\right]\ \vert \ \text{the sequence } (t_k,s_k)_{k\ge 1} \text{ of tuples does not converge }\nonumber \\
	&& \text{ \ to the tuple } (\tau,\eta) \}. 
\end{eqnarray*}
Then for $t\in[d_{2j},d_{2j+2})$, $t_i\ge \tau$, and $s_i\le \eta$, for all $i=1,\dots,j$, the flow of~\eqref{eq:multisystem} satisfies
\begin{equation}\label{eq:solmm}
	\|x(t)\|\le \zeta\ \xi\ \ \left(\prod_{i=1}^j \left\|M_{\sigma(d_{2i-1})}^{-1}e^{J_{\sigma(d_{2i-1})} s_i}\,M_{\sigma(d_{2i-1})} e^{J_1 t_i} \right\| \right)  \|x(0)\|,
\end{equation}
where $M_k=P_{k}^{-1}P_1$, $\xi=\max\left\{\sup_{t\in [0,\infty)}\left\|{\rm e}^{J_1 t}\right\|\sup_{t\in [0,\eta]}\left\|{\rm e}^{J_k t}\right\|\ \colon\ k=2,\dots,N\right\}$, \\
$\zeta=\max\left\{\|P_1\|\|P_1^{-1}\|,\|P_k\|\|M_k\|\|P_1^{-1}\| \ \colon\ k=2,\dots,N \right\}$. Note that $\xi$ is finite since $J_1$ is a stable matrix. \\
In the following results, for given flee time $\eta>0$, we will find $\tau(\eta)>0$ such that for all $t>\tau(\eta)$ and $s<\eta$, $\left\|M_k^{-1}e^{J_k s}\,M_k e^{J_1 t} \right\|<1$, for each $k=2,\dots,N$. Then~\eqref{eq:solmm} will imply stability of the multimodal system~\eqref{eq:multisystem} for all signals $\sigma\in\s_N[\tau(\eta),\eta]$. In the following three results, we obtain dwell-flee relations according to the Jordan form of the Hurwitz stable matrix $A_1$. 

\begin{theorem}\label{star:realstable}
	Let $A_1$ has real eigenvalues with Jordan form $J_1=\text{diag}(-p_1,-q_1)$, where $p_1,q_1>0$. For $2\le j\le k+1$, let $A_j$ has non-real eigenvalues with Jordan form $J_j=\begin{pmatrix}
		\alpha_j & \beta_j\\
		-\beta_j & \alpha_j
	\end{pmatrix}$, where $\alpha_j\ge 0$, $\beta_j\ne 0$; for $k+1< j\le k+\ell+1$, let $A_j$ has non-real eigenvalues with Jordan form $J_j=\begin{pmatrix}
		n_j & 1\\
		0 & n_j
	\end{pmatrix}$, where $n_j\ge 0$; and for $k+\ell+1< j\le N$, let $A_j$ has non-real eigenvalues with Jordan form $J_j=\text{diag}(p_j,q_j)$, where $p_j\le q_j$ with $(p_j,q_j)\ne (0,0)$. For each $j=2,\dots,N$, let $P_j^{-1}P_1 = \begin{pmatrix}
		a_j & b_j\\
		c_j & d_j
	\end{pmatrix}$. 	
	\begin{enumerate}[(a)]
		\item For each $2\le j\le k+1$, \[
		\tau_{1,2}^{(j)}(\lambda,\eta)=\frac{\alpha_j}{p_1}\mathcal{\eta}+\frac{1}{p_1}\cosh^{-1}\left[\frac{1}{2}\left((a_j^2+c_j^2)\lambda^2+\frac{b_j^2+d_j^2}{\lambda^2}\right)\right],
		\]			
		\item for each $k+1< j\le k+\ell+1$, \[
		\tau_{1,2}^{(j)}(\lambda,\eta)=\frac{n_j}{p_1}\mathcal{\eta}+\frac{1}{p_1}\sinh^{-1}\left[\frac{1}{2}\left(c_j^2 \lambda^2+\frac{d_j^2}{\lambda^2}\right)\right], \text{ and}
		\]			
		\item for each $k+\ell+1< j\le N$, $\tau_{1,2}^{(j)}(\lambda,\eta)$ is the unique solution $\tau\ge (q_j/p_1)\eta$ satisfying \[
		\frac{\sinh(p_1\tau-q_j\eta)\,\sinh(p_1\tau-p_j\eta)}{\sinh^2\left(\frac{q_j-p_j}{2}\right)\eta}=\left(\frac{b_jd_j}{\lambda^2}+a_jc_j\lambda^2\right)^2.
		\]
	\end{enumerate}
	Then for given $\epsilon>0$, let $\tau(\eta)$ equals
	\[\begin{cases}
		\inf_{\lambda>0}\max \left\{\tau_{1,2}^{(2)}(\lambda,\eta),\cdots,\tau_{1,2}^{(N)}(\lambda,\eta)\right\}, & \text{if the infimum over $\lambda$ is attained}\\
		\inf_{\lambda>0}\max \left\{\tau_{1,2}^{(2)}(\lambda,\eta),\cdots,\tau_{1,2}^{(N)}(\lambda,\eta)\right\}+\epsilon, & \text{otherwise}.
	\end{cases}
	\]
	Then the switched system~\eqref{eq:multisystem} is stable for all $\sigma\in \s_N[\tau(\eta),\eta]$ and asymptotically stable for all $\sigma\in\s_N'[\tau(\eta),\eta]$. 	
\end{theorem}

\begin{proof}
	Since the eigenvalues of $A_1$ are $-p_1<0,-q_1<0$, its Jordan form $J_1=\text{diag}(-p_1,-q_1)$. Let $D_1=\text{diag}(\lambda,1/\lambda)$ be the scaling matrix with $\lambda>0$. We will use the fact that the curve $\|D_1^{-1}P_1^{-1}P_{j}e^{J_{j}s}P_{j}^{-1}P_1D_1e^{-p_1 t}\|=1$ bounds the set $\|D_1^{-1}P_1^{-1}P_{j}e^{J_{j}s}P_{j}^{-1}P_1D_1e^{J_{1}t}\|=1$ from the right in the $ts$-plane. Substituting $p_1=q_1$ in the relevant results in Section~\ref{sec:RR12}, Theorems~\ref{rc:12}, and~\ref{rn:12}, and repeating the steps therein, we get the expressions for $\tau_{1,2}^{(j)}(\lambda,\eta)$ as given in the statement of the theorem and hence the result follows.
\end{proof}

\begin{theorem}\label{star:defectivestable}
	Let $A_1$ be a defective matrix with Jordan form $J_1=\begin{pmatrix}
		-n_1 & 1\\0& -n_1
	\end{pmatrix}$, where $n_1>0$. Let $A_2,\dots,A_{k+1}$ have a pair of non-real eigenvalues; $A_{k+2},\dots,A_{k+\ell+1}$ be defective; and $A_{k+\ell+2},\dots,A_N$ be real-diagonalizable. Let $P_j^{-1}P_1$ and $J_j$ be as defined in Theorem~\ref{star:realstable}. If $\tau(\eta)=\min_{\epsilon}\max \left\{\tau_{1,2}^{(2)}(\epsilon,\eta),\cdots,\tau_{1,2}^{(N)}(\epsilon,\eta)\right\}$, where \begin{enumerate}[(a)]
		\item for each $2\le j\le k+1$, $\tau_{1,2}^{(j)}(\epsilon,\eta)$ is the unique positive solution $\tau$ of
		\[
		n_1\tau-\ln\theta(\tau)=\alpha_j\eta+\cosh^{-1}\left(\frac{a_j^2+(b_j+\epsilon a_j)^2+c_j^2+(d_j+\epsilon c_j)^2}{2}\right).
		\]			
		\item for each $k+1< j\le k+\ell+1$, $\tau_{1,2}^{(j)}(\epsilon,\eta)$ is the unique positive solution $\tau$ of\[
		n_1\tau-\ln\theta(\tau)=n_j\eta+\sinh^{-1}\left(\left(\frac{c_j^2+(d_j+\epsilon c_j)^2}{2}\right)\eta\right).
		\]			
		\item for each $k+\ell+1< j\le N$, $\tau_{1,2}^{(j)}(\epsilon,\eta)$ is the unique positive solution $\tau$ of \[
		n_1\tau-\ln\theta(\tau)=\left(\frac{p_j+q_j}{2}\right)\eta+\sinh^{-1}\left(\sqrt{K_j(\epsilon)}\sinh\left(\frac{q_j-p_j}{2}\right)\eta\right),
		\]
	\end{enumerate}
	where $K_j(\epsilon)=(a_j^2+(b_j+\epsilon a_j)^2)(c_j^2+(d_j+\epsilon c_j)^2)$. Then the switched system~\eqref{eq:multisystem} is stable for all $\sigma\in\s_N[\tau(\eta),\eta]$ and is asymptotically stable for all $\sigma\in\s_N'[\tau(\eta),\eta]$.
\end{theorem}

\begin{proof}
	The expressions for $\tau_{1,2}^{(j)}(\epsilon,\eta)$ are obtained by Theorems~\ref{nn:12},~\ref{nc:12},~\ref{nr:12}, and the result follows. Using standard continuity arguments, the existence of $\tau(\eta)$ (given as the solution of a minimax problem) can be shown.
\end{proof}

\begin{theorem}\label{star:non-realstable}
	Let $A_1$ has non-real eigenvalues with Jordan form $J_1=\begin{pmatrix}
		-\alpha_1 &\beta_1\\ -\beta_1 & -\alpha_1
	\end{pmatrix}$, where $\alpha_1>0$ and $\beta_1\ne 0$. Let $A_2,\dots,A_{k+1}$ have a pair of non-real eigenvalues, $A_{k+2},\dots,A_{k+\ell+1}$ be defective, and $A_{k+\ell+2},\dots,A_N$ be real-diagonalizable. Let $P_j^{-1}P_1$ and $J_j$ be as defined in Theorem~\ref{star:realstable}. Let $\tau(\eta)=\max_{j\ge 2}\,\left\{\tau_{1,2}^{(j)}(\eta)\right\}$, where \begin{enumerate}[(a)]
		\item for each $2\le j\le k+1$, $\tau_{1,2}^{(j)}(\eta)$ is the unique positive solution $\tau$ of
		\[
		\tau_{1,2}(\eta)=\frac{\alpha_j}{\alpha_1} \eta+\frac{1}{\alpha_1}\cosh^{-1}\left(\frac{a_j^2+b_j^2+c_j^2+d_j^2}{2}\right).
		\]		
		\item for each $k+1< j\le k+\ell+1$, $\tau_{1,2}^{(j)}(\eta)$ is the unique positive solution $\tau$ of\[
		\tau_{1,2}(\eta)=\frac{n_j}{\alpha_1}\eta+\frac{1}{\alpha_1}\sinh^{-1}\left(\left(\frac{c_j^2+d_j^2}{2}\right)\eta\right).
		\]
		
		\item for each $k+\ell+1< j\le N$, $\tau_{1,2}^{(j)}(\eta)$ is the unique positive solution $\tau$ of \[
		\tau_{1,2}(\eta)=\left(\frac{q_j+p_j}{2\alpha_1}\right)\eta+\frac{1}{\alpha_1} \sinh^{-1}\left(\sqrt{(a_j^2+b_j^2)(c_j^2+d_j^2)}\sinh\left(\frac{q_j-p_j}{2}\right)\eta\right).
		\]	
	\end{enumerate}
	Then the switched system~\eqref{eq:multisystem} is stable for all $\sigma\in\s_N[\tau(\eta),\eta]$ and is asymptotically stable for all $\sigma\in\s_N'[\tau(\eta),\eta]$.
\end{theorem}

\begin{proof}
	The proof follows from Theorems~\ref{cr:12},~\ref{cn:12}, and~\ref{cc:12}.
\end{proof}
\begin{remark}
	If $A_1$ is an unstable matrix and all the other subsystem matrices $A_2,\dots,A_N$ are Hurwitz in Theorems~\ref{star:realstable},~\ref{star:defectivestable} and~\ref{star:non-realstable}, then we can use corresponding results for the bimodal system~\eqref{eq:system} for computing $\tau_{2,1}$ instead of $\tau_{1,2}$ and obtain dwell-flee relations as in Theorems~\ref{star:realstable},~\ref{star:defectivestable}, and~\ref{star:non-realstable}. 
\end{remark}
\section{Examples}\label{sec:examples}
In Remark~\ref{rem:examplesRR} and Examples~\ref{example:rc}, and~\ref{example:nr}, we will compare the dwell-flee relations obtained in this paper with the dwell-flee relations given in~\cite{agarwal2018simple} with which direct comparisons are possible. For given flee time $\eta>0$, we will denote the dwell time obtained in~\cite{agarwal2018simple} as $\tau_A(\eta)$. 

\begin{remark}\label{rem:examplesRR}
	Consider a switched system~\eqref{eq:system} with both subsystems real-diagonalizable. Refer to Theorems~\ref{rr:12zero} and~\ref{rr:21zero} for dwell-flee relations.
	
	Let $J_1=\text{diag}(-p_1,-q_1)$ and $J_2=\text{diag}(p_2,q_2)$ with $0<p_1\le q_1$, $q_2>0$, $p_2\le q_2$. Let $V_1=\begin{bmatrix} v_1 & w_1\end{bmatrix}$ and $V_2=\begin{bmatrix} v_2 & w_2	\end{bmatrix}$ be Jordan basis matrices, \textit{having unit norm columns}, of $A_1$ and $A_2$ with Jordan forms $J_1$ and $J_2$, respectively. 
	\begin{itemize}
		\item If $v_1=v_2$, or if $w_1=w_2$, then one of the off-diagonal entries of $V_2^{-1}V_1$ is $0$. In which case,
		\[
		\tau_A(\eta)=\frac{q_2}{p_1}\eta+\frac{\ln\left(\|V_2^{-1}V_1\|\, \|V_1^{-1}V_2\|\right)}{p_1}>\max\left(\frac{q_2}{q_1},\frac{p_2}{p_1}\right)\eta+\epsilon_0=\tau(\eta),
		\]
		for any $0<\epsilon_0< p_1^{-1}\ln\left(\|V_2^{-1}V_1\|\, \|V_1^{-1}V_2\|\right)$. Note that $\lim_{\eta\to\infty}\tau_A(\eta)-\tau(\eta)\to \infty$.		
		\item If $v_1=w_2$, or if $v_2=w_1$, then one of the diagonal entries of $V_2^{-1}V_1$ is $0$. Then
		\[
		\tau_A(\eta)=\frac{q_2}{p_1}\eta+\frac{\ln\left(\|V_2^{-1}V_1\|\, \|V_1^{-1}V_2\|\right)}{p_1}>\frac{q_2}{p_1}\eta+\epsilon_0=\tau(\eta),
		\]
		for any $0<\epsilon_0< p_1^{-1}\ln\left(\|V_2^{-1}V_1\|\, \|V_1^{-1}V_2\|\right)$. Note that $\tau_A(\eta)-\tau(\eta)$ is bounded.
	\end{itemize}
\end{remark}

\begin{example}\label{example:rc}
	Consider a Hurwitz stable (real-diagonalizable) matrix $A_1=\begin{pmatrix}	-0.1& 0\\ 0.4 & -0.2\end{pmatrix}$ and an unstable matrix $A_2=\begin{pmatrix}0 & 1\\ -2& 1\end{pmatrix}$ with non-real eigenvalues. 
	
	For computing dwell-flee relation given in~\cite{agarwal2018simple}, take the Jordan basis matrix $V_1=\begin{pmatrix}
		1/\sqrt{17} & 0\\ 4/\sqrt{17}&1
	\end{pmatrix}$ of $A_1$ having unit norm columns with Jordan form $J_1=\text{diag}\left(-0.1,-0.2\right)$ and the Jordan basis matrix $V_2=\sqrt{\frac{2}{3}}\begin{pmatrix}\frac{1}{4}+\mathrm{i}\frac{\sqrt{7}}{4}&\frac{1}{4}-\mathrm{i}\frac{\sqrt{7}}{4}\\1 &1\end{pmatrix}$ of $A_2$ having unit norm columns with complex Jordan form $\text{diag}\left(\frac{1}{2}-\mathrm{i}\frac{\sqrt{7}}{2},\,\frac{1}{2}+\mathrm{i}\frac{\sqrt{7}}{2}\right)$. Then  $\tau_A(\eta)=5\eta+17.05$. 
	
	Now take a Jordan basis matrix $P_1$ of $A_1$ with Jordan form $J_1$ and a Jordan basis matrix $P_2$ of $A_2$ with real Jordan form $J_2=\begin{pmatrix}1/2 & \sqrt{7}/2\\-\sqrt{7}/2& 1/2 \end{pmatrix}$ such that $M=P_2^{-1}P_1$ has determinant 1. Then $M=\dfrac{1}{\sqrt[4]{7}}\begin{pmatrix}
		\sqrt[2]{7}&\sqrt[2]{7}\\0&1
	\end{pmatrix}$. The dwell-flee relation $\tau_{1,2}(\eta)$ is given by $\eta=0.3\,\tau_{1,2}(\eta)-2\sinh^{-1}\left(\sqrt{7}\cosh\,0.05\,\tau_{1,2}(\eta)+\sinh\, 0.05\,\tau_{1,2}(\eta)\right)$ (Theorem~\ref{rc:12}) and $\tau_{2,1}(\eta)$ is given by $\eta=0.3\, \tau_{2,1}(\eta)-2\sinh^{-1}\left(\sqrt{8}\, \sinh\, 0.05\ \tau_{2,1}(\eta)\right)$ (Theorem~\ref{rc:21}). By Remark~\ref{rem:rcbetter}, $\tau(\eta)=\min\{\tau_{1,2}(\eta),\tau_{2,1}(\eta)\}=\tau_{2,1}(\eta)$. The graphs of $\tau$ and $\tau_A$ are shown in Figure~\ref{fig:rcexample} (left). 
	
	\begin{figure}[h!]
		\centering
		\includegraphics[scale=0.4]{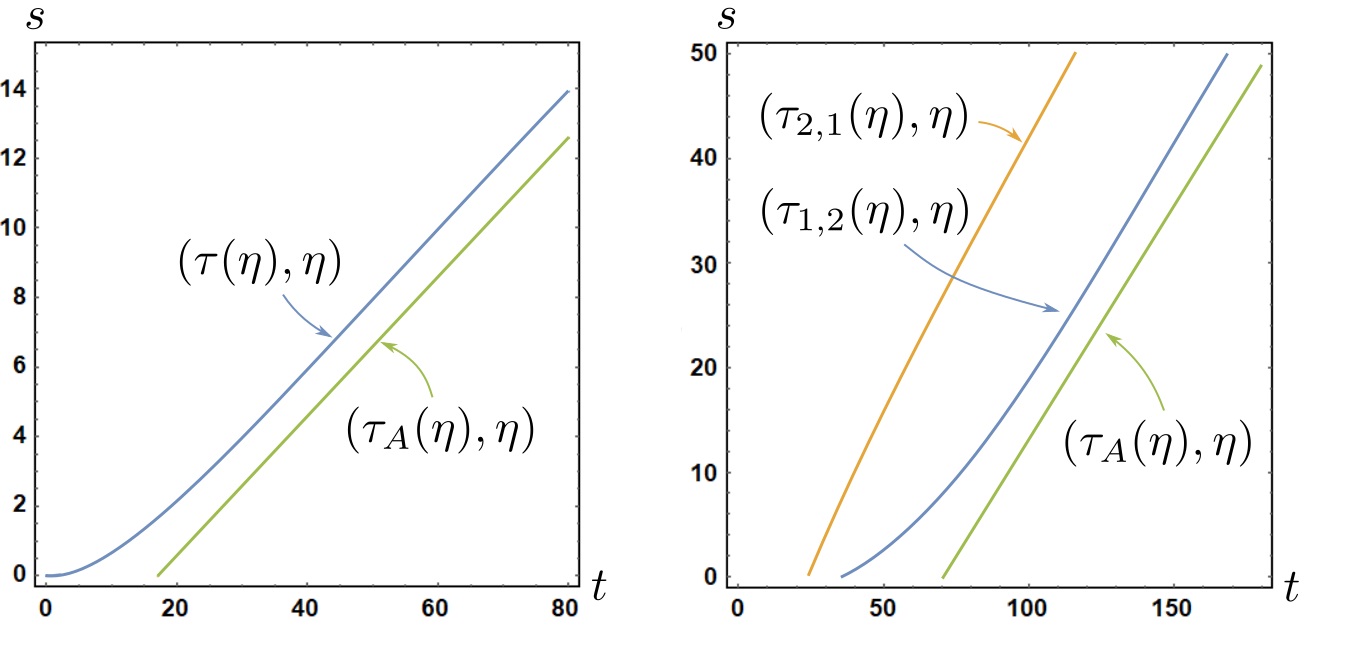}
		\caption{Comparison of dwell-flee relations for system in Example~\ref{example:rc} (left) and system in Example~\ref{example:nr} (right). }
		\label{fig:rcexample}
	\end{figure}
\end{example}

\begin{example}\label{example:nr}
	Consider a Hurwitz stable (defective) matrix $A_1=\begin{pmatrix}-0.1 & \sqrt{2}\\ 0 & -0.1\end{pmatrix}$ and an unstable (real-diagonalizable) matrix $A_2=\begin{pmatrix} 0.1 & 0\\-0.4 & 0.2\end{pmatrix}$. 
	
	For computing the dwell-flee relation given in~\cite{agarwal2018simple}, take the Jordan basis matrix $V_1=\begin{pmatrix}
		1& 1/\sqrt{2}\\ 0&1/\sqrt{2}
	\end{pmatrix}$ of $A_1$ having unit norm columns with Jordan form $J_1=\begin{pmatrix}-0.1&1\\0&-0.1 \end{pmatrix}$ and the Jordan basis matrix $V_2=\begin{pmatrix}1/\sqrt{17} &0\\4/\sqrt{17}&1\end{pmatrix}$ of $A_2$ with unit norm columns with Jordan form $J_2=\text{diag}\left(0.1,\,0.2\right)$. Then $\tau_A(\eta)=\frac{0.2\eta+\ln c_{\theta_1}+2.75}{\theta_1}$,
	for any $0<\theta_1<n$ where $\|{\rm{e}}^{J_1 t}\|\le c_{\theta_1}{\rm{e}}^{-\theta_1 t}$. The least value of $\tau_A(1)$ is $72.36$ which is attained at $\theta_1=0.086$, of $\tau_A(5)$ is $81.56$ which is attained at $\theta_1=0.088$, and of $\tau_A(10)$ is $92.86$ which is attained at $\theta_1=0.089$. 
	
	Now take a Jordan basis matrix $P_1$ of $A_1$ with Jordan form $J_1$ and a Jordan basis matrix $P_2$ of $A_2$ with Jordan form $J_2$ such that the determinant of $M=P_2^{-1}P_1$ is 1. Then $M=\begin{pmatrix}
		2\sqrt[4]{2}&0\\-2\sqrt[4]{2}&1/2\sqrt[4]{2}
	\end{pmatrix}$. Using Theorem~\ref{nr:12}, the dwell-flee relation for $\tau_{1,2}(\eta)$ is given by $0.1\, \tau_{1,2}(\eta)-\ln\theta(\tau_{1,2}(\eta))=0.15\,\eta+\sinh^{-1}\left(5.7\,\sinh 0.05\,\eta\right)$ and using Theorem~\ref{nr:21}(ii)(b), the dwell-flee relation for $\tau_{2,1}(\eta)$ is given by $4\sqrt{2}\cosh 0.05\,\eta=\sinh 0.05\,\eta+\sinh\left(0.1\, \tau_{2,1}(\eta)-0.15\,\eta\right)$. Hence $\tau_{1,2}(1)=41.60$, $\tau_{1,2}(5)=60.07$ and $\tau_{1,2}(10)=76.46$, and $\tau_{2,1}(1)=25.76$, $\tau_{2,1}(5)=31.71$ and $\tau_{2,1}(10)=39.68$. Since $\tau(\eta)=\min\{\tau_{1,2}(\eta),\tau_{2,1}(\eta)\}$, $\tau(1)=25.76$, $\tau(5)=31.71$, and $\tau(10)=39.68$. Refer to Figure~\ref{fig:rcexample} (right) where $\tau_A$ (for $\theta_1=0.089$), $\tau_{1,2}$ and $\tau_{2,1}$ are plotted together.
\end{example}

\begin{example}\label{example:nn}
	Consider a Hurwitz stable (defective) matrix $A_1=\begin{pmatrix}
		-0.1&1\\0&-0.1
	\end{pmatrix}$ with Jordan form $J_1=A_1$, and an unstable (defective) matrix $A_2=\begin{pmatrix}
		-2.8&9\\-1&3.2
	\end{pmatrix}$, with $J_2=\begin{pmatrix}
		0.2&1\\0&0.2
	\end{pmatrix}$. Both $A_1$ and $A_2$ are defective matrices. The matrix $A_1$ is already in normal form, hence $P_1$ is the identity matrix. Consider Jordan basis matrix $P_2=\begin{pmatrix}
		3&8\\1&3
	\end{pmatrix}$, then $M=P_2^{-1}P_1=\begin{pmatrix}
		3&-8\\-1&3
	\end{pmatrix}$ of $A_2$ with Jordan form $J_2$. We will now use the notations and results from Section~\ref{sec:NN} to discuss stability of the system~\eqref{eq:system} with subsystem matrices given here. Using Theorem~\ref{nn:12}, for fixed $E_1$ and $E_2$, the dwell-flee relation $\tau_{1,2}(P_1 E_1,P_2 E_2,\eta)$ is given by the unique positive solution $\tau$ of \[
	0.1\,\tau-\ln\theta(\tau)=0.2\,\eta+\sinh^{-1}\left(0.5\left(1+(3-\epsilon_1)^2\right)\eta\right).
	\]
	Observe that the relation does not depend on $\epsilon_2$. Refer to Figure~\ref{fig:dwell-flee_variation} for the graphs of $\tau_{1,2}(P_1 E_1,P_2 E_2,\eta)$ as functions of $\eta$, for $\epsilon_1=3,0,12$. With this procedure, the ``best" dwell-flee relation is obtained when $\epsilon_1=3$. Let us take $\eta=10$, then $\tau_{1,2}(10)=87.89$. 
	
	Using Theorem~\ref{nn:21}, for fixed $E_1$ and $E_2$, the dwell-flee relation $\tau_{2,1}(P_1 E_1,P_2 E_2,\eta)$ is given by the unique positive solution $\tau$ of \[
	0.2\,\eta+\ln\theta(\eta)=0.1\,\tau-\sinh^{-1}\left(0.5\left(1+(3+\epsilon_2)^2\right)\tau\right).
	\]
	Observe that the relation does not depend on $\epsilon_1$. With this procedure, the ``best" dwell-flee relation is obtained when $\epsilon_2=-3$, in which case, $\tau_{2,1}(10)=87.89$.
	
	Hence $\tau(10)=\min\{\tau_{1,2}(10),\tau_{2,1}(10)\}=87.89$. Consider a switching signal $\sigma\in \mathcal{S}[87.89,10]$ for which $\sigma(0)=1$ and time spent after each switch, $\Delta_k=d_k-d_{k-1}$, for $k=1,2,\cdots,12$, is randomly generated in Mathematica 11.0, such that $\Delta_k\in(5,10)$, for even $k$ (time spent in the unstable subsystem $A_2$) and $\Delta_k\in(87.89,92)$, for $k$ odd (time spent in the stable subsystem). The solution trajectory of the switched system with initial condition $(10,-5)^T$ and the randomly generated switched signal $\sigma$ with $\left(\Delta\right)_{k=1}^{12}$ given by 
	\[
	(90.71, 6.26, 90.3, 9.69, 88.21, 6.88, 89.63, 9.91, 88.56, 7.12, 90.05, 6.96),
	\]
	converges to the origin, see Figure~\ref{fig:dwell-flee_variation}.
	\begin{figure}[h!]
		\centering
		\includegraphics[scale=0.5]{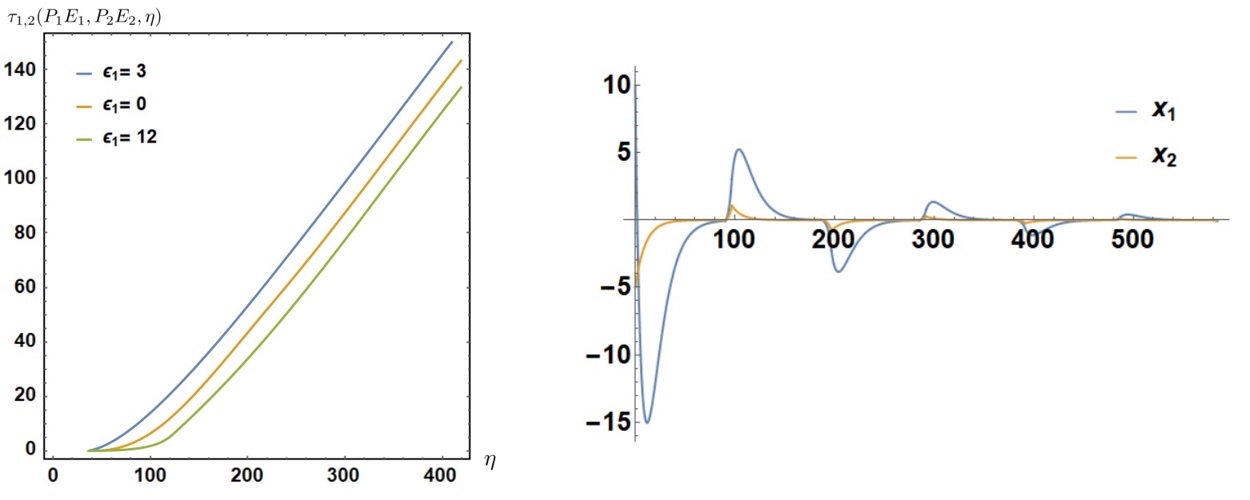}
		\caption{Dwell-flee relations for $\epsilon_1=0, 3, 12$ (left); the $x_1(t)$ and $x_2(t)$ curves of the solution trajectory $x(t)=(x_1(t),x_2(t))^T$ with initial condition $(10,-5)^T$ and switching signal $\sigma$ with $\left(\Delta_k\right)_{k=1}^{12}$ described in Example~\ref{example:nn} (right). }
		\label{fig:dwell-flee_variation}
	\end{figure}
\end{example}

\section*{Funding}

The first named author is supported by the Senior Research Fellowship provided by the Indian Institute of Science Education and Research Bhopal.

\appendix

\section{Details of proofs}\label{ap:proof}

\noindent The following notation will be used in this section $\alpha_i=p_i+q_i$ and $\beta_i=q_i-p_i$, for $i=1,2$.

\begin{proof}[Few details of proof of Theorem~\ref{rr:12zero}]
	The following three statements are true about the sequence of functions $\{T_k\}_{k\in\mathbb{N}}$:
	\begin{enumerate}
		\item Each $T_k$ is continuous on $[0,\eta]$.
		\item $\{T_k\}$ converges pointwise to the function $T(s)$ since $b^2d^2/k^4\to 0$ as $k\to\infty$. 
		\item $T_k(s)\ge T_{k+1}(s)$, for all $s\in[0,\eta]$ and each $k\in\mathbb{N}$.
	\end{enumerate}
	Hence the sequence of functions $\{T_k\}_{k\in\mathbb{N}}$ converges uniformly to the function $\max\left(q_2/q_1,\,p_2/p_1\right)s$ on the closed interval $[0,\eta]$, refer~\cite[Theorem 7.13]{rudin1976principles}.	
\end{proof}

\begin{proof}[Details of proof of Theorem~\ref{rr:12abcd>0}]
	The following facts are either used or referred to in the proof.
	\begin{itemize}
		\item $(\ell_+)_t>0$ in $Q_1$ follows immediately from the expression of $(\ell_+)_t$.
		\item $(\ell_+)_s<0$ if and only if \begin{eqnarray}\label{lpluss}
			q_2e^{p_2 s}-p_2e^{q_2 s}<\beta_2(e^{q_1t}-e^{p_1 t})+e^{\alpha_1t}(q_2 e^{-p_2 s}-p_2e^{-q_2 s}).
		\end{eqnarray}		
		Since $q_2 e^{-p_2 s}-p_2e^{-q_2 s}>0$, the term on the right of the inequality~\eqref{lpluss} is increasing in $t$, and hence is bounded below by the following function in $\Omega$: $\beta_2\left(e^{\frac{\alpha_2}{\alpha_1}q_1s}-e^{\frac{\alpha_2}{\alpha_1} p_1 s}\right)+(q_2 e^{q_2 s}-p_2e^{p_2 s})$. This function clearly is an upper bound for the term on the left of the inequality~\eqref{lpluss}. Hence, $(\ell_+)_s<0$ in $\Omega$.
		\item Similarly, the term on the right of the inequality~\eqref{lpluss} is bounded below by the following function in $\Omega_3$: $\beta_2(e^{q_2 s}-e^{\frac{p_1q_2 s}{q_1}})+e^{\frac{q_2p_1 s}{q_1}}(q_2e^{\beta_2s}-p_2)$. Moreover the term on the left of the inequality~\eqref{lpluss} is bounded above by $\beta_2e^{q_2 s}$. Since $\beta_2e^{q_2 s}\le \beta_2(e^{q_2 s}-e^{\frac{p_1q_2 s}{q_1}})+e^{\frac{q_2p_1 s}{q_1}}(q_2e^{\beta_2s}-p_2)$, we can conclude that $(\ell_+)_s<0$ in $\Omega_3$.
		\item $(\ell_-)_s(t,s)>0$ if and only if \[
		F_1(t,s)=e^{\alpha_1t}\left(q_2e^{-p_2 s}-p_2e^{-q_2 s}\right)+p_2e^{q_2 s}-q_2e^{p_2 s}>\beta_2(e^{q_1t}-e^{p_1 t})=F_2(t,s).
		\]
		We claim that $(F_1)_t(t,s)>(F_2)_t(t,s)$ for all $(t,s)\in\Omega_0$. Note that $(F_1)_t(t,s)> (F_2)_t(t,s)$ if and only if $\frac{\alpha_1}{\beta_2}\left(q_2e^{-p_2 s}-p_2e^{-q_2 s}\right)>(q_1e^{-p_1t}-p_1e^{-q_1 t})$. The term on the left decreases with $s$. For a point $(t,s)\in \Omega_0$, $0<s<(p_1/p_2)t$, thus the above inequality holds true in $\Omega_0$ if $\frac{\alpha_1}{\beta_2}\left(q_2-p_2e^{-\left(\frac{p_1q_2}{p_2}-p_1\right) t}\right)>q_1-p_1e^{-\beta_1 t}$, which is true since the function on the right of the inequality is bounded above by $q_1$, and the function of the left of the inequality is bounded below by $\alpha_1$. Since we have proved $(F_1)_t>(F_2)_t$ in $\Omega_0$, if we show that $F_1\left(\frac{p_2}{p_1}s,s\right)\ge F_2\left(\frac{p_2}{p_1}s,s\right)$, that is, $F_1\ge F_2$ on $\partial\Omega_0$, we will get $(\ell_-)_s>0$ in $\Omega_0$. The aforementioned inequality is true since $(\ell_-)_s\left(\frac{p_2}{p_1}s,s\right)>0$.
		\item $(\ell_-)_t<0$ if and only if $\beta_1\cosh\left(\frac{\beta_2 s}{2}\right)<p_1\cosh\left(q_1 t-\frac{\alpha_2 s}{2}\right)+q_1\cosh\left(p_1 t-\frac{\alpha_2 s}{2} s\right)$. For $(t,s)\in \Omega_1$, $0<q_2 s<p_1t$, hence the right side of the inequality is bounded below by $q_1\cosh\left(\beta_2s/2\right)$. Thus the above inequality holds true in $\Omega_1$.
	\end{itemize}
\end{proof}

\begin{proof}[Details of the proof of Theorem~\ref{nr:21}]
	The following facts were used in the proof.
	\begin{itemize}
		\item $\ell_+(t,s)$ is decreasing in $s$, $\ell_-(t,s)$ is increasing in $s$, and $\ell_+(t,s)$ is increasing in $t$ in the region $\Omega_0$, follow easily from their respective expressions. 
		\item For each $s_0>0$, $\ell_-(t,s_0)$ has a unique point of maxima $t_0$ such that $(t_0,s_0)\in \overline{\Omega}_0$: $\left(\ell_-\right)_t(t,s)=\frac{1}{t^2}\sech\left(\frac{\beta_2 s}{2}\right) R(t,s)$, where $R(t,s)=-n_1t\cosh\left(n_1t-\frac{\alpha_2 s}{2}\right)+\sinh\left(\frac{\beta_2 s}{2}\right)+\sinh\left(n_1t-\frac{\alpha_2 s}{2}\right)$. Since $R_t(t,s)<0$ in $\Omega_0$, for each $s_0>0$, $\ell_-(t,s_0)$ is either monotonically decreasing in $t$ throughout or it has a unique point of maxima. Moreover on the line $2n_1t=\alpha_2s$, we have $R\left(\frac{\alpha_2}{2n_1}\xi,\xi\right)=-\frac{\alpha_2\xi}{2}+\sinh\left(\frac{\beta_2\xi}{2}\right)$. Using the fact that for $\gamma>0$, the function $\xi^{-1}\sinh(\gamma\xi)$ is unbounded and increasing for $\xi\in [0,\infty)$, with minimum value $\gamma$ attained at $\xi=0$, there exists $\xi_0> 0$ such that $R\left(\frac{\alpha_2}{2n_1}\xi_0,\xi_0\right)=0$. Then, for $0<s_0\le\xi_0$, $\ell_-(t,s_0)$ is monotonically decreasing in $t$, in $\Omega_0$. Further for $s_0>\xi_0$, $\ell_-(t,s_0)$ first increases and then decreases in $t$, in $\Omega_0$, and hence it has a unique maxima.
	\end{itemize}
\end{proof}

	\bibliographystyle{siam}
	\bibliography{References}

\end{document}